\documentclass[aap]{imsart}

\RequirePackage{amsthm,amsmath,amsfonts,amssymb}
\RequirePackage[numbers]{natbib}
\RequirePackage[colorlinks,citecolor=blue,urlcolor=blue]{hyperref}
\RequirePackage{graphicx}
\RequirePackage{enumitem}
\RequirePackage{tikz}

\startlocaldefs
\makeatletter\def\journal@name{}\makeatother
\numberwithin{equation}{section}
\theoremstyle{plain}
\newtheorem{theorem}{Theorem}[section]
\newtheorem{proposition}[theorem]{Proposition}
\newtheorem{lemma}[theorem]{Lemma}
\newtheorem{corollary}[theorem]{Corollary}
\theoremstyle{definition}
\newtheorem{definition}[theorem]{Definition}
\newtheorem{example}[theorem]{Example}
\newtheorem{remark}[theorem]{Remark}

\newcommand{\R}{\mathbb{R}}

\newcommand{\E}{\mathbb{E}}
\renewcommand{\P}{\mathbb{P}}
\newcommand{\ip}[2]{\langle #1,#2\rangle}
\newcommand{\cZ}{\mathcal{Z}}

\newcommand{\cH}{\mathcal{H}}
\newcommand{\cL}{\mathcal{L}}

\newcommand{\cE}{\mathcal{E}}
\newcommand{\Wt}{\mathsf{W}_2}
\newcommand{\TV}{d_{\mathrm{TV}}}
\newcommand{\eps}{\varepsilon}
\newcommand{\one}{\mathbf{1}}
\DeclareMathOperator{\haus}{haus}

\DeclareMathOperator{\tr}{tr}
\DeclareMathOperator{\Law}{Law}
\DeclareMathOperator{\diam}{diam}
\DeclareMathOperator{\aff}{aff}
\DeclareMathOperator{\Leb}{Leb}

\DeclareMathOperator{\osc}{osc}
\endlocaldefs

\begin{document}

\begin{frontmatter}
\title{Zero temperature versus long time for diffusions driven by metric projections onto convex sets}
\runtitle{Diffusions driven by projections onto convex sets}

\begin{aug}
\author[A]{\fnms{Narayana}~\snm{Darapaneni}\ead[label=e1]{narayana.darapaneni@northwestern.edu}}
\address[A]{Northwestern University SPS\printead[presep={,\ }]{e1}}
\runauthor{N. Darapaneni}
\end{aug}

\begin{abstract}
For a closed subset $C$ of Euclidean space, we study the diffusion whose drift pulls the state towards its nearest point in $C$, with noise intensity $T$. We let $T$ tend to zero and time tend to infinity. Convexity of $C$ is equivalent to a pathwise bound of the distance from $C$ by the norm of the process of Ornstein and Uhlenbeck. For arbitrary closed sets the bound holds in law. For compact convex $C$, we describe the invariant law by the curvature measures and the intrinsic volumes of $C$. The limit points of the law of the state form the orbit of reflected Brownian motion in $C$ started from the projection of the initial law. So the two limits commute only when this projection is the uniform law on $C$. Unless $C$ is a point, the mixing time is of order $1/T$ up to a logarithm. For the control problem the two limits of the value function commute for every compact convex set and every bounded Lipschitz terminal cost. A companion paper treats finite disjoint unions of convex sets.
\end{abstract}



\end{frontmatter}

\tableofcontents

\section{Introduction}\label{sec:intro}

\subsection{The model}
Let $C$ be a nonempty closed subset of $\R^d$. We write
\[
d_C(x)=\min_{c\in C}|x-c| ,\qquad V_C(x)=\tfrac12\,d_C(x)^2 ,
\]
and we denote by $P_Cx$ a nearest point of $x$ in $C$. The nearest point is unique for almost every $x$, and for every $x$ when $C$ is convex. This paper is about the diffusion
\begin{equation}\label{eq:SDEintro}
dX_t=-(X_t-P_CX_t)\,dt+\sqrt{2T}\,dW_t ,
\end{equation}
where $W$ is a standard Brownian motion in $\R^d$ and $T>0$ is the intensity of the noise. We call $T$ the temperature and $X_t$ the state. The drift moves each state along the straight segment to its metric projection on $C$. We have $\nabla V_C(x)=x-P_Cx$ at every point $x$ where the nearest point is unique. So \eqref{eq:SDEintro} is a gradient diffusion with potential $V_C$. The law $\rho_{T,t}$ of $X_t$ solves the forward Kolmogorov equation
\begin{equation}\label{eq:FP}
\partial_t\rho=\nabla\cdot(\rho\,\nabla V_C)+T\,\Delta\rho .
\end{equation}
When $C$ is compact, the invariant law is the Gibbs state
\[
\gamma_T^C(dx)=\cZ_T(C)^{-1}e^{-d_C(x)^2/2T}\,dx ,
\]
where $\cZ_T(C)$ is the normalising constant.

The diffusion \eqref{eq:SDEintro} appears in several places. After the change of time $s=tT$ it is the penalised equation used to construct reflected diffusions in convex domains \cite{Menaldi1983,LionsSznitman1984}. The function $d_C^2/2T$ is the Moreau envelope, with parameter $T$, of the function which is $0$ on $C$ and $+\infty$ outside $C$. For this reason \eqref{eq:SDEintro}, in the time $s$, is the continuous time limit of the proximal Langevin algorithm of \cite{Brosse2017} for sampling from the uniform law on a convex set. The projected algorithm of \cite{BubeckEldanLehec2018} is a discretisation of the reflected Brownian motion, which appears below as a limit of \eqref{eq:SDEintro}. One can also read \eqref{eq:SDEintro} as a simple model for the internal state of a system perturbed by noise. The drift pulls this state towards the set $C$ of states that are feasible in a given context. Models with a latent state which evolves in time are common in machine learning, see for example \cite{HaSchmidhuber2018,Song2021}.

There are two parameters, the time $t$ and the temperature $T$. We ask how the law of $X_t$ behaves when $t\to\infty$ and $T\to0$ together. The simplest form of the question is whether the two limits commute:
\begin{equation}\label{eq:question}
\lim_{T\to0}\ \lim_{t\to\infty}\rho_{T,t}\;\overset{?}{=}\;\lim_{t\to\infty}\ \lim_{T\to0}\rho_{T,t} .
\end{equation}
When the potential is smooth and has one nondegenerate minimum, both sides are the Dirac mass at the minimum. The potential $V_C$ is different in two respects. First, $V_C$ vanishes on $C$, so its set of minima is the whole set $C$. Thus the wells are flat. At zero temperature the selection is decided by geometry and not by energy. Second, $V_C$ is of class $C^{1,1}$ when $C$ is convex, and in general not better. For a set which is not convex, the drift of \eqref{eq:SDEintro} is discontinuous across the set of points which have more than one nearest point. The classical tools for small noise are the Laplace method for Gibbs measures \cite{Hwang1980} and the theory of \cite{FreidlinWentzell2012}, which is written for Lipschitz coefficients and for domains with smooth boundary. When $C$ has nonempty interior, the zero temperature limit $u_C$ of the Gibbs state can also be obtained by the Laplace method. These tools do not describe the joint behaviour in $t$ and $T$ inside a flat well. For a set which is not convex they do not apply directly, because the drift is not continuous.

One more process appears throughout the paper. We denote by $\zeta$ the process of Ornstein and Uhlenbeck (OU process for short) which solves $d\zeta=-\zeta\,dt+\sqrt{2T}\,dW$, $\zeta_0=0$, and which is driven by the same Brownian motion as the state.

\subsection{Results}
We describe the results in four items, (A) to (D). Most of them are about compact convex sets.

\medskip\noindent\textbf{(A) Envelopes.} Let $C(t)$, $t\ge0$, be closed convex sets in a separable Hilbert space such that the Hausdorff distance of $C(s)$ and $C(t)$ is at most $v|t-s|$. Consider the state equation \eqref{eq:SDEintro} with $C(t)$ in place of $C$. Theorem \ref{thm:envelope} states that every path of the state satisfies
\begin{equation}\label{eq:introenv}
d_{C(t)}(X_t)\le e^{-t}d_{C(0)}(X_0)+v(1-e^{-t})+|\zeta_t|\qquad(t\ge0).
\end{equation}
The right side does not depend on the shape of the sets. In $\R^d$ convexity is also necessary. Indeed, if $C$ is closed and not convex, then \eqref{eq:introenv} fails with positive probability (Theorem \ref{thm:converse}). The closed convex subsets of $\R^d$ are exactly the Chebyshev sets \cite[Chapter 12]{Deutsch2001}, that is, the sets in which every point of the space has exactly one nearest point. So the pathwise envelope characterises the sets for which the metric projection is single valued.

For an arbitrary closed set a weaker statement remains true (Theorem \ref{thm:compare}). The distance $d_C(X_t)$ is dominated, path by path, by a process which has the law of the norm of an OU process started at distance $d_C(X_0)$ from the origin. Hence, in $\R^d$, a state which starts in $C$ stays in $\{d_C<r\}$ up to time $\exp(c/T)$ with probability tending to one, for every $c<r^2/2$ (Corollaries \ref{cor:viability} and \ref{cor:viabilityclosed}). The dominating process is driven by a Brownian motion which depends on $C$, so this comparison is weaker than \eqref{eq:introenv}.

\medskip\noindent\textbf{(B) The Gibbs state.} For a compact convex set $C$, Theorem \ref{thm:gibbs} gives the exact structure of $\gamma_T^C$. Let $x$ be a random point with law $\gamma_T^C$. Then the law of the pair $(P_Cx,d_C(x))$ is a mixture over $j=0,\dots,k$, where $k$ is the dimension of $C$. The term of index $j$ is the product of the normalised curvature measure of order $j$ of $C$ and the law of $\sqrt T$ times a chi variable with $d-j$ degrees of freedom. The weights are proportional to $(2\pi T)^{(d-j)/2}V_j(C)$, where $V_j$ are the intrinsic volumes. In particular $\cZ_T(C)=\sum_j(2\pi T)^{(d-j)/2}V_j(C)$. As $T\to0$, the Gibbs state $\gamma_T^C$ converges to the uniform probability $u_C$ on $C$.

\medskip\noindent\textbf{(C) Convex sets: the limit diagram and the mixing time.} Here $C$ is compact and convex. The initial law $\rho_0$ has a finite second moment. Theorem \ref{thm:diagram} shows that, when $T\to0$ and $t\to\infty$ with $tT\to s\in[0,\infty]$,
\begin{equation}\label{eq:introdiagram}
\rho_{T,t}\ \longrightarrow\ \mu_s:=\begin{cases}(P_C)_\#\rho_0,& s=0,\\[2pt] \Law(R_s),&0<s<\infty,\\[2pt] u_C,& s=\infty,\end{cases}
\end{equation}
where $R$ is the reflected Brownian motion in $C$ started from the law $(P_C)_\#\rho_0$. Thus the left side of \eqref{eq:question} is $u_C$ and the right side is $(P_C)_\#\rho_0$. Every other limit point lies on the curve $s\mapsto\mu_s$ which joins them. The two limits commute if and only if $(P_C)_\#\rho_0=u_C$ (Corollary \ref{cor:commute}).

Theorem \ref{thm:tvmix} gives convergence in total variation. There is a constant $\alpha_d\in(0,1)$, depending only on the dimension, with the following property. Uniformly over starting points at bounded distance from $C$, the total variation distance of the law of $X_t$ from $\gamma_T^C$ is at most $(1-\alpha_d)^n$, up to an exponentially small term, as soon as $t\ge nA\log(1/T)/T$. Here $A$ is a constant which does not depend on $T$, and $T$ is small. By the mixing time we mean the first time $t$ at which the law of $X_t$ is within total variation distance $\frac14$ of $\gamma_T^C$ for every starting point in $C$. Together with \eqref{eq:introdiagram}, the estimate above shows that, when $C$ is not a point, the mixing time is of order $1/T$ up to a logarithm (Corollary \ref{cor:mixtime}).

\medskip\noindent\textbf{(D) The control problem.} Suppose that we may add a drift $a$ to \eqref{eq:SDEintro} at the running cost $\frac12|a|^2$, and that we pay a terminal cost $g$ at the time $t$. Let $v_T(t,x)$ be the least expected cost of steering the state from $x$ over the horizon $t$. By the logarithmic transformation, $v_T(t,x)=-2T\log\E\exp(-g(X^x_t)/2T)$, where $X^x$ is the solution of \eqref{eq:SDEintro} started at $x$. In Section \ref{sec:hjb} we take this formula as the definition of $v_T$ for every compact set $C$. Theorem \ref{thm:hjb} shows that for compact convex $C$
\[
\lim_{T\to0}\lim_{t\to\infty}v_T(t,x)=\lim_{t\to\infty}\lim_{T\to0}v_T(t,x)=\min_{y\in\R^d}\{g(y)+d_C(y)^2\}.
\]
So the two limits commute for every bounded Lipschitz function $g$, although by (C) they do not commute in general for the law of the state. The equality $\lim_{T\to0}\lim_{t\to\infty}v_T(t,x)=\min_y\{g(y)+d_C(y)^2\}$ holds for every compact set $C$ (Theorem \ref{thm:hjb}(a)).

\medskip\noindent\textbf{The companion paper.} In \cite{DarapaneniMeta} we treat the case in which $C$ is the union of finitely many pairwise disjoint compact convex sets $C_1,\dots,C_m$. Let $\delta_i$ be the distance from $C_i$ to the nearest other component. Let $B_i$ be the set of points which are nearer to $C_i$ than to the other components. A state which starts in $B_i$ leaves $B_i$ at a time $\exp((\delta_i^2/8+o(1))/T)$, with probability tending to one. Before the smallest of these times the weight of each set $B_i$ is frozen. For two and three components the law of the state is close to equilibrium after a time $\exp(c/T)$ with $c>\max_i\delta_i^2/8$. For the control problem the two limits commute within this class exactly when $C$ is convex. The proofs in \cite{DarapaneniMeta} use the following tools of the present paper: the envelope (Theorem \ref{thm:envelope}), the comparison theorem (Theorem \ref{thm:compare}), the maximal inequality (Proposition \ref{prop:oumax}), the limit diagram (Theorem \ref{thm:diagram}) and the lower bound for the transition density (Lemma \ref{lem:density}).

\subsection{Methods and relation to earlier work}
\emph{Penalisation.} The middle line of \eqref{eq:introdiagram} is the classical convergence of the penalised equation to the reflected one \cite{Menaldi1983,LionsSznitman1984,Pettersson1997,Slominski2013}, read in another time scale. Reflected Brownian motion in convex sets is constructed in \cite{Tanaka1979}, see also \cite{Cepa1998}. The rate $T^{1/4}$ which we obtain agrees, up to a logarithmic factor, with the order of the strong rate of \cite{Slominski2013} for general convex sets. Our proof is different. It uses the envelope \eqref{eq:introenv} to control the distance to $C$. As far as we know, the description of all joint limits and the commutation criterion in (C) are new. The total variation estimate gives the order $1/T$ of the mixing time, up to a logarithm. Bounds in total variation for a discretisation of \eqref{eq:SDEintro} in the time $s$ are given in \cite{Brosse2017}. Our estimate uses a lower bound for the transition density at the scale $\sqrt T$ and the Poincar\'e inequality for logarithmically concave measures \cite{Bobkov1999}, see also \cite{KLS1995}.

\emph{Metastability.} We do not study here the passage of the state between several wells. For smooth potentials this passage is described in \cite{FreidlinWentzell2012,BEGK2004,BovierdenHollander2015}. For wells that are disjoint compact convex sets it is the subject of \cite{DarapaneniMeta}.

\emph{Convex geometry.} After a change of variable, the formula for $\cZ_T(C)$ is the integral representation of the Wills functional \cite{Hadwiger1975}. In \cite{Lotz2020} the same Gaussian weight is used to study the intrinsic volumes, and the law of $d_C(x)$ alone is closely related to the distance integrals studied there. Theorem \ref{thm:gibbs} gives in addition the joint law of $d_C(x)$ and $P_Cx$. The refinement in Theorem \ref{thm:gibbs} follows from the local Steiner formula \cite{Federer1959,Schneider2014}.

\emph{Moving sets and control.} Moving convex sets in Hilbert space are the subject of the sweeping process \cite{Moreau1977}. The equation \eqref{eq:SDEintro} without noise and with a moving set is a regularisation of this process. Whether a state stays near a prescribed set is a question of viability theory \cite{Aubin1991}. The link between \eqref{eq:FP} and the control problem through a logarithmic change of the unknown goes back to \cite{Fleming1978}. See \cite[Chapter VI]{FlemingSoner2006} for a complete account, and also \cite{BoueDupuis1998,EvansIshii1985}. The gradient flow structure of \eqref{eq:FP} is described in \cite{JKO1998,AGS2008}. Simulated annealing, in which the temperature decreases along a schedule, is treated in \cite{GemanHwang1986}.

\subsection{Organisation}
Section \ref{sec:prelim} contains the notation and the basic facts on projections, on the state equation and on the Gibbs state. We prove the envelopes in Section \ref{sec:envelope}. Section \ref{sec:gibbs} is about the Gibbs state, Section \ref{sec:diagram} about compact convex sets and Section \ref{sec:hjb} about the control problem. Section \ref{sec:open} lists some open questions.

\section{Preliminaries}\label{sec:prelim}

\subsection{Metric projections}
Let $H$ be a real separable Hilbert space with inner product $\ip{\cdot}{\cdot}$ and norm $|\cdot|$. For a nonempty closed set $C\subset H$ we put $d_C(x)=\inf\{|x-c|:c\in C\}$ and $V_C=\frac12d_C^2$. If $C$ is convex, then every $x\in H$ has a unique nearest point $P_Cx$ in $C$. We now recall some standard facts. They hold for every nonempty closed convex set $C$ and all $x,y\in H$. Proofs are given in \cite[Chapters 3, 4 and 12]{BauschkeCombettes2017}.
\begin{enumerate}[label=\textup{(P\arabic*)},leftmargin=3em]
\item\label{P1} $\ip{x-P_Cx}{c-P_Cx}\le0$ for every $c\in C$. The point $P_Cx$ is the only point of $C$ with this property.
\item\label{P2} $\ip{(x-P_Cx)-(y-P_Cy)}{x-y}\ge|(x-P_Cx)-(y-P_Cy)|^2$. Hence the map $x\mapsto x-P_Cx$ is monotone and Lipschitz with constant $1$. A map which is Lipschitz with constant $1$ is called nonexpansive. The map $P_C$ is also nonexpansive.
\item\label{P3} The function $d_C$ is convex and nonexpansive.
\item\label{P4} The function $V_C$ is Fr\'echet differentiable and $\nabla V_C(x)=x-P_Cx$.
\item\label{P5} For $\lambda\in[0,1]$ we have $P_C\big(P_Cx+\lambda(x-P_Cx)\big)=P_Cx$.
\end{enumerate}
The Hausdorff distance of two nonempty closed sets is
\[
\haus(C,C')=\max\Big\{\sup_{c\in C}d_{C'}(c),\ \sup_{c'\in C'}d_C(c')\Big\}\in[0,\infty].
\]
We shall use the inequality
\begin{equation}\label{eq:dhaus}
|d_C(x)-d_{C'}(x)|\le\haus(C,C')\qquad(x\in H).
\end{equation}
It follows from the triangle inequality. Indeed, let $c\in C$ and $c'\in C'$. Then
\[
d_{C'}(x)\le|x-c'|\le|x-c|+|c-c'| .
\]
We take the infimum over $c'\in C'$ and get
\[
d_{C'}(x)\le|x-c|+d_{C'}(c)\le|x-c|+\haus(C,C') .
\]
Now we take the infimum over $c\in C$ and get $d_{C'}(x)\le d_C(x)+\haus(C,C')$. In the same way $d_C(x)\le d_{C'}(x)+\haus(C,C')$. This proves \eqref{eq:dhaus}.

The next lemma compares the projections onto two convex sets. Estimates of this type go back to \cite{Moreau1977}. We include the proof, since we use the constant later.

\begin{lemma}\label{lem:twosets}
Let $C,C'\subset H$ be nonempty closed convex sets with $\haus(C,C')<\infty$. Then for every $x\in H$
\[
|P_Cx-P_{C'}x|^2\le\big(d_C(x)+d_{C'}(x)\big)\,\haus(C,C').
\]
\end{lemma}

\begin{proof}
Put $h=\haus(C,C')$. Write $p=P_Cx$ and $p'=P_{C'}x$, so that
\[
|x-p|=d_C(x),\qquad|x-p'|=d_{C'}(x).
\]
Put $c'=P_{C'}p\in C'$ and $c=P_Cp'\in C$. Since $p\in C$ and $p'\in C'$, the definition of $h$ gives
\[
|p-c'|=d_{C'}(p)\le\sup_{a\in C}d_{C'}(a)\le h,\qquad |p'-c|=d_C(p')\le\sup_{a'\in C'}d_C(a')\le h.
\]
We write $p-p'=(x-p')-(x-p)$ in the first factor of $\ip{p-p'}{p-p'}$. This gives
\begin{align*}
|p-p'|^2&=\ip{(x-p')-(x-p)}{p-p'}\\
&=\ip{x-p'}{p-p'}-\ip{x-p}{p-p'}\\
&=\ip{x-p'}{p-p'}+\ip{x-p}{p'-p}.
\end{align*}
We estimate the two terms separately. In the first term we write $p-p'=(c'-p')+(p-c')$. Since $c'\in C'$ and $p'=P_{C'}x$, we have $\ip{x-p'}{c'-p'}\le0$ by \ref{P1} for the set $C'$. For the other part we use the Schwarz inequality. Thus
\begin{align*}
\ip{x-p'}{p-p'}&=\ip{x-p'}{c'-p'}+\ip{x-p'}{p-c'}\\
&\le0+|x-p'|\,|p-c'|\\
&\le d_{C'}(x)\,h.
\end{align*}
In the second term we write $p'-p=(c-p)+(p'-c)$. Since $c\in C$ and $p=P_Cx$, we have $\ip{x-p}{c-p}\le0$ by \ref{P1} for the set $C$. Thus
\begin{align*}
\ip{x-p}{p'-p}&=\ip{x-p}{c-p}+\ip{x-p}{p'-c}\\
&\le0+|x-p|\,|p'-c|\\
&\le d_C(x)\,h.
\end{align*}
We add the two estimates and obtain
\[
|p-p'|^2\le d_{C'}(x)\,h+d_C(x)\,h=\big(d_C(x)+d_{C'}(x)\big)\,h .
\]
This completes the proof.
\end{proof}

\begin{definition}\label{def:path}
A \emph{moving convex set of speed $v\ge0$} is a family $\{C(t)\}_{t\ge0}$ of nonempty closed convex subsets of $H$ such that
\[
\haus(C(s),C(t))\le v\,|t-s|\qquad(s,t\ge0).
\]
\end{definition}

We give two examples. Let $q\mapsto C_q$ be a map from a metric space of parameters into the closed convex sets with the Hausdorff distance. Suppose that this map is Lipschitz with constant $\ell$ and that $t\mapsto q(t)$ is Lipschitz with constant $\ell'$. Then $C(t)=C_{q(t)}$ is a moving convex set of speed $\ell\ell'$. Indeed, $\haus(C_{q(s)},C_{q(t)})$ is at most $\ell$ times the distance of $q(s)$ and $q(t)$. This distance is at most $\ell'|t-s|$. A constant family $C(t)=C$ has speed $0$.

\subsection{The state equation}
Let $\{C(t)\}$ be a moving convex set of speed $v$, let $x_0\in H$ and let $w:[0,\infty)\to H$ be continuous with $w(0)=0$. We consider the integral equation
\begin{equation}\label{eq:state}
x(t)=x_0-\int_0^t\big(x(u)-P_{C(u)}x(u)\big)\,du+w(t)\qquad(t\ge0),
\end{equation}
and together with it the linear equation
\begin{equation}\label{eq:ou}
\zeta(t)=-\int_0^t\zeta(u)\,du+w(t)\qquad(t\ge0).
\end{equation}

\begin{lemma}\label{lem:wellposed}
Equation \eqref{eq:state} has a unique continuous solution $x$. Equation \eqref{eq:ou} has the unique continuous solution
\begin{equation}\label{eq:ouformula}
\zeta(t)=w(t)-\int_0^te^{-(t-u)}w(u)\,du .
\end{equation}
\end{lemma}

\begin{proof}
First we show that for $x,x'\in H$ and $u,u'\ge0$,
\begin{equation}\label{eq:jointcont}
|P_{C(u)}x-P_{C(u')}x'|\le|x-x'|+\big((d_{C(u)}(x')+d_{C(u')}(x'))\,v\,|u-u'|\big)^{1/2}.
\end{equation}
By the triangle inequality,
\[
|P_{C(u)}x-P_{C(u')}x'|\le|P_{C(u)}x-P_{C(u)}x'|+|P_{C(u)}x'-P_{C(u')}x'| .
\]
The first term on the right is at most $|x-x'|$, because $P_{C(u)}$ is nonexpansive by \ref{P2}. For the second term we apply Lemma \ref{lem:twosets} to the sets $C(u)$ and $C(u')$ and the point $x'$. Then we use Definition \ref{def:path}:
\begin{align*}
|P_{C(u)}x'-P_{C(u')}x'|^2&\le\big(d_{C(u)}(x')+d_{C(u')}(x')\big)\,\haus(C(u),C(u'))\\
&\le\big(d_{C(u)}(x')+d_{C(u')}(x')\big)\,v\,|u-u'| .
\end{align*}
This proves \eqref{eq:jointcont}. Now fix $(u',x')$. By \eqref{eq:dhaus} we have $d_{C(u)}(x')\le d_{C(u')}(x')+v\,|u-u'|$, so the right side of \eqref{eq:jointcont} is at most
\[
|x-x'|+\big((2d_{C(u')}(x')+v\,|u-u'|)\,v\,|u-u'|\big)^{1/2},
\]
which tends to zero as $(u,x)\to(u',x')$. Hence the map $(u,x)\mapsto P_{C(u)}x$ is continuous on $[0,\infty)\times H$.

Put
\[
F(u,z)=-\big(z+w(u)-P_{C(u)}(z+w(u))\big)\qquad(u\ge0,\ z\in H).
\]
Then $F$ is continuous on $[0,\infty)\times H$, because $w$ is continuous and $(u,x)\mapsto P_{C(u)}x$ is continuous. Since $y\mapsto y-P_{C(u)}y$ is nonexpansive by \ref{P2}, we have
\[
|F(u,z)-F(u,z')|\le|(z+w(u))-(z'+w(u))|=|z-z'| .
\]
By the existence and uniqueness theorem for ordinary differential equations with a Lipschitz right side, the equation $z'=F(u,z)$, $z(0)=x_0$, has a unique solution on $[0,\infty)$. In integral form,
\[
z(t)=x_0+\int_0^tF(u,z(u))\,du=x_0-\int_0^t\big(z(u)+w(u)-P_{C(u)}(z(u)+w(u))\big)\,du .
\]
Put $x=z+w$. Then the last identity reads $x(t)-w(t)=x_0-\int_0^t(x(u)-P_{C(u)}x(u))\,du$, which is \eqref{eq:state}. Conversely, let $x$ be a continuous solution of \eqref{eq:state} and put $z=x-w$. Then $z(t)=x_0+\int_0^tF(u,z(u))\,du$ and the integrand is continuous in $u$. So $z$ is continuously differentiable and solves $z'=F(u,z)$ with $z(0)=x_0$. Hence it is the unique solution of this problem. Thus $x=z+w$ is the unique continuous solution of \eqref{eq:state}.

We now turn to \eqref{eq:ou}. Let
\[
y(t)=\int_0^te^{-(t-u)}w(u)\,du=e^{-t}\int_0^te^{u}w(u)\,du .
\]
We differentiate the product and get
\[
y'(t)=-e^{-t}\int_0^te^{u}w(u)\,du+e^{-t}e^{t}w(t)=-y(t)+w(t),\qquad y(0)=0 .
\]
Let $\zeta=w-y$, which is the function in \eqref{eq:ouformula}. Then $y'=\zeta$, and so
\[
\int_0^t\zeta(u)\,du=\int_0^ty'(u)\,du=y(t)-y(0)=y(t)=w(t)-\zeta(t) .
\]
That is, $\zeta(t)=w(t)-\int_0^t\zeta(u)\,du$. So $\zeta$ solves \eqref{eq:ou}. To prove uniqueness, let $\zeta_1$ and $\zeta_2$ be two continuous solutions and put $\eta=\zeta_1-\zeta_2$. Then $\eta(t)=-\int_0^t\eta(u)\,du$. So $\eta$ is continuously differentiable, $\eta'=-\eta$ and $\eta(0)=0$. Hence
\[
\frac{d}{dt}\big(e^{t}\eta(t)\big)=e^{t}\big(\eta'(t)+\eta(t)\big)=0 ,
\]
so that $e^{t}\eta(t)=\eta(0)=0$ for all $t$. Thus $\zeta_1=\zeta_2$. This completes the proof.
\end{proof}

In the stochastic setting $W$ is a Wiener process in $H$ with covariance operator $Q$, where $Q$ is symmetric, nonnegative and of trace class. For these notions see \cite{DPZ2014}. We write $\tr Q$ for the trace of $Q$ and $\|Q\|$ for its operator norm. For $T>0$ and an initial state $X_0$ independent of $W$, the \emph{state process} $X=X^T$ and the \emph{envelope process} $\zeta=\zeta^T$ are the solutions of \eqref{eq:state} and \eqref{eq:ou} with $w=\sqrt{2T}\,W$. In differential form,
\begin{equation}\label{eq:SDE}
dX_t=-\big(X_t-P_{C(t)}X_t\big)\,dt+\sqrt{2T}\,dW_t,\qquad d\zeta_t=-\zeta_t\,dt+\sqrt{2T}\,dW_t,\quad\zeta_0=0 .
\end{equation}
Both processes are adapted to the filtration generated by $X_0$ and $W$. For $\zeta$ this is clear from \eqref{eq:ouformula}. For $X$ we use the Picard iterates $x^0=x_0+w$ and
\[
x^{k+1}(t)=x_0-\int_0^t\big(x^k(u)-P_{C(u)}x^k(u)\big)\,du+w(t)\qquad(k\ge0).
\]
The value $x^k(t)$ depends only on $x_0$ and on the values of $w$ on $[0,t]$. By \ref{P2}, $|x^{k+1}(t)-x^k(t)|\le\int_0^t|x^k(u)-x^{k-1}(u)|\,du$, for $k\ge1$. Since the bound below is trivial for $k=0$, induction on $k$ gives
\[
|x^{k+1}(t)-x^k(t)|\le\frac{t^k}{k!}\,\sup_{u\le S}|x^1(u)-x^0(u)|\qquad(0\le t\le S).
\]
So the iterates converge uniformly on every bounded interval. The limit is continuous and solves \eqref{eq:state}. Hence it is the solution $x$ of Lemma \ref{lem:wellposed}. Thus $X_t$ is a limit of variables which are measurable with respect to $X_0$ and $(W_u)_{u\le t}$.

Next we compute the law of $\zeta_t$. For the Wiener integral of the smooth function $u\mapsto e^{u}$, integration by parts gives
\[
\int_0^te^{u}\,dW_u=e^{t}W_t-\int_0^te^{u}W_u\,du .
\]
We multiply by $\sqrt{2T}\,e^{-t}$ and compare with \eqref{eq:ouformula} for $w=\sqrt{2T}\,W$:
\[
\sqrt{2T}\int_0^te^{-(t-u)}\,dW_u=\sqrt{2T}\,W_t-\int_0^te^{-(t-u)}\sqrt{2T}\,W_u\,du=\zeta_t .
\]
So $\zeta_t$ is a Wiener integral of a deterministic function. Hence $\zeta_t$ is a centred Gaussian variable. For $h,k\in H$ we have
\[
\E\,\ip{\zeta_t}{h}\ip{\zeta_t}{k}=2T\int_0^te^{-2(t-u)}\,du\ \ip{Qh}{k} .
\]
The substitution $r=t-u$ gives $\int_0^te^{-2(t-u)}\,du=\int_0^te^{-2r}\,dr=\frac12(1-e^{-2t})$. Thus the covariance operator of $\zeta_t$ is
\begin{equation}\label{eq:oucov}
2T\int_0^te^{-2(t-u)}du\,Q=T(1-e^{-2t})\,Q .
\end{equation}
Taking the trace we also get $\E|\zeta_t|^2=T(1-e^{-2t})\tr Q$. When $H=\R^d$ we always take $Q$ to be the identity. Then $W$ is a standard Brownian motion and $\zeta$ is the OU process of the Introduction.

\subsection{Sets which are not convex}\label{subsec:nonconvex}
Let $C\subset\R^d$ be nonempty and closed. The set of nearest points $\Pi_C(x)=\{c\in C:|x-c|=d_C(x)\}$ is nonempty and compact for every $x$. The function $V_C$ is locally Lipschitz and differentiable almost everywhere. At each point $x$ of differentiability the set $\Pi_C(x)$ has exactly one point $p$. At such a point $\nabla V_C(x)=x-p$. We fix once and for all a Borel map $P_C:\R^d\to C$ with $P_Cx\in\Pi_C(x)$ for all $x$. Such a map exists by the measurable selection theorem \cite[Chapter 14]{RockafellarWets1998}. When $C$ is convex this is the metric projection. The state equation in $\R^d$ is the equation \eqref{eq:SDEintro} of the Introduction, which we repeat for the closed set $C$:
\begin{equation}\label{eq:SDEgen}
dX_t=-\big(X_t-P_CX_t\big)\,dt+\sqrt{2T}\,dW_t .
\end{equation}
Its drift $b(x)=-(x-P_Cx)$ is Borel and has linear growth. To check the growth, note that $|b(x)|=|x-P_Cx|=d_C(x)$. For every $c\in C$ we have $d_C(x)\le|x-c|\le|x|+|c|$. Taking the infimum over $c\in C$ we get
\[
|b(x)|=d_C(x)\le|x|+d_C(0) .
\]
We recall three facts.
\begin{enumerate}[label=\textup{(S\arabic*)},leftmargin=3em]
\item\label{S1} For every initial condition independent of $W$, equation \eqref{eq:SDEgen} has a unique strong solution. The drift is bounded on every ball, so strong existence and pathwise uniqueness hold up to the exit time from each ball \cite{Veretennikov1981,KrylovRockner2005}. The linear growth of the drift together with the Gronwall lemma excludes explosion. To see this, fix $S>0$. As long as the solution is defined and $t\le S$, we have
\[
|X_t|\le|X_0|+\int_0^t\big(|X_u|+d_C(0)\big)\,du+\sqrt{2T}\,|W_t|\le a_S+\int_0^t|X_u|\,du ,
\]
where $a_S=|X_0|+d_C(0)\,S+\sqrt{2T}\,\sup_{u\le S}|W_u|$. The Gronwall lemma gives $|X_t|\le a_Se^{t}\le a_Se^{S}$, which is finite almost surely. It follows that the martingale problem for \eqref{eq:SDEgen} is well posed. The solutions started from the points of $\R^d$ then form a strong Markov family with measurable transition kernels \cite[Chapter 6]{StroockVaradhan1979}. Also, solutions driven by the same Brownian motion restart consistently at stopping times. This means that after a finite stopping time $\tau$ the solution is the solution which starts at $X_\tau$ and is driven by the Brownian motion $W_{\tau+t}-W_\tau$, $t\ge0$.
\item\label{S2} Let $t_0>0$ and $x\in\R^d$. Then the law on $C([0,t_0],\R^d)$ of the solution started at $x$ is equivalent to the law of $x+\sqrt{2T}\,W$. This follows from the Girsanov theorem, which gives a weak solution whose law has this property. Here the exponential density is a martingale because the drift has linear growth \cite[Section 3.5]{KaratzasShreve1991}. By \ref{S1} the solution is unique in law. It follows that for $t>0$ the law of $X_t$ is equivalent to the Lebesgue measure. Let $N$ be the set of points with more than one nearest point. It is contained in the set where $V_C$ is not differentiable, which is Lebesgue null. The set $N$ is Borel, because the diameter of the compact set $\Pi_C(x)$ is an upper semicontinuous function of $x$. So $\P(X_t\in N)=0$ for every $t>0$. By the Fubini theorem,
\[
\E\int_0^\infty\one_N(X_t)\,dt=\int_0^\infty\P(X_t\in N)\,dt=0 .
\]
Thus the time which $X$ spends in $N$ is zero almost surely. Two selections $P_C$ differ only on $N$, so they give the same drift integral $\int_0^t(X_u-P_CX_u)\,du$ along almost every path. Hence the solution does not depend on the choice of the selection $P_C$.
\item\label{S3} By the It\^o formula, $\rho_{T,t}=\Law(X_t)$ is a solution of the forward Kolmogorov equation \eqref{eq:FP} in the sense of distributions, with $\nabla V_C(x)=x-P_Cx$. Indeed, let $f$ be a smooth function with compact support. The It\^o formula for \eqref{eq:SDEgen} reads
\begin{align*}
f(X_t)=f(X_0)&+\int_0^t\Big(T\Delta f(X_u)-\ip{X_u-P_CX_u}{\nabla f(X_u)}\Big)\,du\\
&+\sqrt{2T}\int_0^t\ip{\nabla f(X_u)}{dW_u} .
\end{align*}
The stochastic integral has mean zero because $\nabla f$ is bounded. Taking expectations we get
\[
\int f\,d\rho_{T,t}=\int f\,d\rho_0+\int_0^t\int\Big(T\Delta f-\ip{\nabla V_C}{\nabla f}\Big)\,d\rho_{T,u}\,du ,
\]
where $\rho_0$ is the law of $X_0$. This is \eqref{eq:FP} tested against $f$.
\end{enumerate}
We denote by $P_t(x,\cdot)$ the transition probabilities of \eqref{eq:SDEgen}.

\subsection{The Gibbs state}
Let $C\subset\R^d$ be compact and $T>0$. Put $\|C\|=\max_{c\in C}|c|$. For every $c\in C$ we have $|x|-\|C\|\le|x-c|\le|x|+\|C\|$. This gives
\[
|x|-\|C\|\le d_C(x)\le|x|+\|C\|\qquad(x\in\R^d).
\]
If $|x|\ge2\|C\|$, then $d_C(x)\ge\frac12|x|$. If $|x|<2\|C\|$, then $d_C(x)\le3\|C\|$. Hence, for every $a>0$ and every $k\ge0$,
\begin{equation}\label{eq:xs1a}
\begin{aligned}
\int_{\R^d}\big(1+d_C(x)\big)^ke^{-a\,d_C(x)^2}\,dx&\le\int_{|x|<2\|C\|}(1+3\|C\|)^k\,dx\\
&\quad+\int_{\R^d}\big(1+\|C\|+|x|\big)^ke^{-a|x|^2/4}\,dx<\infty .
\end{aligned}
\end{equation}
With $k=0$ and $a=1/2T$ this shows that the number
\[
\cZ_T(C)=\int_{\R^d}e^{-d_C(x)^2/2T}\,dx
\]
is finite. We define the \emph{Gibbs state} of $C$ at temperature $T$ by
\[
\gamma_T^C(dx)=\cZ_T(C)^{-1}e^{-d_C(x)^2/2T}\,dx=\cZ_T(C)^{-1}e^{-V_C(x)/T}\,dx .
\]
A formal computation shows that the density of $\gamma_T^C$ is a stationary solution of \eqref{eq:FP}. Indeed, for $\rho=e^{-V_C/T}$ we have $T\nabla\rho=-\rho\,\nabla V_C$, so that
\[
\nabla\cdot(\rho\,\nabla V_C)+T\Delta\rho=\nabla\cdot\big(\rho\,\nabla V_C+T\nabla\rho\big)=0 .
\]
This computation is only formal, since $\nabla V_C$ may be discontinuous. So we give a proof that $\gamma_T^C$ is invariant for \eqref{eq:SDEgen} and that it attracts every initial law.

\begin{lemma}\label{lem:invariant}
Let $C\subset\R^d$ be compact and $T>0$. Then $\gamma_T^C$ is an invariant probability measure for \eqref{eq:SDEgen}. For every initial law $\rho_0$ and every Borel set $A\subset\R^d$,
\[
\rho_{T,t}(A)\longrightarrow\gamma_T^C(A)\qquad(t\to\infty).
\]
\end{lemma}

\begin{proof}
We write $V=V_C$ and $\gamma=\gamma_T^C$. As above, $b(x)=-(x-P_Cx)$ is the drift of \eqref{eq:SDEgen}. Here $b=-\nabla V$ almost everywhere and $|b|=d_C$. We use the total variation distance $\TV(\mu,\nu)=\sup_A|\mu(A)-\nu(A)|$ of two probability measures $\mu$ and $\nu$, where $A$ runs over the Borel sets. For measures with densities $g$ and $g'$ it is equal to $\frac12\int|g-g'|\,dx$.

\emph{Step 1: smoothing.} Let $\theta$ be a smooth probability density supported in the unit ball, let $\theta_n(x)=n^d\theta(nx)$ and put
\[
V_n(x)=(V*\theta_n)(x)=\int_{\R^d}V(x-z)\,\theta_n(z)\,dz .
\]
The density $\theta_n$ is supported in the ball of radius $1/n$ and its integral is one. The function $V_n$ is smooth. Since $V$ is locally Lipschitz, we have $\nabla V_n=(\nabla V)*\theta_n$. Let $|z|\le1/n$. As $d_C$ is nonexpansive, we have $|d_C(x-z)-d_C(x)|\le\frac1n$ and $d_C(x-z)\le d_C(x)+\frac1n$. Hence
\begin{align*}
|V(x-z)-V(x)|&=\tfrac12\,\big|d_C(x-z)^2-d_C(x)^2\big|\\
&=\tfrac12\,|d_C(x-z)-d_C(x)|\,\big(d_C(x-z)+d_C(x)\big)\\
&\le\tfrac1{2n}\Big(2d_C(x)+\tfrac1n\Big)=\frac{d_C(x)}{n}+\frac1{2n^2}\le\frac{d_C(x)+1}{n}.
\end{align*}
We integrate against $\theta_n(z)\,dz$ and get
\begin{align*}
|V_n(x)-V(x)|&=\Big|\int\big(V(x-z)-V(x)\big)\theta_n(z)\,dz\Big|\\
&\le\int|V(x-z)-V(x)|\,\theta_n(z)\,dz\le\frac{d_C(x)+1}{n} .
\end{align*}
For the gradient we use $|\nabla V|=d_C$ almost everywhere:
\[
|\nabla V_n(x)|=\Big|\int\nabla V(x-z)\,\theta_n(z)\,dz\Big|\le\int d_C(x-z)\,\theta_n(z)\,dz\le d_C(x)+\frac1n .
\]
We have proved that
\begin{equation}\label{eq:Vn}
|V_n(x)-V(x)|\le\frac{d_C(x)+1}{n},\qquad|\nabla V_n(x)|\le d_C(x)+\frac1n .
\end{equation}
Next we show that $\nabla V_n\to\nabla V$ almost everywhere. Since $\theta_n\le n^d\sup\theta$, we have
\begin{align*}
|\nabla V_n(x)-\nabla V(x)|&\le\int|\nabla V(x-z)-\nabla V(x)|\,\theta_n(z)\,dz\\
&\le\sup\theta\cdot n^d\int_{|z|\le1/n}|\nabla V(x-z)-\nabla V(x)|\,dz .
\end{align*}
The function $\nabla V$ is locally bounded. By the Lebesgue differentiation theorem the right side tends to zero for almost every $x$.

Now we bound the densities. By \eqref{eq:Vn} and $n\ge1$,
\[
V_n\ge V-\frac{d_C+1}{n}\ge\tfrac12d_C^2-d_C-1 .
\]
For every real $r$ we have $\frac12r^2-r-1-\big(\frac14r^2-2\big)=\frac14r^2-r+1=\frac14(r-2)^2\ge0$. With $r=d_C(x)$ this gives $V_n\ge\frac14d_C^2-2$. Hence
\begin{equation}\label{eq:xs1b}
e^{-V_n/T}\le e^{2/T}e^{-d_C^2/4T}\qquad(n\ge1).
\end{equation}
The right side is integrable by \eqref{eq:xs1a}. Put $\cZ_n=\int e^{-V_n/T}\,dx$ and let
\[
\gamma_n(dx)=g_n(x)\,dx,\qquad g_n=\cZ_n^{-1}e^{-V_n/T},\qquad g=\cZ_T(C)^{-1}e^{-V/T} .
\]
So $g$ is the density of $\gamma$. By \eqref{eq:Vn}, $V_n\to V$ at every point. So $e^{-V_n/T}\to e^{-V/T}$ at every point as well. From \eqref{eq:xs1b} and dominated convergence we get
\[
\int\big|e^{-V_n/T}-e^{-V/T}\big|\,dx\longrightarrow0,\qquad\cZ_n\longrightarrow\cZ_T(C) .
\]
Now
\[
|g_n-g|\le\frac1{\cZ_n}\big|e^{-V_n/T}-e^{-V/T}\big|+e^{-V/T}\Big|\frac1{\cZ_n}-\frac1{\cZ_T(C)}\Big| ,
\]
while $\int e^{-V/T}\,dx=\cZ_T(C)$. Hence
\[
2\,\TV(\gamma_n,\gamma)=\int|g_n-g|\,dx\le\frac1{\cZ_n}\int\big|e^{-V_n/T}-e^{-V/T}\big|\,dx+\frac{|\cZ_T(C)-\cZ_n|}{\cZ_n}\longrightarrow0 .
\]
Thus the probability measures $\gamma_n$ converge to $\gamma$ in total variation. Since $\cZ_n>0$ for all $n$ and $\cZ_n\to\cZ_T(C)>0$, we have $\inf_n\cZ_n>0$. By \eqref{eq:xs1b} there is a constant $K$ such that
\begin{equation}\label{eq:xs1c}
g_n\le K\,e^{-d_C^2/4T}\qquad(n\ge1).
\end{equation}

\emph{Step 2: the smooth equation.} For $c\in C$ we have $\frac12|x-c|^2=\frac12|x|^2-\ip xc+\frac12|c|^2$. Taking the minimum over $c\in C$ we get
\[
V(x)=\tfrac12|x|^2-\psi(x),\qquad\psi(x)=\max_{c\in C}\big\{\ip xc-\tfrac12|c|^2\big\} .
\]
The function $\psi$ is convex, because it is a maximum of affine functions. It is Lipschitz with constant $\|C\|=\max_{c\in C}|c|$. Indeed, for every $c\in C$,
\[
\ip xc-\tfrac12|c|^2\le\ip yc-\tfrac12|c|^2+\|C\|\,|x-y|\le\psi(y)+\|C\|\,|x-y| ,
\]
and we take the maximum over $c$. So $|\nabla\psi|\le\|C\|$ almost everywhere. We have $V_n=\frac12|\cdot|^2*\theta_n-\psi*\theta_n$. The first term is
\[
\tfrac12\int|x-z|^2\theta_n(z)\,dz=\tfrac12|x|^2-\Big\langle x,\int z\,\theta_n(z)\,dz\Big\rangle+\tfrac12\int|z|^2\theta_n(z)\,dz ,
\]
and its matrix of second derivatives is the identity. For the second term we use $\partial_i\partial_j(\psi*\theta_n)=(\partial_i\psi)*(\partial_j\theta_n)$ and get
\[
|\partial_i\partial_j(\psi*\theta_n)|\le\|C\|\int|\partial_j\theta_n(z)|\,dz=\|C\|\,n\int|\partial_j\theta(z)|\,dz .
\]
Hence the second derivatives of $V_n$ are bounded and $\nabla V_n$ is globally Lipschitz. So the equation $dX^n=-\nabla V_n(X^n)dt+\sqrt{2T}dW$ has a unique global solution. The measure $\gamma_n$ is invariant for it. This is the classical case of a smooth potential \cite[Chapter 4]{Pavliotis2014}, \cite{BGL2014}.

Fix $t>0$. Let $\mathbf P^n$ be the law on $C([0,t],\R^d)$ of $X^n$ started from $\gamma_n$. Similarly, $\widetilde{\mathbf P}^n$ denotes the law of the solution of \eqref{eq:SDEgen} started from $\gamma_n$. Both are laws of diffusions with the same diffusion coefficient and with drifts $-\nabla V_n$ and $b$ of linear growth. We compare them by the Girsanov theorem. Let $X$ denote the coordinate process on $C([0,t],\R^d)$. Under $\mathbf P^n$ the process
\[
B_u=\frac1{\sqrt{2T}}\Big(X_u-X_0+\int_0^u\nabla V_n(X_r)\,dr\Big)\qquad(0\le u\le t)
\]
is a standard Brownian motion. Let $\beta_n=(2T)^{-1/2}(\nabla V_n+b)$, which is the difference of the two drifts divided by $\sqrt{2T}$, and put
\[
D_n=\exp\Big(\int_0^t\ip{\beta_n(X_u)}{dB_u}-\frac12\int_0^t|\beta_n(X_u)|^2\,du\Big) .
\]
By \eqref{eq:Vn} we have $|\beta_n|\le(2T)^{-1/2}(2d_C+1)$, so $\beta_n$ has linear growth. As in \ref{S2}, this gives $\E_{\mathbf P^n}D_n=1$. The Girsanov theorem says that under the probability measure $D_n\,d\mathbf P^n$ the process $\widetilde B_u=B_u-\int_0^u\beta_n(X_r)\,dr$ is a standard Brownian motion. Under this measure the law of $X_0$ is still $\gamma_n$. Since $\sqrt{2T}\,\beta_n=\nabla V_n+b$, we get
\begin{align*}
X_u&=X_0-\int_0^u\nabla V_n(X_r)\,dr+\sqrt{2T}\,B_u\\
&=X_0-\int_0^u\nabla V_n(X_r)\,dr+\int_0^u\big(\nabla V_n(X_r)+b(X_r)\big)\,dr+\sqrt{2T}\,\widetilde B_u\\
&=X_0+\int_0^ub(X_r)\,dr+\sqrt{2T}\,\widetilde B_u .
\end{align*}
So under $D_n\,d\mathbf P^n$ the process $X$ is a solution of \eqref{eq:SDEgen} with initial law $\gamma_n$. The solution of \eqref{eq:SDEgen} is unique in law by \ref{S1}. Hence
\[
\frac{d\widetilde{\mathbf P}^n}{d\mathbf P^n}=D_n,\qquad\log\frac{d\mathbf P^n}{d\widetilde{\mathbf P}^n}=-\int_0^t\ip{\beta_n(X_u)}{dB_u}+\frac12\int_0^t|\beta_n(X_u)|^2\,du .
\]
The relative entropy of $\mathbf P^n$ with respect to $\widetilde{\mathbf P}^n$ is the mean of the last expression under $\mathbf P^n$. The measure $\gamma_n$ is invariant under $\mathbf P^n$, which means that the law of $X_u$ under $\mathbf P^n$ is $\gamma_n$ for every $u$. Also $|\beta_n|^2=\frac1{2T}|\nabla V_n-\nabla V|^2$ almost everywhere, while $\gamma_n$ has a density. By the Fubini theorem,
\[
\E_{\mathbf P^n}\int_0^t|\beta_n(X_u)|^2\,du=\int_0^t\int_{\R^d}|\beta_n|^2\,d\gamma_n\,du=\frac{t}{2T}\int_{\R^d}|\nabla V_n-\nabla V|^2\,d\gamma_n .
\]
From \eqref{eq:Vn} and \eqref{eq:xs1c} we get
\begin{equation}\label{eq:xs1d}
|\nabla V_n-\nabla V|^2\,g_n\le\big(|\nabla V_n|+|\nabla V|\big)^2g_n\le K\,(2d_C+1)^2e^{-d_C^2/4T}
\end{equation}
almost everywhere. The right side is integrable by \eqref{eq:xs1a}, so the last integral is finite. Since $B$ is a Brownian motion under $\mathbf P^n$, it follows that the stochastic integral $\int_0^t\ip{\beta_n(X_u)}{dB_u}$ has mean zero under $\mathbf P^n$. Thus the relative entropy of $\mathbf P^n$ with respect to $\widetilde{\mathbf P}^n$ is
\[
\frac1{4T}\,\E_{\mathbf P^n}\int_0^t|\nabla V_n-\nabla V|^2(X_u)\,du=\frac{t}{4T}\int_{\R^d}|\nabla V_n-\nabla V|^2\,d\gamma_n .
\]
We let $n\to\infty$. The function $|\nabla V_n-\nabla V|^2g_n$ tends to zero almost everywhere by Step 1. It is bounded by the integrable function in \eqref{eq:xs1d}, which does not depend on $n$. By dominated convergence the right side tends to zero. The Pinsker inequality says that the square of the total variation distance of two probability measures is at most one half of the relative entropy. Hence
\[
\TV\big(\mathbf P^n,\widetilde{\mathbf P}^n\big)^2\le\frac{t}{8T}\int_{\R^d}|\nabla V_n-\nabla V|^2\,d\gamma_n\longrightarrow0 .
\]
Under $\mathbf P^n$ the law of $X_t$ is $\gamma_n$. Under $\widetilde{\mathbf P}^n$ it is $\gamma_nP_t=\int\gamma_n(dx)P_t(x,\cdot)$. So for every Borel set $A$,
\[
|\gamma_n(A)-\gamma_nP_t(A)|=\big|\mathbf P^n(X_t\in A)-\widetilde{\mathbf P}^n(X_t\in A)\big|\le\TV\big(\mathbf P^n,\widetilde{\mathbf P}^n\big) .
\]
We take the supremum over $A$ and obtain
\[
\TV\big(\gamma_n,\ \gamma_nP_t\big)\longrightarrow0 .
\]

\emph{Step 3: invariance.} First we show that $\TV(\gamma_nP_t,\gamma P_t)\le\TV(\gamma_n,\gamma)$. Let $A$ be a Borel set and put $f(x)=P_t(x,A)$, so that $0\le f\le1$. Since $\int(g_n-g)\,dx=0$,
\[
\gamma_nP_t(A)-\gamma P_t(A)=\int f\,(g_n-g)\,dx=\int\big(f-\tfrac12\big)(g_n-g)\,dx .
\]
Since $|f-\frac12|\le\frac12$, we get
\[
|\gamma_nP_t(A)-\gamma P_t(A)|\le\tfrac12\int|g_n-g|\,dx=\TV(\gamma_n,\gamma) ,
\]
and the claim follows by taking the supremum over $A$. By the triangle inequality,
\begin{align*}
\TV(\gamma,\gamma P_t)&\le\TV(\gamma,\gamma_n)+\TV(\gamma_n,\gamma_nP_t)+\TV(\gamma_nP_t,\gamma P_t)\\
&\le2\,\TV(\gamma,\gamma_n)+\TV(\gamma_n,\gamma_nP_t) .
\end{align*}
By Step 1 and Step 2 the right side tends to zero as $n\to\infty$. Since the left side does not depend on $n$, we get $\TV(\gamma,\gamma P_t)=0$. Thus $\gamma P_t=\gamma$ for every $t>0$.

\emph{Step 4: convergence.} By \ref{S2} the measures $P_t(x,\cdot)$, $x\in\R^d$, $t>0$, are all equivalent to the Lebesgue measure, hence to each other. Moreover $P_t(x,\cdot)\to\delta_x$ weakly as $t\to0$. Indeed, let $f$ be a bounded continuous function and let $X$ be the solution started at $x$. The paths are continuous, so $f(X_t)\to f(x)$ as $t\to0$. By bounded convergence, $\int f(y)\,P_t(x,dy)=\E f(X_t)\to f(x)$. A Markov semigroup with these properties which has an invariant probability measure $\gamma$ satisfies $P_t(x,A)\to\gamma(A)$ for all $x$ and all Borel sets $A$. This is the theorem of \cite{Doob1948}, in the form \cite[Theorem 4.2.1]{DPZ1996}. Finally,
\[
\rho_{T,t}(A)=\int P_t(x,A)\,\rho_0(dx)\longrightarrow\int\gamma(A)\,\rho_0(dx)=\gamma(A)
\]
by dominated convergence, since $0\le P_t(x,A)\le1$ and $\rho_0$ is a probability measure. This completes the proof.
\end{proof}

\subsection{Distances between probability measures}
For probability measures $\mu,\nu$ on $\R^d$ we use the total variation distance $\TV$ defined above. When the second moments are finite, we also use the Wasserstein distance
\[
\Wt(\mu,\nu)=\inf\big\{(\E|\xi-\xi'|^2)^{1/2}:\ \Law(\xi)=\mu,\ \Law(\xi')=\nu\big\}.
\]
We recall \cite[Chapter 7]{Villani2003} that $\Wt(\mu_n,\mu)\to0$ if and only if $\mu_n\to\mu$ weakly and the second moments converge, and that $\Wt^2$ is convex in each argument. Suppose that $\mu$ and $\nu$ are carried by a set of diameter $D$. Then
\begin{equation}\label{eq:WTV}
\Wt(\mu,\nu)^2\le D^2\,\TV(\mu,\nu).
\end{equation}
To see this, we take a coupling $(\xi,\xi')$ of $\mu$ and $\nu$ with $\P(\xi\ne\xi')=\TV(\mu,\nu)$. Almost surely both variables lie in the set which carries $\mu$ and $\nu$, so $|\xi-\xi'|\le D$. Hence
\[
\Wt(\mu,\nu)^2\le\E|\xi-\xi'|^2=\E\big(|\xi-\xi'|^2\,\one_{\{\xi\ne\xi'\}}\big)\le D^2\,\P(\xi\ne\xi')=D^2\,\TV(\mu,\nu) .
\]
We recall the formula $\TV(\mu,\nu)=\frac12\sup\{|\int f\,d\mu-\int f\,d\nu|:f\text{ continuous},\ |f|\le1\}$. It shows that the total variation distance is lower semicontinuous for weak convergence. That is, if $\mu_n\to\mu$ and $\nu_n\to\nu$ weakly, then
\begin{equation}\label{eq:TVlsc}
\TV(\mu,\nu)\le\liminf_{n\to\infty}\TV(\mu_n,\nu_n).
\end{equation}
For the proof, let $f$ be continuous with $|f|\le1$. By weak convergence and by the same formula for $\mu_n$ and $\nu_n$,
\[
\frac12\Big|\int f\,d\mu-\int f\,d\nu\Big|=\lim_{n\to\infty}\frac12\Big|\int f\,d\mu_n-\int f\,d\nu_n\Big|\le\liminf_{n\to\infty}\TV(\mu_n,\nu_n) ,
\]
and we take the supremum over $f$. We denote by $f_\#\mu$ the image of a measure $\mu$ under a Borel map $f$.

We shall use two elementary facts about total variation. For a nonnegative measure $\mu$ and a number $a\ge0$ we write $\mu\ge a\,m$ if $\mu(A)\ge a\,m(A)$ for all Borel sets $A$.

\begin{lemma}\label{lem:doeblin}
\begin{enumerate}[label=\textup{(\roman*)}]
\item Let $\mu_1,\mu_2,m_1,m_2$ be probability measures on $\R^d$ and $a\in[0,1]$ with $\mu_1\ge a\,m_1$ and $\mu_2\ge a\,m_2$. Then $\TV(\mu_1,\mu_2)\le1-a\,(1-\TV(m_1,m_2))$.
\item Let $\Pi$ be a Markov kernel on $\R^d$, let $M\subset\R^d$ be a Borel set and let $\alpha\in[0,1]$ be such that $\TV(\Pi(y,\cdot),\Pi(y',\cdot))\le1-\alpha$ for all $y,y'\in M$. Then for all probability measures $\mu,\mu'$ on $\R^d$
\[
\TV(\mu \Pi,\mu'\Pi)\le(1-\alpha)\,\TV(\mu,\mu')+\alpha\,\big(\mu(M^c)+\mu'(M^c)\big).
\]
\end{enumerate}
\end{lemma}

\begin{proof}
(i) Put $\lambda=\mu_1+\mu_2+m_1+m_2$, and let $f_1,f_2,g_1,g_2$ be the densities of $\mu_1,\mu_2,m_1,m_2$ with respect to $\lambda$. The largest measure below $\mu_1$ and $\mu_2$ is the measure $\mu_1\wedge\mu_2$ with density $\min(f_1,f_2)$. We first recall that
\begin{equation}\label{eq:xs1e}
\TV(\mu_1,\mu_2)=1-\int\min(f_1,f_2)\,d\lambda=1-(\mu_1\wedge\mu_2)(\R^d) .
\end{equation}
Indeed, for a Borel set $B$ the difference $\mu_1(B)-\mu_2(B)=\int_B(f_1-f_2)\,d\lambda$ is largest for $B=\{f_1>f_2\}$. On this set $f_1-f_2=f_1-\min(f_1,f_2)$, while $f_1-\min(f_1,f_2)=0$ outside it. So
\[
\sup_B\big(\mu_1(B)-\mu_2(B)\big)=\int\big(f_1-\min(f_1,f_2)\big)\,d\lambda=1-\int\min(f_1,f_2)\,d\lambda .
\]
Since $\mu_2(B)-\mu_1(B)=\mu_1(B^c)-\mu_2(B^c)$, the supremum of $|\mu_1(B)-\mu_2(B)|$ has the same value. This proves \eqref{eq:xs1e}. The same formula holds for $m_1$ and $m_2$. Now the hypothesis $\mu_i\ge a\,m_i$ means that $f_i\ge a\,g_i$ almost everywhere with respect to $\lambda$, for $i=1,2$. Hence $\min(f_1,f_2)\ge a\min(g_1,g_2)$. We integrate and get
\[
1-\TV(\mu_1,\mu_2)=\int\min(f_1,f_2)\,d\lambda\ge a\int\min(g_1,g_2)\,d\lambda=a\,\big(1-\TV(m_1,m_2)\big) .
\]
This is the claim.

(ii) Put $\varkappa(y,y')=\TV(\Pi(y,\cdot),\Pi(y',\cdot))$. The function $\varkappa$ is measurable, because the supremum in the definition of the total variation distance may be taken over a countable algebra of sets which generates the Borel sets. Let $(Y,Y')$ be a coupling of $\mu$ and $\mu'$ with $\P(Y\ne Y')=\TV(\mu,\mu')$. Such a coupling exists. For every Borel set $A$ we have $\mu \Pi(A)=\E\,\Pi(Y,A)$ and $\mu'\Pi(A)=\E\,\Pi(Y',A)$. Hence
\[
|\mu \Pi(A)-\mu'\Pi(A)|\le\E\,|\Pi(Y,A)-\Pi(Y',A)|\le\E\,\varkappa(Y,Y') .
\]
Taking the supremum over $A$ we get $\TV(\mu \Pi,\mu'\Pi)\le\E\,\varkappa(Y,Y')$. Now $\varkappa(y,y')=0$ if $y=y'$, $\varkappa(y,y')\le1-\alpha$ if $y$ and $y'$ are in $M$, and $\varkappa\le1$ always. Hence
\[
\varkappa(Y,Y')\le\one_{\{Y\ne Y'\}}-\alpha\,\one_{\{Y\ne Y',\ Y\in M,\ Y'\in M\}} .
\]
We take expectations and obtain
\[
\TV(\mu \Pi,\mu'\Pi)\le\E\,\TV\big(\Pi(Y,\cdot),\Pi(Y',\cdot)\big)\le\P(Y\ne Y')-\alpha\,\P\big(Y\ne Y',\ Y\in M,\ Y'\in M\big).
\]
The event $\{Y\ne Y'\}$ is contained in the union of the events $\{Y\ne Y',\ Y\in M,\ Y'\in M\}$, $\{Y\notin M\}$ and $\{Y'\notin M\}$. So the last probability is at least
\[
\P(Y\ne Y')-\P(Y\notin M)-\P(Y'\notin M)=\TV(\mu,\mu')-\mu(M^c)-\mu'(M^c) .
\]
We insert this bound and get
\begin{align*}
\TV(\mu \Pi,\mu'\Pi)&\le\TV(\mu,\mu')-\alpha\,\big(\TV(\mu,\mu')-\mu(M^c)-\mu'(M^c)\big)\\
&=(1-\alpha)\,\TV(\mu,\mu')+\alpha\,\big(\mu(M^c)+\mu'(M^c)\big) .
\end{align*}
This completes the proof.
\end{proof}

\section{Envelopes by OU processes}\label{sec:envelope}

The simplest feasible set is a single point, say $C=\{0\}$. Then \eqref{eq:state} is the linear equation $x(t)=x_0-\int_0^tx(u)du+w(t)$. Its solution is $x(t)=e^{-t}x_0+\zeta(t)$ with $\zeta$ given by \eqref{eq:ouformula}, and so $d_C(x(t))\le e^{-t}d_C(x_0)+|\zeta(t)|$. The first theorem of this section says that the same inequality is true for every closed convex set in every Hilbert space. It also says that a moving set costs only an additive term equal to its speed. By the second theorem, the inequality fails for sets which are not convex. The third theorem says what remains true for an arbitrary closed set.

\subsection{The envelope for convex sets}

\begin{theorem}\label{thm:envelope}
Let $\{C(t)\}_{t\ge0}$ be a moving convex set of speed $v$ in $H$, let $x_0\in H$ and let $w:[0,\infty)\to H$ be continuous with $w(0)=0$. Let $x$ and $\zeta$ be the solutions of \eqref{eq:state} and \eqref{eq:ou}. Then for every $t\ge0$
\begin{align}
d_{C(t)}\big(x(t)-\zeta(t)\big)&\le e^{-t}\,d_{C(0)}(x_0)+v\,(1-e^{-t}),\label{eq:env1}\\
d_{C(t)}\big(x(t)\big)&\le e^{-t}\,d_{C(0)}(x_0)+v\,(1-e^{-t})+|\zeta(t)|.\label{eq:env2}
\end{align}
In particular, the state process and the envelope process of \eqref{eq:SDE} satisfy, almost surely and for all $t\ge0$,
\begin{equation}\label{eq:env3}
d_{C(t)}(X_t)\le e^{-t}\,d_{C(0)}(X_0)+v\,(1-e^{-t})+|\zeta_t| .
\end{equation}
\end{theorem}

\begin{proof}
\emph{Step 1.} Put $p(u)=P_{C(u)}x(u)$. Then $p(u)\in C(u)$ for every $u$. We show that the function $p:[0,\infty)\to H$ is continuous. By \eqref{eq:jointcont}, for $u,u'\ge0$,
\[
|p(u)-p(u')|\le|x(u)-x(u')|+\Big(\big(d_{C(u)}(x(u'))+d_{C(u')}(x(u'))\big)\,v\,|u-u'|\Big)^{1/2}.
\]
Moreover $d_{C(u)}(x(u'))\le d_{C(u')}(x(u'))+v\,|u-u'|$ by \eqref{eq:dhaus} and Definition \ref{def:path}. Since $x$ is continuous, the right side of the last display tends to zero when $u\to u'$. Hence $p$ is continuous.

\emph{Step 2.} Let $y=x-\zeta$. We subtract \eqref{eq:ou} from \eqref{eq:state}. The terms $w(t)$ cancel and we get
\[
y(t)=x_0-\int_0^t\big(x(u)-p(u)-\zeta(u)\big)\,du=x_0-\int_0^t\big(y(u)-p(u)\big)\,du .
\]
The function under the integral sign is continuous. So $y$ is continuously differentiable, $y'=-y+p$ and $y(0)=x_0$. We solve this linear equation. For every $u\ge0$,
\[
\frac{d}{du}\big(e^{u}y(u)\big)=e^{u}\big(y'(u)+y(u)\big)=e^{u}p(u).
\]
Integrating from $0$ to $t$ we get
\[
e^{t}y(t)-x_0=\int_0^te^{u}p(u)\,du .
\]
Now we multiply by $e^{-t}$ and obtain
\begin{equation}\label{eq:bary}
y(t)=e^{-t}x_0+\int_0^te^{-(t-u)}p(u)\,du .
\end{equation}

\emph{Step 3.} Fix $t\ge0$. Since
\[
e^{-t}+\int_0^te^{-(t-u)}\,du=e^{-t}+(1-e^{-t})=1 ,
\]
formula \eqref{eq:bary} shows that $y(t)$ is the barycentre of a probability measure. This measure gives the mass $e^{-t}$ to the point $x_0$ and the mass $e^{-(t-u)}du$ to the point $p(u)$, $0\le u\le t$. The function $f=d_{C(t)}$ is convex and continuous by \ref{P3}. By the Jensen inequality,
\begin{equation}\label{eq:jensen}
f(y(t))\le e^{-t}f(x_0)+\int_0^te^{-(t-u)}f(p(u))\,du .
\end{equation}
We give the proof of \eqref{eq:jensen} in the present vector valued setting. For $t=0$ there is nothing to prove, so let $t>0$. Let $n\ge1$, put $u_i=it/n$ for $i=0,1,\dots,n$, and let $I_i=[u_{i-1},u_i]$ and
\[
m_i=\int_{I_i}e^{-(t-u)}\,du\qquad(i=1,\dots,n).
\]
Then $m_i\ge0$ and $e^{-t}+\sum_{i=1}^nm_i=1$. So the point
\[
y_n=e^{-t}x_0+\sum_{i=1}^nm_i\,p(u_i)
\]
is a finite convex combination. By the definition of convexity
\begin{equation}\label{eq:xs2a}
f(y_n)\le e^{-t}f(x_0)+\sum_{i=1}^nm_i\,f(p(u_i)).
\end{equation}
Let $\delta_n=\sup\{|p(u)-p(u')|:u,u'\in[0,t],\ |u-u'|\le t/n\}$. The function $p$ is uniformly continuous on the compact interval $[0,t]$, so $\delta_n\to0$ as $n\to\infty$. By \eqref{eq:bary},
\[
y_n-y(t)=\sum_{i=1}^n\int_{I_i}e^{-(t-u)}\big(p(u_i)-p(u)\big)\,du ,\qquad\text{so}\qquad|y_n-y(t)|\le\delta_n\sum_{i=1}^nm_i\le\delta_n .
\]
The function $f$ is nonexpansive by \ref{P3}. Therefore $|f(y_n)-f(y(t))|\le\delta_n$. Also $|f(p(u_i))-f(p(u))|\le\delta_n$ for $u\in I_i$, and this gives
\begin{align*}
\Big|\sum_{i=1}^nm_i\,f(p(u_i))-\int_0^te^{-(t-u)}f(p(u))\,du\Big|&=\Big|\sum_{i=1}^n\int_{I_i}e^{-(t-u)}\big(f(p(u_i))-f(p(u))\big)\,du\Big|\\
&\le\delta_n\sum_{i=1}^nm_i\le\delta_n .
\end{align*}
Letting $n\to\infty$ in \eqref{eq:xs2a} we obtain \eqref{eq:jensen}.

We estimate the two terms on the right of \eqref{eq:jensen}. By \eqref{eq:dhaus} and Definition \ref{def:path},
\[
f(x_0)=d_{C(t)}(x_0)\le d_{C(0)}(x_0)+\haus(C(0),C(t))\le d_{C(0)}(x_0)+vt .
\]
Since $p(u)\in C(u)$, the definition of the Hausdorff distance gives
\[
f(p(u))=d_{C(t)}(p(u))\le\haus(C(u),C(t))\le v\,(t-u).
\]
With the substitution $r=t-u$, and since the derivative of $-(1+r)e^{-r}$ is $r\,e^{-r}$,
\[
\int_0^te^{-(t-u)}v\,(t-u)\,du=v\int_0^tr\,e^{-r}\,dr=v\,\Big[-(1+r)\,e^{-r}\Big]_{r=0}^{r=t}=v\,\big(1-e^{-t}-t\,e^{-t}\big).
\]
We insert these three facts in \eqref{eq:jensen} and get
\[
f(y(t))\le e^{-t}d_{C(0)}(x_0)+v\,t\,e^{-t}+v\,\big(1-e^{-t}-t\,e^{-t}\big)=e^{-t}d_{C(0)}(x_0)+v\,(1-e^{-t}).
\]
This is \eqref{eq:env1}.

\emph{Step 4.} The function $d_{C(t)}$ is nonexpansive and $x(t)-y(t)=\zeta(t)$. Hence
\[
d_{C(t)}(x(t))\le d_{C(t)}(y(t))+|x(t)-y(t)|=d_{C(t)}(y(t))+|\zeta(t)|,
\]
and \eqref{eq:env2} follows from \eqref{eq:env1}. Finally \eqref{eq:env3} is \eqref{eq:env2} applied to each path $w=\sqrt{2T}\,W(\omega)$. This completes the proof.
\end{proof}

\begin{remark}\label{rem:sharp}
Each term on the right of \eqref{eq:env2} is needed. Take $H=\R$. For $C(t)=\{0\}$ and $x_0=0$ we have $x=\zeta$, so \eqref{eq:env2} is an equality. Now let $w=0$, $C(t)=\{vt\}$ and $x_0=-a$ with $a\ge0$. Here \eqref{eq:state} reads $x'(t)=vt-x(t)$, so the function $z(t)=vt-x(t)$ solves $z'=v-z$, $z(0)=a$. Hence $z(t)=v+(a-v)e^{-t}\ge0$ and
\[
d_{C(t)}(x(t))=z(t)=e^{-t}a+v\,(1-e^{-t}),
\]
which is again an equality. So a state which follows a set moving with speed $v$ lags behind it by $v$ in the long run. This lag cannot be improved.
\end{remark}

\begin{remark}\label{rem:excess}
In Step 3 we used only the one sided quantity $\sup_{c\in C(u)}d_{C(t)}(c)\le v\,(t-u)$ for $u\le t$, and not the full Hausdorff distance. Hence, if the moving set is increasing, that is, $C(u)\subset C(t)$ for $u\le t$, then \eqref{eq:env1}, \eqref{eq:env2} and \eqref{eq:env3} hold with $v=0$, whatever its speed is. A feasible set which grows causes no lag.
\end{remark}

\begin{remark}\label{rem:step2}
Step 2 of the proof does not use convexity. If $C\subset\R^d$ is any closed set and a continuous function $x$ satisfies $x(t)=x_0-\int_0^t(x(u)-p(u))du+w(t)$ with a locally bounded Borel function $p$, then $y=x-\zeta$ is given by \eqref{eq:bary}. Indeed, in this case $y(t)=x_0-\int_0^t(y(u)-p(u))du$. So $e^{u}y(u)$ is the integral of its derivative $e^{u}p(u)$, which exists for almost every $u$, and the computation of Step 2 applies. Convexity enters only in Step 3, where we use that a barycentre of points of $C$ is close to $C$.
\end{remark}

\subsection{The converse}
We now work in $\R^d$ and use the notation of Subsection \ref{subsec:nonconvex}.

\begin{theorem}\label{thm:converse}
Let $C\subset\R^d$ be nonempty and closed. The following statements are equivalent.
\begin{enumerate}[label=\textup{(\roman*)}]
\item $C$ is convex.
\item For every $x_0\in\R^d$, every continuous $w:[0,\infty)\to\R^d$ with $w(0)=0$ and every pair $(x,p)$ of a continuous function $x$ and a Borel function $p$ such that $p(u)\in\Pi_C(x(u))$ for all $u$ and $x(t)=x_0-\int_0^t(x(u)-p(u))du+w(t)$ for all $t$, we have
\begin{equation}\label{eq:env4}
d_C(x(t))\le e^{-t}d_C(x_0)+|\zeta(t)|\qquad(t\ge0),
\end{equation}
where $\zeta$ is given by \eqref{eq:ouformula}.
\item There is $T>0$ such that for every $x_0\in C$ the solution $X$ of \eqref{eq:SDEgen} with $X_0=x_0$ satisfies $d_C(X_t)\le|\zeta_t|$ for all $t\ge0$, almost surely.
\end{enumerate}
More precisely, if $C$ is not convex, then for every $T>0$ there are $x_0\in C$ and $t_0>0$ such that the solution of \eqref{eq:SDEgen} with $X_0=x_0$ satisfies $\P\big(d_C(X_{t_0})>|\zeta_{t_0}|\big)>0$.
\end{theorem}

\begin{proof}
(i)$\Rightarrow$(ii). If $C$ is convex, then $\Pi_C(x)=\{P_Cx\}$ for every $x$. So $p(u)=P_Cx(u)$. Thus $x$ is the solution of \eqref{eq:state} for the constant family $C(t)=C$, which has speed $0$. Hence \eqref{eq:env4} is \eqref{eq:env2} with $v=0$.

(ii)$\Rightarrow$(iii). Let $T>0$ and $x_0\in C$. Almost every path of the solution of \eqref{eq:SDEgen} satisfies the integral equation in (ii) with $w=\sqrt{2T}W$ and $p(u)=P_C(X_u)$. Also $d_C(x_0)=0$. So \eqref{eq:env4} gives $d_C(X_t)\le|\zeta_t|$ for all $t\ge0$.

(iii)$\Rightarrow$(i). We prove the last statement of the theorem, which is the negation of (iii) when $C$ is not convex. Let $T>0$ and assume that $C$ is not convex. Then there are $c_1,c_2\in C$ and $\lambda\in(0,1)$ such that the point
\[
c_\lambda:=(1-\lambda)c_1+\lambda c_2
\]
does not belong to $C$. Since $C$ is closed, the distance of $c_\lambda$ from $C$ is positive. Put
\[
\eta_0:=d_C(c_\lambda)>0,\qquad R:=|c_1-c_2|>0,\qquad t_\lambda:=\log\frac1{1-\lambda},\quad\text{so that }e^{-t_\lambda}=1-\lambda .
\]

\emph{Step 1: a path to follow.} Let $0<\eps<t_\lambda$ and $t_0=t_\lambda+\eps$. Define $\bar x:[0,t_0]\to\R^d$ by
\[
\bar x(u)=\begin{cases}c_1+\dfrac u\eps\,(c_2-c_1),&0\le u\le\eps,\\[6pt] c_2,&\eps\le u\le t_\lambda,\\[4pt] c_2+\dfrac{u-t_\lambda}\eps\,(c_\lambda-c_2),&t_\lambda\le u\le t_0.\end{cases}
\]
This is a continuous path which starts at $c_1$, moves quickly to $c_2$, waits there until time $t_\lambda$ and then moves quickly to $c_\lambda$. All its values lie on the segment $[c_1,c_2]$.

\emph{Step 2: the state follows this path with positive probability.} Now $X$ denotes the solution of \eqref{eq:SDEgen} with $X_0=c_1$. Let $\eps'>0$ and consider the event
\[
G=\Big\{\sup_{0\le u\le t_0}|X_u-\bar x(u)|<\eps'\Big\}.
\]
By \ref{S2} the law of $(X_u)_{u\le t_0}$ is equivalent to the law of $(c_1+\sqrt{2T}W_u)_{u\le t_0}$. So it is enough to show that
\begin{equation}\label{eq:xs2b}
\P\Big(\sup_{0\le u\le t_0}\big|c_1+\sqrt{2T}\,W_u-\bar x(u)\big|<\eps'\Big)>0 .
\end{equation}
It is well known that the Wiener measure gives positive mass to every uniform neighbourhood of every continuous path $h$ starting at the origin. We need this only for the path
\[
h(u)=\frac{\bar x(u)-c_1}{\sqrt{2T}}\qquad(0\le u\le t_0),
\]
and for this path we give the proof. The function $h$ is continuous and piecewise linear with $h(0)=0$. So $h(u)=\int_0^uh'(s)\,ds$, where $h'$ is defined except at the two points $\eps$ and $t_\lambda$ and $|h'|\le R/(\eps\sqrt{2T})$. By the Girsanov theorem the law of $(W_u-h(u))_{u\le t_0}$ is equivalent to the law of $(W_u)_{u\le t_0}$. Moreover $\P(\sup_{u\le t_0}|W_u|<r)>0$ for every $r>0$. Taking $r=\eps'/\sqrt{2T}$ we get
\[
\P\Big(\sup_{0\le u\le t_0}\big|c_1+\sqrt{2T}\,W_u-\bar x(u)\big|<\eps'\Big)=\P\Big(\sup_{0\le u\le t_0}|W_u-h(u)|<r\Big)>0 ,
\]
which is \eqref{eq:xs2b}. Hence $\P(G)>0$.

\emph{Step 3: bounds on the event $G$.} We work on $G$ and write $p_u=P_C(X_u)$. Every point of the segment $[c_1,c_2]$ has distance at most $R/2$ from the set $\{c_1,c_2\}\subset C$. So $d_C(\bar x(u))\le R/2$. Since the function $d_C$ is nonexpansive, we get for all $u\le t_0$
\[
d_C(X_u)\le d_C(\bar x(u))+|X_u-\bar x(u)|\le d_C(\bar x(u))+\eps'\le\frac R2+\eps' .
\]
We have $|p_u-X_u|=d_C(X_u)$. Also $|\bar x(u)-c_2|\le R$, because both points lie on the segment $[c_1,c_2]$. Therefore
\begin{equation}\label{eq:pbound1}
|p_u-c_2|\le|p_u-X_u|+|X_u-\bar x(u)|+|\bar x(u)-c_2|\le\Big(\frac R2+\eps'\Big)+\eps'+R=\frac{3R}2+2\eps' .
\end{equation}
For $\eps\le u\le t_\lambda$ we have $\bar x(u)=c_2\in C$. So $|X_u-c_2|<\eps'$ and $d_C(X_u)\le|X_u-c_2|<\eps'$, which gives the better bound
\begin{equation}\label{eq:pbound2}
|p_u-c_2|\le|p_u-X_u|+|X_u-c_2|<2\eps'\qquad(\eps\le u\le t_\lambda).
\end{equation}

\emph{Step 4: the barycentre is close to $c_\lambda$.} Let $y=X-\zeta$. The function $u\mapsto p_u$ is Borel and locally bounded, because $|p_u|\le|X_u|+d_C(X_u)$. By Remark \ref{rem:step2}, formula \eqref{eq:bary} holds with $x_0=c_1$ and $p(u)=p_u$, that is,
\[
y(t_0)=e^{-t_0}c_1+\int_0^{t_0}e^{-(t_0-u)}p_u\,du .
\]
The weights in \eqref{eq:bary} have total mass one, so $c_2=e^{-t_0}c_2+\int_0^{t_0}e^{-(t_0-u)}c_2\,du$. We subtract this from the last display and get
\[
y(t_0)-c_2=e^{-t_0}(c_1-c_2)+\int_0^{t_0}e^{-(t_0-u)}(p_u-c_2)\,du .
\]
On the other hand $c_\lambda-c_2=(1-\lambda)(c_1-c_2)=e^{-t_\lambda}(c_1-c_2)$. A second subtraction gives
\[
y(t_0)-c_\lambda=-\big(e^{-t_\lambda}-e^{-t_0}\big)(c_1-c_2)+\int_0^{t_0}e^{-(t_0-u)}(p_u-c_2)\,du .
\]
We take norms, use $e^{-(t_0-u)}\le1$ and split the integral into the integrals over $[0,\eps]$, $[\eps,t_\lambda]$ and $[t_\lambda,t_0]$. On the first and on the third interval, both of length $\eps$, we use \eqref{eq:pbound1}. On $[\eps,t_\lambda]$ we use \eqref{eq:pbound2}. This gives
\begin{align*}
|y(t_0)-c_\lambda|&\le\big(e^{-t_\lambda}-e^{-t_0}\big)R+2\eps\Big(\frac{3R}2+2\eps'\Big)+(t_\lambda-\eps)\,2\eps'\\
&\le\eps R+3\eps R+4\eps\eps'+2t_\lambda\eps'=4\eps R+4\eps\eps'+2t_\lambda\eps' ,
\end{align*}
where we used $e^{-t_\lambda}-e^{-t_0}=e^{-t_\lambda}(1-e^{-\eps})\le1-e^{-\eps}\le\eps$ and $(t_\lambda-\eps)\,2\eps'\le2t_\lambda\eps'$.

\emph{Step 5: the envelope is violated.} On $G$ we have $|X_{t_0}-c_\lambda|<\eps'$ since $\bar x(t_0)=c_\lambda$. Therefore, by Step 4 and since $d_C$ is nonexpansive,
\begin{align*}
|\zeta_{t_0}|&=|X_{t_0}-y(t_0)|\le|X_{t_0}-c_\lambda|+|c_\lambda-y(t_0)|<\eps'+4\eps R+4\eps\eps'+2t_\lambda\eps' ,\\
d_C(X_{t_0})&\ge d_C(c_\lambda)-|X_{t_0}-c_\lambda|>\eta_0-\eps' .
\end{align*}
Now choose
\[
\eps=\min\Big\{\frac{t_\lambda}2,\ \frac{\eta_0}{32R}\Big\},\qquad\eps'=\frac{\eta_0}{4(1+2\eps+t_\lambda)} .
\]
Then $0<\eps<t_\lambda$, as required in Step 1. Further $4\eps R\le4R\,\eta_0/(32R)=\eta_0/8$ and
\[
\eps'(2+4\eps+2t_\lambda)=2\eps'(1+2\eps+t_\lambda)=\frac{\eta_0}2 .
\]
We subtract the two bounds above and obtain
\begin{align*}
d_C(X_{t_0})-|\zeta_{t_0}|&>\eta_0-\eps'-\eps'-4\eps R-4\eps\eps'-2t_\lambda\eps'\\
&\ge\eta_0-\eps'-\eps'-\frac{\eta_0}8-4\eps\eps'-2t_\lambda\eps'=\frac{7\eta_0}8-\eps'(2+4\eps+2t_\lambda)\\
&=\frac{7\eta_0}8-\frac{\eta_0}2=\frac{3\eta_0}8>0 .
\end{align*}
Thus $d_C(X_{t_0})>|\zeta_{t_0}|$ on $G$, and $\P(G)>0$. Since $X_0=c_1\in C$, this proves the last statement. Hence (iii) fails.
\end{proof}

\begin{remark}
A closed subset of $\R^d$ is convex if and only if every point of $\R^d$ has a unique nearest point in it, that is, if and only if it is a Chebyshev set \cite[Chapter 12]{Deutsch2001}. So Theorem \ref{thm:converse} says that the envelope \eqref{eq:env4} holds exactly for the Chebyshev sets. The proof also shows why it fails for a set which is not convex. The barycentre \eqref{eq:bary} of past projections may lie far from the set although all the projections lie in it.
\end{remark}

\subsection{Arbitrary closed sets: a comparison in law}
By Theorem \ref{thm:converse}, when $C$ is not convex the distance $d_C(X_t)$ is not bounded by the envelope process driven by the Brownian motion of the state. We now show that it is always bounded by a process which has the law of the norm of an OU process. This process is driven by a real Brownian motion constructed from the state. So the laws of the two sides of \eqref{eq:env4} are ordered for every closed set. For a convex set Theorem \ref{thm:envelope} says more, because there the bound holds with the OU process $\zeta$ which is driven by the Brownian motion of the state.

\begin{theorem}\label{thm:compare}
Let $C\subset\R^d$ be nonempty and closed. Let $T>0$, $x\in\R^d$ and let $X$ be the solution of \eqref{eq:SDEgen} with $X_0=x$.
\begin{enumerate}[label=\textup{(\roman*)}]
\item There is a continuous adapted increasing process $A$ with $A_0=0$ such that, for all $t\ge0$,
\begin{equation}\label{eq:Vito}
V_C(X_t)=V_C(x)+\int_0^t\big(Td-2V_C(X_u)\big)\,du+\sqrt{2T}\int_0^t\ip{X_u-P_CX_u}{dW_u}-A_t .
\end{equation}
\item Let $v\in\R^d$ with $|v|=d_C(x)$. On an extension of the probability space there is a continuous process $\varrho$ which has the law of $(|e^{-t}v+\zeta_t|)_{t\ge0}$ and satisfies
\[
d_C(X_t)\le\varrho_t\qquad\text{for all }t\ge0,\ \text{almost surely}.
\]
\item For all $t\ge0$, $N>0$ and $r>0$,
\begin{equation}\label{eq:msq}
\begin{aligned}
\E\,d_C(X_t)^2&\le e^{-2t}d_C(x)^2+dT\,(1-e^{-2t}),\\
\P\big(d_C(X_t)\ge r\big)&\le\P\big(|\zeta_t|\ge r-d_C(x)\big),\\
\P\Big(\sup_{0\le t\le N}d_C(X_t)\ge r\Big)&\le\P\Big(\sup_{0\le t\le N}|\zeta_t|\ge r-d_C(x)\Big).
\end{aligned}
\end{equation}
\end{enumerate}
\end{theorem}

\begin{proof}
We write $V=V_C$.

(i) For $y\in\R^d$ and $c\in C$ we have $\frac12|y-c|^2=\frac12|y|^2-\big(\ip yc-\frac12|c|^2\big)$. Taking the infimum over $c\in C$ we see that $V(y)=\frac12|y|^2-\psi(y)$ with $\psi(y)=\sup_{c\in C}\{\ip yc-\frac12|c|^2\}$. The function $\psi$ is a supremum of affine functions, so it is convex. It is finite because $\psi=\frac12|\cdot|^2-V$. Let $\theta_n$ be the densities of Step 1 of the proof of Lemma \ref{lem:invariant} and $V_n=V*\theta_n$. Then $V_n=\frac12|\cdot|^2*\theta_n-\psi*\theta_n$. The function $\psi*\theta_n$ is convex and smooth, so its matrix of second derivatives is nonnegative. Next,
\[
\big(\tfrac12|\cdot|^2*\theta_n\big)(y)=\int_{\R^d}\tfrac12|y-z|^2\,\theta_n(z)\,dz ,
\]
and we may differentiate twice under the integral sign. Since $\int\theta_n(z)\,dz=1$, the matrix of second derivatives of this function is the identity matrix. Therefore the second derivative of $V_n$ is at most the identity matrix. Taking the trace we get
\[
\Delta V_n\le d .
\]
The function $V_n$ is smooth and $dX_u=-(X_u-P_CX_u)\,du+\sqrt{2T}\,dW_u$. By the It\^o formula, for $0\le s\le t$,
\begin{align*}
V_n(X_t)-V_n(X_s)&=\int_s^t\big[-\ip{\nabla V_n(X_u)}{X_u-P_CX_u}+T\Delta V_n(X_u)\big]\,du\\
&\quad+\sqrt{2T}\int_s^t\ip{\nabla V_n(X_u)}{dW_u},
\end{align*}
and hence
\begin{equation}\label{eq:Vnito}
V_n(X_t)-V_n(X_s)\le\int_s^t\big[-\ip{\nabla V_n(X_u)}{X_u-P_CX_u}+Td\big]\,du+\sqrt{2T}\int_s^t\ip{\nabla V_n(X_u)}{dW_u} .
\end{equation}
We let $n\to\infty$ and consider the three terms of \eqref{eq:Vnito} one after the other.

\emph{The left side.} The estimates \eqref{eq:Vn} hold for every closed set, with the same proof. So $|V_n-V|\le(d_C+1)/n$ and $|\nabla V_n|\le d_C+1$. Hence $V_n\to V$ uniformly on compact sets, and $V_n(X_t)-V_n(X_s)\to V(X_t)-V(X_s)$.

\emph{The integral with respect to $du$.} We have $\nabla V_n(y)\to\nabla V(y)=y-P_Cy$ for all $y$ outside a Lebesgue null set $F$. By \ref{S2} the law of $X_u$ is equivalent to the Lebesgue measure for every $u>0$, so $\P(X_u\in F)=0$ and
\[
\E\int_0^\infty\one_F(X_u)\,du=\int_0^\infty\P(X_u\in F)\,du=0 .
\]
Thus the time which $X$ spends in $F$ is zero almost surely. We fix a path with this property. For almost every $u\in[s,t]$ we then have
\[
\ip{\nabla V_n(X_u)}{X_u-P_CX_u}\longrightarrow|X_u-P_CX_u|^2=d_C(X_u)^2=2V(X_u),
\]
and for all $u\in[s,t]$ and all $n$
\[
\big|\ip{\nabla V_n(X_u)}{X_u-P_CX_u}\big|\le\big(d_C(X_u)+1\big)\,d_C(X_u)\le\sup_{s\le u'\le t}\big(d_C(X_{u'})+1\big)^2<\infty .
\]
By dominated convergence along each such path,
\[
\int_s^t\ip{\nabla V_n(X_u)}{X_u-P_CX_u}\,du\longrightarrow\int_s^t|X_u-P_CX_u|^2\,du=\int_s^t2V(X_u)\,du .
\]

\emph{The stochastic integral.} Put $M^n_t=\int_0^t\ip{\nabla V_n(X_u)}{dW_u}$ and $M_t=\int_0^t\ip{X_u-P_CX_u}{dW_u}$. Let $\tau_k=\inf\{u\ge0:|X_u|\ge k\}$ be the exit time of $X$ from the ball of radius $k$. For $u<\tau_k$ we have $|X_u|<k$. So $d_C(X_u)\le|X_u|+d_C(0)\le k+d_C(0)$. By the It\^o isometry,
\[
\E\big|M^n_{t\wedge\tau_k}-M_{t\wedge\tau_k}\big|^2=\E\int_0^{t\wedge\tau_k}\big|\nabla V_n(X_u)-(X_u-P_CX_u)\big|^2\,du .
\]
The function under the integral sign is at most $(2d_C(X_u)+1)^2\le(2k+2d_C(0)+1)^2$. It tends to zero for almost every $u$, almost surely. By dominated convergence the right side tends to zero as $n\to\infty$. So the stochastic integrals, stopped at $\tau_k$, converge in $L^2$. By the diagonal procedure there is a subsequence along which, almost surely, $M^n_{t\wedge\tau_k}\to M_{t\wedge\tau_k}$ for all rational $t\ge0$ and all $k$. Since the paths of $X$ are continuous, $\tau_k\to\infty$ as $k\to\infty$. Hence $M^n_t\to M_t$ along this subsequence for all rational $t\ge0$, almost surely.

Now we pass to the limit in \eqref{eq:Vnito} along this subsequence. We get, almost surely for all rational $s\le t$,
\[
V(X_t)-V(X_s)\le\int_s^t\big(Td-2V(X_u)\big)\,du+\sqrt{2T}\int_s^t\ip{X_u-P_CX_u}{dW_u} .
\]
Both sides are continuous in $s$ and $t$, so the inequality holds for all $s\le t$. Define $A$ by \eqref{eq:Vito}, that is,
\[
A_t=V(x)-V(X_t)+\int_0^t\big(Td-2V(X_u)\big)\,du+\sqrt{2T}\int_0^t\ip{X_u-P_CX_u}{dW_u} .
\]
Then the last inequality says that $A_t-A_s\ge0$ for $s\le t$, that is, $A$ is increasing. It is continuous and adapted, and $A_0=0$.

(ii) \emph{Step 1: an equation in dimension one.} Put $\mathcal V_t=V(X_t)$. Let $B'$ be a real Brownian motion independent of $W$ and of $X_0$, defined on an extension of the probability space. We work with the filtration generated by $W$ and $B'$. For this filtration both $W$ and $B'$ are Brownian motions. The processes $X$ and $A$ are adapted to it. Put
\[
\beta_t=\int_0^t\one_{\{\mathcal V_u>0\}}\,\frac{\ip{X_u-P_CX_u}{dW_u}}{d_C(X_u)}+\int_0^t\one_{\{\mathcal V_u=0\}}\,dB'_u .
\]
Then $\beta$ is a continuous local martingale. We compute its quadratic variation. The first integral is a stochastic integral with respect to $W$ of the vector $\one_{\{\mathcal V_u>0\}}(X_u-P_CX_u)/d_C(X_u)$. The norm of this vector is $\one_{\{\mathcal V_u>0\}}$, because $|X_u-P_CX_u|=d_C(X_u)$. The two integrals have zero covariation because $W$ and $B'$ are independent. Hence the quadratic variation of $\beta$ at time $t$ is
\[
\int_0^t\one_{\{\mathcal V_u>0\}}\,du+\int_0^t\one_{\{\mathcal V_u=0\}}\,du=t .
\]
A continuous local martingale which starts at zero and whose quadratic variation at every time $t$ is equal to $t$ is a Brownian motion \cite[Chapter IV, Section 3]{RevuzYor1999}. Hence $\beta$ is a real Brownian motion. On the set $\{\mathcal V_u=0\}$ we have $X_u-P_CX_u=0$ and $d_C(X_u)=0$. Therefore
\[
d_C(X_u)\,d\beta_u=\one_{\{\mathcal V_u>0\}}\ip{X_u-P_CX_u}{dW_u}=\ip{X_u-P_CX_u}{dW_u}.
\]
Since $d_C(X_u)=\sqrt{2\mathcal V_u}$, this gives
\[
\sqrt{2T}\,\ip{X_u-P_CX_u}{dW_u}=\sqrt{2T}\,\sqrt{2\mathcal V_u}\,d\beta_u=2\sqrt{T\mathcal V_u}\,d\beta_u ,
\]
and \eqref{eq:Vito} becomes
\begin{equation}\label{eq:cirV}
\mathcal V_t=V(x)+\int_0^t\big(Td-2\mathcal V_u\big)\,du+2\sqrt T\int_0^t\sqrt{\mathcal V_u}\,d\beta_u-A_t .
\end{equation}

\emph{Step 2: the equation without $A$.} Consider the equation
\begin{equation}\label{eq:cirY}
Y_t=V(x)+\int_0^t\big(Td-2Y_u\big)\,du+2\sqrt T\int_0^t\sqrt{Y_u}\,d\beta_u
\end{equation}
where $\sqrt{Y_u}$ is read as $\sqrt{(Y_u)_+}$. It has a unique strong solution. Indeed, a weak solution exists because the coefficients are continuous with linear growth \cite[Chapter IV]{IkedaWatanabe1989}. Pathwise uniqueness holds because the drift is Lipschitz and the diffusion coefficient is H\"older of order $\frac12$ \cite[Chapter IX, Section 3]{RevuzYor1999}. The H\"older property is the inequality
\[
\big|\sqrt{a_+}-\sqrt{b_+}\big|^2\le\big|\sqrt{a_+}-\sqrt{b_+}\big|\,\big(\sqrt{a_+}+\sqrt{b_+}\big)=|a_+-b_+|\le|a-b|\qquad(a,b\in\R).
\]
Let $Y$ be the solution of \eqref{eq:cirY} driven by the Brownian motion $\beta$ of Step 1.

Next we show that the solution is nonnegative. Suppose that $Y_{t_1}<0$ for some $t_1>0$. Since $Y_0=V(x)\ge0$, there is a last time $t_2<t_1$ with $Y_{t_2}=0$, and $Y<0$ on $(t_2,t_1]$. On this interval $\sqrt{(Y_u)_+}=0$. So the quadratic variation of the stochastic integral in \eqref{eq:cirY} is constant on this interval. A continuous local martingale is constant on every interval on which its quadratic variation is constant \cite[Chapter IV, Section 1]{RevuzYor1999}. Hence the stochastic integral in \eqref{eq:cirY} does not change on this interval. The equation reduces to $dY_t=(Td-2Y_t)\,dt$ with a positive right side, and
\[
Y_{t_1}=Y_{t_2}+\int_{t_2}^{t_1}\big(Td-2Y_u\big)\,du>0 .
\]
This is impossible.

The law of $Y$ is the law of $\frac12|\xi|^2$, where $\xi_t=e^{-t}v+\zeta_t$ is the OU process started at $v$. Indeed, $d\xi_t=-e^{-t}v\,dt-\zeta_t\,dt+\sqrt{2T}\,dW_t=-\xi_t\,dt+\sqrt{2T}\,dW_t$. Each of the $d$ coordinates of $\xi$ has quadratic variation $2Tt$. So the It\^o formula for the function $\frac12|\cdot|^2$ gives
\begin{align*}
d\big(\tfrac12|\xi_t|^2\big)&=\ip{\xi_t}{d\xi_t}+\tfrac12\cdot2T\,d\,dt=\ip{\xi_t}{-\xi_t\,dt+\sqrt{2T}\,dW_t}+Td\,dt\\
&=\big(Td-|\xi_t|^2\big)\,dt+\sqrt{2T}\,\ip{\xi_t}{dW_t}.
\end{align*}
Put $\hat\beta_t=\int_0^t\one_{\{\xi_u\ne0\}}\ip{\xi_u}{dW_u}/|\xi_u|$. For $u>0$ the vector $\xi_u$ is Gaussian with the covariance matrix $T(1-e^{-2u})$ times the identity, by \eqref{eq:oucov}. So $\P(\xi_u=0)=0$. As above, it follows that the set of times with $\xi_u=0$ has Lebesgue measure zero almost surely. Hence the quadratic variation of $\hat\beta$ at time $t$ is $\int_0^t\one_{\{\xi_u\ne0\}}du=t$. So $\hat\beta$ is a real Brownian motion. Also $\ip{\xi_t}{dW_t}=|\xi_t|\,d\hat\beta_t$, and we get
\[
d\big(\tfrac12|\xi_t|^2\big)=\big(Td-|\xi_t|^2\big)\,dt+\sqrt{2T}\,|\xi_t|\,d\hat\beta_t .
\]
Write $Z_t=\frac12|\xi_t|^2$. Then $|\xi_t|^2=2Z_t$ and $\sqrt{2T}\,|\xi_t|=\sqrt{2T}\sqrt{2Z_t}=2\sqrt T\sqrt{Z_t}$. So
\[
Z_t=\tfrac12|v|^2+\int_0^t\big(Td-2Z_u\big)\,du+2\sqrt T\int_0^t\sqrt{Z_u}\,d\hat\beta_u .
\]
Thus $\frac12|\xi|^2$ is a solution of \eqref{eq:cirY} driven by $\hat\beta$ with the initial value $\frac12|v|^2=\frac12d_C(x)^2=V(x)$. Pathwise uniqueness implies that solutions of \eqref{eq:cirY} are unique in law. Hence $Y$ and $\frac12|\xi|^2$ have the same law.

\emph{Step 3: comparison.} We show that $\mathcal V_t\le Y_t$ for all $t$, almost surely. This is the comparison argument of \cite[Chapter VI]{IkedaWatanabe1989}. We give the details because of the term $A$.

First we construct some auxiliary functions. There are numbers $1=a_0>a_1>\dots\to0$ and continuous functions $\kappa_n\ge0$ which vanish outside $(a_n,a_{n-1})$ with
\[
\kappa_n(z)\le\frac{2}{nz},\qquad\int_{a_n}^{a_{n-1}}\kappa_n(z)\,dz=1 .
\]
For example, put $a_n=e^{-n(n+1)/2}$, so that $a_{n-1}/a_n=e^{n}$ and
\[
\int_{a_n}^{a_{n-1}}\frac{dz}{z}=\log\frac{a_{n-1}}{a_n}=n .
\]
Let $\chi_n$ be a continuous function on $\R$ with $0\le\chi_n\le1$ which vanishes outside $(a_n,a_{n-1})$ and is equal to $1$ on the interval $[a_ne^{n/4},a_{n-1}e^{-n/4}]$. Then
\[
J_n:=\int_{a_n}^{a_{n-1}}\chi_n(z)\,\frac{2}{nz}\,dz\ge\frac2n\int_{a_ne^{n/4}}^{a_{n-1}e^{-n/4}}\frac{dz}{z}=\frac2n\Big(n-\frac n2\Big)=1 ,
\]
and the function $\kappa_n(z)=\chi_n(z)\,2/(nzJ_n)$ for $z>0$, $\kappa_n(z)=0$ for $z\le0$, has the required properties.

Put $\phi_n(z)=\int_0^{z}\int_0^{y}\kappa_n(u)\,du\,dy$ for $z\ge0$ and $\phi_n(z)=0$ for $z<0$. Then $\phi_n$ is of class $C^2$, $0\le\phi_n'\le1$, $\phi_n'=0$ on $(-\infty,a_n]$, $\phi_n''=\kappa_n$ and $\phi_n(z)\to z_+$ as $n\to\infty$. We verify these properties. For $z\ge0$ we have $\phi_n'(z)=\int_0^z\kappa_n(u)\,du$ and $\phi_n''(z)=\kappa_n(z)$. Since $\kappa_n=0$ on $[0,a_n]$, the function $\phi_n$ vanishes on $(-\infty,a_n]$. So it is of class $C^2$ on $\R$ with $\phi_n''=\kappa_n$. From $\kappa_n\ge0$ and $\int\kappa_n=1$ we get $0\le\phi_n'\le1$. Also $\phi_n'(z)=1$ for $z\ge a_{n-1}$. Hence $\phi_n(z)\le z_+$ for all $z$, and $\phi_n(z)\ge\int_{a_{n-1}}^z1\,dy=z-a_{n-1}$ for $z\ge a_{n-1}$. In all cases
\begin{equation}\label{eq:xs2c}
z_+-a_{n-1}\le\phi_n(z)\le z_+\qquad(z\in\R),
\end{equation}
and $a_{n-1}\to0$.

Let $\Gamma=\mathcal V-Y$. We subtract \eqref{eq:cirY} from \eqref{eq:cirV}. The terms $V(x)$ and $Td$ cancel and
\[
\Gamma_t=-2\int_0^t\Gamma_u\,du+2\sqrt T\int_0^t\big(\sqrt{\mathcal V_u}-\sqrt{Y_u}\big)\,d\beta_u-A_t .
\]
So $\Gamma$ is a continuous semimartingale with $\Gamma_0=0$. Its quadratic variation is $4T\int_0^t(\sqrt{\mathcal V_u}-\sqrt{Y_u})^2du$. The It\^o formula for the function $\phi_n$ of class $C^2$ reads
\[
\phi_n(\Gamma_t)=\phi_n(0)+\int_0^t\phi_n'(\Gamma_u)\,d\Gamma_u+\frac12\int_0^t\phi_n''(\Gamma_u)\,4T\big(\sqrt{\mathcal V_u}-\sqrt{Y_u}\big)^2\,du .
\]
Since $\phi_n(0)=0$, this is
\begin{align*}
\phi_n(\Gamma_t)&=-2\int_0^t\phi_n'(\Gamma_u)\,\Gamma_u\,du-\int_0^t\phi_n'(\Gamma_u)\,dA_u+2\sqrt T\int_0^t\phi_n'(\Gamma_u)\big(\sqrt{\mathcal V_u}-\sqrt{Y_u}\big)\,d\beta_u\\
&\quad+2T\int_0^t\phi_n''(\Gamma_u)\big(\sqrt{\mathcal V_u}-\sqrt{Y_u}\big)^2\,du .
\end{align*}
We consider the four terms on the right.

The first term is nonpositive because $\phi_n'(z)z\ge0$ for all $z$. Indeed, $\phi_n'(z)=0$ for $z\le0$ and $\phi_n'(z)\ge0$ for $z>0$.

Since $\phi_n'\ge0$ and $A$ is increasing, the second term is nonpositive.

For the third term, let $\sigma_k$ be the first time at which $\mathcal V+Y$ exceeds $k$. For $u<\sigma_k$ we have $|\phi_n'(\Gamma_u)|\le1$ and $(\sqrt{\mathcal V_u}-\sqrt{Y_u})^2\le\mathcal V_u+Y_u\le k$. So the stochastic integral stopped at $\sigma_k$ is a martingale which starts at zero. Its expectation at the time $t\wedge\sigma_k$ is therefore zero.

In the last term, the inequality of Step 2 gives $(\sqrt{\mathcal V_u}-\sqrt{Y_u})^2\le|\mathcal V_u-Y_u|=|\Gamma_u|$. Also $\phi_n''(z)|z|=\kappa_n(z)\,z\le2/n$ for $z>0$ and $\phi_n''(z)=0$ for $z\le0$. Hence
\[
2T\int_0^t\phi_n''(\Gamma_u)\big(\sqrt{\mathcal V_u}-\sqrt{Y_u}\big)^2\,du\le2T\int_0^t\phi_n''(\Gamma_u)\,|\Gamma_u|\,du\le2T\cdot\frac2n\cdot t=\frac{4Tt}{n} ,
\]
so this term is at most $4Tt/n$.

We take expectations at the time $t\wedge\sigma_k$. The stochastic integral disappears and
\[
\E\,\phi_n(\Gamma_{t\wedge\sigma_k})\le\frac{4Tt}{n} .
\]
By \eqref{eq:xs2c},
\[
\E\,(\Gamma_{t\wedge\sigma_k})_+\le\E\,\phi_n(\Gamma_{t\wedge\sigma_k})+a_{n-1}\le\frac{4Tt}{n}+a_{n-1} .
\]
Letting $n\to\infty$ gives $\E(\Gamma_{t\wedge\sigma_k})_+=0$, that is, $\Gamma_{t\wedge\sigma_k}\le0$ almost surely. The process $\mathcal V+Y$ is continuous, so $\sigma_k\to\infty$ as $k\to\infty$. Letting $k\to\infty$ we get $\Gamma_t\le0$ almost surely. Hence $\Gamma_t\le0$ for all rational $t$, almost surely. By continuity this holds for all $t$ simultaneously.

\emph{Step 4.} Put $\varrho_t=\sqrt{2Y_t}$. Then $d_C(X_t)=\sqrt{2\mathcal V_t}\le\varrho_t$ by Step 3. By Step 2 the process $\varrho$ has the law of $|\xi|=(|e^{-t}v+\zeta_t|)_{t\ge0}$.

(iii) By (ii), $\E\,d_C(X_t)^2\le\E\,\varrho_t^2=\E|e^{-t}v+\zeta_t|^2$. Since $\zeta_t$ is centred and $|v|=d_C(x)$,
\[
\E|e^{-t}v+\zeta_t|^2=e^{-2t}|v|^2+2e^{-t}\ip{v}{\E\,\zeta_t}+\E|\zeta_t|^2=e^{-2t}d_C(x)^2+\E|\zeta_t|^2 .
\]
By \eqref{eq:oucov} the covariance matrix of $\zeta_t$ is $T(1-e^{-2t})$ times the identity matrix of order $d$. The quantity $\E|\zeta_t|^2$ is its trace, so $\E|\zeta_t|^2=dT(1-e^{-2t})$. This proves the first inequality of \eqref{eq:msq}. For the two other inequalities we note that $e^{-t}\le1$ gives
\[
|e^{-t}v+\zeta_t|\le e^{-t}|v|+|\zeta_t|\le d_C(x)+|\zeta_t| .
\]
Since $d_C(X_t)\le\varrho_t$ and $\varrho$ has the law of $(|e^{-t}v+\zeta_t|)_{t\ge0}$,
\[
\P\big(d_C(X_t)\ge r\big)\le\P\big(\varrho_t\ge r\big)=\P\big(|e^{-t}v+\zeta_t|\ge r\big)\le\P\big(d_C(x)+|\zeta_t|\ge r\big),
\]
which is the second inequality. In the same way
\begin{align*}
\P\Big(\sup_{0\le t\le N}d_C(X_t)\ge r\Big)&\le\P\Big(\sup_{0\le t\le N}\varrho_t\ge r\Big)=\P\Big(\sup_{0\le t\le N}|e^{-t}v+\zeta_t|\ge r\Big)\\
&\le\P\Big(d_C(x)+\sup_{0\le t\le N}|\zeta_t|\ge r\Big),
\end{align*}
which is the third inequality. The theorem is proved.
\end{proof}

\begin{remark}
The first inequality in \eqref{eq:msq} is an equality when $C$ is a point. Theorems \ref{thm:converse} and \ref{thm:compare} do not contradict each other. In Theorem \ref{thm:compare} the dominating process is driven by the real Brownian motion $\beta$. While $X$ is outside $C$, the increments of $\beta$ are the projections of the increments of $W$ on the direction $X_u-P_CX_u$. While $X$ is in $C$, they are the increments of an independent Brownian motion. So the dominating process is not of the form $e^{-t}d_C(x)+|\zeta_t|$ with the envelope process $\zeta$ of the state.
\end{remark}

\subsection{A maximal inequality for the envelope process}
To use \eqref{eq:env3} we need a bound for the supremum of $|\zeta_t|$ over a long time interval, with the correct exponential rate. The following estimate is elementary and has explicit constants. We write $a_+=\max\{a,0\}$ and $\log_+a=(\log a)_+$.

\begin{proposition}\label{prop:oumax}
Let $W$ be a Wiener process in $H$ with covariance $Q\neq0$, let $T>0$ and let $\zeta$ be the envelope process of \eqref{eq:SDE}. Then for all $N>0$, $r>0$ and $\eta\in(0,1)$,
\begin{equation}\label{eq:oumax}
\begin{aligned}
\P\Big(\sup_{0\le t\le N}|\zeta_t|\ge r\Big)\le\Big(\frac{16N}{\eta^2}+1\Big)\bigg[&\exp\Big(-\frac{\big((1-\eta)r-\sqrt{T\tr Q}\big)_+^2}{2T\|Q\|}\Big)\\
&+\exp\Big(\frac{\tr Q}{2\|Q\|}-\frac{r^2}{2T\|Q\|}\Big)\bigg].
\end{aligned}
\end{equation}
\end{proposition}

\begin{proof}
\emph{Step 1: increments.} Let $0\le a\le t$. We claim that
\begin{equation}\label{eq:ouincr}
\zeta_t=e^{-(t-a)}\zeta_a+\sqrt{2T}\Big[(W_t-W_a)-\int_a^te^{-(t-u)}(W_u-W_a)\,du\Big].
\end{equation}
Indeed, denote the right side by $\xi_t$ and put
\[
\tilde w(t)=\sqrt{2T}(W_t-W_a),\qquad y(t)=\int_a^te^{-(t-u)}\tilde w(u)\,du ,
\]
so that $\xi_t=e^{-(t-a)}\zeta_a+\tilde w(t)-y(t)$. Differentiating the integral we get
\[
y'(t)=\tilde w(t)-\int_a^te^{-(t-u)}\tilde w(u)\,du=\tilde w(t)-y(t),\qquad y(a)=0 .
\]
Since $\int_a^te^{-(u-a)}du=1-e^{-(t-a)}$ and $\tilde w-y=y'$, we have
\[
\int_a^t\xi_u\,du=(1-e^{-(t-a)})\zeta_a+\int_a^ty'(u)\,du=(1-e^{-(t-a)})\zeta_a+y(t),
\]
and so
\[
\zeta_a-\int_a^t\xi_u\,du+\tilde w(t)=e^{-(t-a)}\zeta_a-y(t)+\tilde w(t)=\xi_t .
\]
On the other hand, by \eqref{eq:ou} with $w=\sqrt{2T}\,W$, taken at the times $t$ and $a$,
\[
\zeta_t=\zeta_a-\int_a^t\zeta_u\,du+\tilde w(t)\qquad(t\ge a).
\]
Thus $\xi$ solves on $[a,\infty)$ the same linear equation as $\zeta$ with the same value at $a$. The difference $D=\xi-\zeta$ satisfies $D_t=-\int_a^tD_u\,du$ for $t\ge a$. So $D$ is differentiable with $D'=-D$ and $D_a=0$. Therefore $D=0$. Hence $\xi=\zeta$, which is \eqref{eq:ouincr}.

Now let $h>0$ and $a\le t\le a+h$. From \eqref{eq:ouincr}, since $e^{-(t-a)}\le1$ and $|W_u-W_a|\le\sup_{a\le u\le a+h}|W_u-W_a|$ for $a\le u\le t$,
\begin{equation}\label{eq:ouincr2}
\begin{aligned}
|\zeta_t|&\le|\zeta_a|+\sqrt{2T}\Big(1+\int_a^te^{-(t-u)}du\Big)\sup_{a\le u\le a+h}|W_u-W_a|\\
&\le|\zeta_a|+2\sqrt{2T}\sup_{a\le u\le a+h}|W_u-W_a| .
\end{aligned}
\end{equation}
In the second line we used $\int_a^te^{-(t-u)}du=1-e^{-(t-a)}\le1$.

\emph{Step 2: the value at a fixed time.} By \eqref{eq:oucov}, $\zeta_a$ is a centred Gaussian vector with the covariance operator $s_a^2Q$, where $s_a^2=T(1-e^{-2a})\le T$. So its covariance operator is at most $TQ$. Hence, by the Schwarz inequality,
\[
\E|\zeta_a|\le\big(\E|\zeta_a|^2\big)^{1/2}=\big(s_a^2\tr Q\big)^{1/2}\le\sqrt{T\tr Q},
\]
and $\E\ip{\zeta_a}{e}^2=s_a^2\ip{Qe}{e}\le T\|Q\|$ for every unit vector $e$. Thus $|\zeta_a|$ is a Lipschitz function, with constant at most $(T\|Q\|)^{1/2}$, of a sequence of independent standard Gaussian variables. We write this out. For $a=0$ we have $\zeta_0=0$, so let $a>0$. Let $q_1\ge q_2\ge\dots$ be the positive eigenvalues of $Q$, repeated according to multiplicity, and let $e_1,e_2,\dots$ be orthonormal eigenvectors, $Qe_i=q_ie_i$. Then $\|Q\|=q_1$. The variables $g_i=\ip{\zeta_a}{e_i}/(s_a\sqrt{q_i})$ are jointly Gaussian and centred with $\E\,g_ig_j=\ip{Qe_i}{e_j}/\sqrt{q_iq_j}$, which is $1$ for $i=j$ and $0$ for $i\ne j$. So they are independent standard Gaussian variables. Moreover $|\zeta_a|^2=s_a^2\sum_iq_ig_i^2$ almost surely. For $m\ge1$ put
\[
F_m(g_1,\dots,g_m)=\Big(s_a^2\sum_{i\le m}q_ig_i^2\Big)^{1/2},
\]
which is the norm of the projection of $\zeta_a$ on the span of $e_1,\dots,e_m$. By the triangle inequality in $\R^m$, for $g,g'\in\R^m$,
\[
|F_m(g)-F_m(g')|\le\Big(s_a^2\sum_{i\le m}q_i(g_i-g_i')^2\Big)^{1/2}\le\big(T\|Q\|\big)^{1/2}\,|g-g'| .
\]
Also $\E\,F_m\le\E|\zeta_a|\le\sqrt{T\tr Q}$. We apply the Gaussian concentration inequality \cite[Theorem 5.6]{BLM2013} to these projections of $\zeta_a$ on the span of finitely many eigenvectors of $Q$. For every $\ell'\ge0$ it gives
\[
\P\big(F_m\ge\sqrt{T\tr Q}+\ell'\big)\le\P\big(F_m\ge\E\,F_m+\ell'\big)\le\exp\Big(-\frac{\ell'^2}{2T\|Q\|}\Big).
\]
Now we pass to the limit. The sequence $F_m$ increases to $|\zeta_a|$ as $m\to\infty$. Let $0\le\ell'<\ell$. If $|\zeta_a|\ge\sqrt{T\tr Q}+\ell$, then $F_m\ge\sqrt{T\tr Q}+\ell'$ for all large $m$. Hence
\[
\P\big(|\zeta_a|\ge\sqrt{T\tr Q}+\ell\big)\le\lim_{m\to\infty}\P\big(F_m\ge\sqrt{T\tr Q}+\ell'\big)\le\exp\Big(-\frac{\ell'^2}{2T\|Q\|}\Big).
\]
Letting $\ell'\to\ell$ we get
\begin{equation}\label{eq:gaussconc}
\P\big(|\zeta_a|\ge\sqrt{T\tr Q}+\ell\big)\le\exp\Big(-\frac{\ell^2}{2T\|Q\|}\Big)\qquad(\ell\ge0).
\end{equation}
For $\ell=0$ the bound is trivial.

\emph{Step 3: the supremum of the Wiener process on a short interval.} Let $h>0$, $\lambda>0$ and $\alpha=1/(4h\|Q\|)$. Let $q_1\ge q_2\ge\dots\ge0$ be the eigenvalues of $Q$, so that $|W_h|^2$ has the law of $h\sum_iq_ig_i^2$ with independent standard Gaussian variables $g_i$. For a standard Gaussian variable $g$ and a number $\theta<\frac12$ we have
\begin{equation}\label{eq:xs2d}
\E\,e^{\theta g^2}=\frac1{\sqrt{2\pi}}\int_{-\infty}^\infty e^{\theta z^2-z^2/2}\,dz=\frac1{\sqrt{2\pi}}\int_{-\infty}^\infty e^{-(1-2\theta)z^2/2}\,dz=(1-2\theta)^{-1/2}.
\end{equation}
We use this with $\theta=\alpha hq_i$. Here $2\alpha hq_i=q_i/(2\|Q\|)\le\frac12$. We also use the inequality $-\log(1-z)\le2z$ for $0\le z\le\frac12$. It holds because the function $2z+\log(1-z)$ vanishes at $z=0$ and has the derivative $2-1/(1-z)\ge0$ on $[0,\frac12]$. By independence and monotone convergence,
\begin{align*}
\E\,e^{\alpha|W_h|^2}&=\prod_i\big(1-2\alpha hq_i\big)^{-1/2}=\exp\Big(-\frac12\sum_i\log\Big(1-\frac{q_i}{2\|Q\|}\Big)\Big)\\
&\le\exp\Big(\sum_i\frac{q_i}{2\|Q\|}\Big)=\exp\Big(\frac{\tr Q}{2\|Q\|}\Big).
\end{align*}
The same computation with $u\le h$ in place of $h$ shows that $\E\,e^{\alpha|W_u|^2}$ is finite for $0\le u\le h$.

The process $|W_u|^2$ is a submartingale. Indeed, for $s\le u$ we have $|W_u|^2=|W_s|^2+2\ip{W_s}{W_u-W_s}+|W_u-W_s|^2$. The increment $W_u-W_s$ is centred and independent of $(W_r)_{r\le s}$. So the conditional expectation of $|W_u|^2$ given $(W_r)_{r\le s}$ is $|W_s|^2+(u-s)\tr Q\ge|W_s|^2$. Since the function $z\mapsto e^{\alpha z}$ is convex and increasing, the Jensen inequality for conditional expectations shows that $e^{\alpha|W_u|^2}$ is a positive submartingale on $[0,h]$. The Doob maximal inequality for a positive submartingale $(M_u)_{u\le h}$ with continuous paths says that $\P(\sup_{u\le h}M_u\ge b)\le\E\,M_h/b$ for $b>0$. With $M_u=e^{\alpha|W_u|^2}$ and $b=e^{\alpha\lambda^2}$ we get
\begin{equation}\label{eq:wsup}
\begin{aligned}
\P\Big(\sup_{0\le u\le h}|W_u|\ge\lambda\Big)&=\P\Big(\sup_{0\le u\le h}e^{\alpha|W_u|^2}\ge e^{\alpha\lambda^2}\Big)\le e^{-\alpha\lambda^2}\,\E\,e^{\alpha|W_h|^2}\\
&\le\exp\Big(\frac{\tr Q}{2\|Q\|}-\frac{\lambda^2}{4h\|Q\|}\Big).
\end{aligned}
\end{equation}

\emph{Step 4: union bound.} Let $h=\eta^2/16$ and cover $[0,N]$ by the intervals $[a_n,a_n+h]$, $a_n=nh$, $0\le n<N/h$. Their number is at most $N/h+1=16N/\eta^2+1$. Fix $n$ and suppose that $\sup_{a_n\le t\le a_n+h}|\zeta_t|\ge r$. Then by \eqref{eq:ouincr2} with $a=a_n$,
\[
|\zeta_{a_n}|+2\sqrt{2T}\sup_{a_n\le u\le a_n+h}|W_u-W_{a_n}|\ge r=(1-\eta)r+\eta r ,
\]
and so at least one of the inequalities
\[
|\zeta_{a_n}|\ge(1-\eta)r,\qquad2\sqrt{2T}\sup_{a_n\le u\le a_n+h}|W_u-W_{a_n}|\ge\eta r
\]
holds. The probability of the first one is bounded by \eqref{eq:gaussconc} with $\ell=((1-\eta)r-\sqrt{T\tr Q})_+$. Indeed, if $\ell>0$ then $(1-\eta)r=\sqrt{T\tr Q}+\ell$. When $\ell=0$ the bound is trivial. For the second one we use \eqref{eq:wsup} with $\lambda=\eta r/(2\sqrt{2T})$. This is possible because $(W_{a_n+u}-W_{a_n})_{u\ge0}$ is again a Wiener process with covariance $Q$. Here $\lambda^2=\eta^2r^2/(8T)$ and $4h=\eta^2/4$, so that
\[
\frac{\lambda^2}{4h\|Q\|}=\frac{\eta^2r^2}{8T}\cdot\frac{16}{4\eta^2\|Q\|}=\frac{r^2}{2T\|Q\|} .
\]
Thus, for each $n$,
\[
\P\Big(\sup_{a_n\le t\le a_n+h}|\zeta_t|\ge r\Big)\le\exp\Big(-\frac{\big((1-\eta)r-\sqrt{T\tr Q}\big)_+^2}{2T\|Q\|}\Big)+\exp\Big(\frac{\tr Q}{2\|Q\|}-\frac{r^2}{2T\|Q\|}\Big).
\]
If $\sup_{0\le t\le N}|\zeta_t|\ge r$, then the supremum over at least one of the intervals $[a_n,a_n+h]$ is at least $r$. Summing over $n$ gives \eqref{eq:oumax}.
\end{proof}

In $\R^d$ with $Q$ the identity we have $\tr Q=d$ and $\|Q\|=1$. We shall also need a bound for the second moment of the supremum.

\begin{corollary}\label{cor:oumoment}
Let $H=\R^d$ and let $Q$ be the identity. Then for all $T>0$ and $N>0$
\[
\E\sup_{0\le t\le N}|\zeta_t|^2\le120\,T\,\big(d+\log(1+N)\big).
\]
\end{corollary}

\begin{proof}
Let $n_0$ be the smallest integer with $n_0\ge N$. Since $N>0$, we have $1\le n_0\le N+1$. For $0\le n<n_0$ we put $S_n=\sup_{n\le u\le n+1}|W_u-W_n|$. Every $t\in[0,N]$ lies in an interval $[n,n+1]$ with $0\le n<n_0$. For such $t$ the inequality \eqref{eq:ouincr2} with $a=n$ and $h=1$, divided by $\sqrt T$, gives $|\zeta_t|/\sqrt T\le|\zeta_n|/\sqrt T+2\sqrt2\,S_n$. By this and by the inequality $(z_1+z_2)^2\le2z_1^2+2z_2^2$, where $2\,(2\sqrt2)^2=16$,
\begin{equation}\label{eq:supsplit}
\sup_{0\le t\le N}\frac{|\zeta_t|^2}{T}\le\max_{n<n_0}\Big(\frac{|\zeta_n|}{\sqrt T}+2\sqrt2\,S_n\Big)^2\le2\max_{n<n_0}\frac{|\zeta_n|^2}{T}+16\max_{n<n_0}S_n^2 .
\end{equation}
For nonnegative random variables $\xi_n$ the Jensen inequality for the concave function $\log$ gives
\[
\E\max_n\xi_n=4\,\E\log\max_ne^{\xi_n/4}\le4\log\E\max_ne^{\xi_n/4}\le4\log\sum_n\E\,e^{\xi_n/4}.
\]
In the last step we used that a maximum of positive numbers is at most their sum.

We apply this first to $\xi_n=|\zeta_n|^2/T$. By \eqref{eq:oucov}, $\zeta_n/\sqrt T$ is a centred Gaussian vector in $\R^d$ with covariance $\sigma^2$ times the identity, $\sigma^2=1-e^{-2n}\le1$. So $|\zeta_n|^2/(4T)$ has the law of $(\sigma^2/4)\sum_{i\le d}g_i^2$ with independent standard Gaussian variables $g_i$. By \eqref{eq:xs2d} with $\theta=\sigma^2/4$,
\[
\E\exp\Big(\frac{|\zeta_n|^2}{4T}\Big)=\Big(1-\frac{\sigma^2}2\Big)^{-d/2}\le2^{d/2}.
\]
There are $n_0$ indices $n$. Hence
\[
\E\max_{n<n_0}\frac{|\zeta_n|^2}{T}\le4\log\big(n_0\,2^{d/2}\big)=4\log n_0+2d\log2 .
\]
Next we take $\xi_n=S_n^2$. Each $S_n$ has the law of $\sup_{0\le u\le1}|W_u|$. As in Step 3 of the proof of Proposition \ref{prop:oumax}, $e^{|W_u|^2/8}$ is a positive submartingale on $[0,1]$. The Doob inequality in $L^2$ says that $\E\sup_{u\le1}M_u^2\le4\,\E\,M_1^2$ for such a submartingale $M$. With $M_u=e^{|W_u|^2/8}$, and by \eqref{eq:xs2d} with $\theta=\frac14$, this gives
\[
\E\,e^{S_n^2/4}=\E\sup_{0\le u\le1}\big(e^{|W_u|^2/8}\big)^2\le4\,\E\,e^{|W_1|^2/4}=4\,\big(1-\tfrac12\big)^{-d/2}=4\cdot2^{d/2}.
\]
Hence $\E\max_{n<n_0}S_n^2\le4\log(4n_0\,2^{d/2})=4\log n_0+8\log2+2d\log2$. Taking expectations in \eqref{eq:supsplit},
\begin{align*}
\E\sup_{0\le t\le N}\frac{|\zeta_t|^2}{T}&\le2\,\big(4\log n_0+2d\log2\big)+16\,\big(4\log n_0+8\log2+2d\log2\big)\\
&=72\log n_0+36\,d\log2+128\log2\le72\log(1+N)+25\,d+89 .
\end{align*}
Here we used $n_0\le1+N$, $36\log2\le25$ and $128\log2\le89$. Since $d\ge1$ we have $25d+89\le114\,d$. So the right side is at most $72\log(1+N)+114\,d\le120\,(d+\log(1+N))$. This proves the claim.
\end{proof}

\subsection{Viability over exponentially long times}
Combining Theorem \ref{thm:envelope} and Proposition \ref{prop:oumax} we obtain the first consequence for the state process.

\begin{corollary}\label{cor:viability}
Let $\{C(t)\}$ be a moving convex set of speed $v$ in $H$, let $X$ be the state process of \eqref{eq:SDE} and let $r>0$, $N>0$, $\eta\in(0,1)$. Then, with probability at least one minus the right side of \eqref{eq:oumax},
\[
d_{C(t)}(X_t)<e^{-t}d_{C(0)}(X_0)+v\,(1-e^{-t})+r\qquad\text{for all }t\in[0,N].
\]
In particular, if $X_0\in C(0)$ and $0<c<r^2/(2\|Q\|)$, then
\[
\P\Big(d_{C(t)}(X_t)<v+r\ \text{ for all }t\le e^{c/T}\Big)\longrightarrow1\qquad(T\to0).
\]
\end{corollary}

\begin{proof}
The first statement follows from \eqref{eq:env3} and \eqref{eq:oumax}. Indeed, by \eqref{eq:oumax} the event $\{\sup_{0\le t\le N}|\zeta_t|<r\}$ has probability at least one minus the right side of \eqref{eq:oumax}. On this event \eqref{eq:env3} gives the stated inequality for all $t\in[0,N]$.

For the second statement let $X_0\in C(0)$, so that $d_{C(0)}(X_0)=0$. Since $v\,(1-e^{-t})\le v$, the first statement gives $d_{C(t)}(X_t)<v+r$ for all $t\le N$ with probability at least one minus the right side of \eqref{eq:oumax}. Choose $\eta\in(0,1)$ with $(1-\eta)^2r^2/(2\|Q\|)>c$ and take $N=e^{c/T}$. We have to show that, as $T\to0$,
\begin{align*}
\Big(\frac{16e^{c/T}}{\eta^2}+1\Big)\exp\Big(-\frac{((1-\eta)r-\sqrt{T\tr Q})_+^2}{2T\|Q\|}\Big)&\longrightarrow0,\\
\Big(\frac{16e^{c/T}}{\eta^2}+1\Big)\exp\Big(\frac{\tr Q}{2\|Q\|}-\frac{r^2}{2T\|Q\|}\Big)&\longrightarrow0 .
\end{align*}
Put $\kappa_0=(1-\eta)^2r^2/(2\|Q\|)-c>0$. Since $e^{c/T}\ge1$, the first factor in both products is at most $(16/\eta^2+1)\,e^{c/T}$. For the first product put
\[
\kappa(T)=\frac{((1-\eta)r-\sqrt{T\tr Q})_+^2}{2\|Q\|}-c .
\]
Then $\kappa(T)\to\kappa_0$ as $T\to0$, so $\kappa(T)\ge\kappa_0/2$ for all small $T$, and
\begin{align*}
\Big(\frac{16e^{c/T}}{\eta^2}+1\Big)\exp\Big(-\frac{((1-\eta)r-\sqrt{T\tr Q})_+^2}{2T\|Q\|}\Big)&\le\Big(\frac{16}{\eta^2}+1\Big)e^{-\kappa(T)/T}\\
&\le\Big(\frac{16}{\eta^2}+1\Big)e^{-\kappa_0/(2T)}\longrightarrow0 .
\end{align*}
For the second product we note that $r^2/(2\|Q\|)-c\ge\kappa_0$. Hence
\[
\Big(\frac{16e^{c/T}}{\eta^2}+1\Big)\exp\Big(\frac{\tr Q}{2\|Q\|}-\frac{r^2}{2T\|Q\|}\Big)\le\Big(\frac{16}{\eta^2}+1\Big)\exp\Big(\frac{\tr Q}{2\|Q\|}\Big)\,e^{-\kappa_0/T}\longrightarrow0 .
\]
So both products tend to zero. This completes the proof.
\end{proof}

\begin{corollary}\label{cor:viabilityclosed}
Let $C\subset\R^d$ be nonempty and closed, let $X$ be the solution of \eqref{eq:SDEgen} with $X_0=x\in C$, let $r>0$, and let $0<c<r^2/2$. Then
\[
\P\Big(d_C(X_t)<r\ \text{ for all }t\le e^{c/T}\Big)\longrightarrow1\qquad(T\to0).
\]
\end{corollary}

\begin{proof}
Put $N=e^{c/T}$. The complementary event is $\{\sup_{t\le N}d_C(X_t)\ge r\}$. By the third inequality of \eqref{eq:msq} with $d_C(x)=0$, its probability is at most $\P(\sup_{t\le N}|\zeta_t|\ge r)$. In $\R^d$ we have $\|Q\|=1$, so $c<r^2/(2\|Q\|)$. Hence this probability tends to zero by \eqref{eq:oumax}, as in the proof of Corollary \ref{cor:viability}.
\end{proof}

\begin{remark}
The constant $r^2/(2\|Q\|)$ is the best possible in general. Let $C=\{0\}$ and $X_0=0$, so that $X=\zeta$, and let $e$ be a unit vector with $Qe=\|Q\|e$. Then $\ip{\zeta}{e}$ is a real OU process, $d\ip\zeta e=-\ip\zeta e\,dt+\sqrt{2T\|Q\|}\,dB$ with a real Brownian motion $B$. Its exit time from $(-r,r)$ is $\exp((1+o(1))r^2/(2\|Q\|T))$ in probability \cite{FreidlinWentzell2012}. Also $|\zeta|\ge|\ip\zeta e|$. The bound of Corollary \ref{cor:viability} depends on the dimension only through $\tr Q$ and on the feasible sets only through their speed. This is useful when the state space is a space of high dimension and nothing is known about the shape of $C(t)$ except convexity.
\end{remark}

\section{The Gibbs state and the intrinsic volumes}\label{sec:gibbs}

In this section $C\subset\R^d$ is a nonempty compact convex set. We denote by $k=\dim C\in\{0,1,\dots,d\}$ the dimension of its affine hull, and by $\cH^k$ the Hausdorff measure of dimension $k$. We write
\[
u_C=\frac{\cH^k(C\cap\,\cdot\,)}{\cH^k(C)}
\]
for the uniform probability measure on $C$. When $k=0$ this is a Dirac mass. We begin with an example which shows the structure of the general result.

\begin{example}\label{ex:square}
Let $C=[0,1]^2\subset\R^2$. The plane is the union of $C$, of four half strips over the sides and of four quadrants at the vertices. On a half strip, $d_C$ is the distance to the side. On a quadrant it is the distance to the vertex. The common part of two of these nine pieces is contained in a line, which has area zero. So we can compute $\cZ_T(C)=\int_{\R^2}e^{-d_C(x)^2/2T}\,dx$ piece by piece.

On $C$ we have $d_C=0$. The integrand is equal to $1$ there. Hence the integral over $C$ is the area of $C$, which is $1$.

Next take the half strip $S=\{0\le x_1\le1,\ x_2<0\}$ over the lower side. For $x\in S$ we have $P_Cx=(x_1,0)$ and $d_C(x)=-x_2$. We put $r=-x_2$ and get
\[
\int_Se^{-d_C(x)^2/2T}\,dx=\int_0^1dx_1\int_0^\infty e^{-r^2/2T}\,dr=\int_0^\infty e^{-r^2/2T}\,dr ,
\]
and
\[
\int_0^\infty e^{-r^2/2T}\,dr=\frac12\int_{-\infty}^{\infty}e^{-r^2/2T}\,dr=\frac12\sqrt{2\pi T}.
\]
Here we used that the function $e^{-r^2/2T}$ is even and that the Gaussian density with mean zero and variance $T$ has integral one, that is, $\int_{-\infty}^{\infty}e^{-r^2/2T}\,dr=\sqrt{2\pi T}$. The other three half strips give the same value.

Now take the quadrant $Q=\{x_1<0,\ x_2<0\}$ at the vertex $(0,0)$. For $x\in Q$ we have $P_Cx=(0,0)$ and $d_C(x)=|x|$. The function $e^{-|x|^2/2T}$ does not change when the sign of a coordinate is changed, so it has the same integral over each of the four quadrants of the plane. Hence
\[
\int_Qe^{-d_C(x)^2/2T}\,dx=\frac14\int_{\R^2}e^{-|x|^2/2T}\,dx=\frac14\Big(\int_{-\infty}^{\infty}e^{-r^2/2T}\,dr\Big)^2=\frac14\cdot2\pi T .
\]
After a translation which moves the vertex to the origin, the same calculation applies at the other three vertices. Adding the nine integrals we obtain
\[
\cZ_T(C)=1+4\int_0^\infty e^{-r^2/2T}\,dr+4\cdot\frac14\int_{\R^2}e^{-|x|^2/2T}\,dx=1+2\sqrt{2\pi T}+2\pi T .
\]
The three coefficients $1$, $2$, $1$ are the area, half of the perimeter and the Euler characteristic of the square.

The same decomposition gives the law of the projection. Let $x$ be distributed according to the Gibbs state. Then $x$ lies in $C$ with probability $1/\cZ_T$. In this case $P_Cx=x$. It is uniform on the square because the density of the Gibbs state is constant on $C$. The point $x$ lies in one of the four half strips with probability $2\sqrt{2\pi T}/\cZ_T$. On $S$ the density depends on $x_2$ only, so $P_Cx=(x_1,0)$ is uniform on the lower side. The four sides have the same probability. So in this case $P_Cx$ is uniform on the boundary. Finally $x$ lies in one of the four quadrants with probability $2\pi T/\cZ_T$. Then $P_Cx$ is a vertex, each vertex having the same probability. Thus under the Gibbs state the projection $P_Cx$ is uniform on the square with probability $1/\cZ_T$, uniform on the boundary with probability $2\sqrt{2\pi T}/\cZ_T$ and equal to a uniformly chosen vertex with probability $2\pi T/\cZ_T$.
\end{example}

\subsection{The structure theorem}
We recall the notions from convex geometry which we need. We refer to \cite[Section 4.2]{Schneider2014} for these notions. Let $\omega_n=\pi^{n/2}/\Gamma(\frac n2+1)$ be the volume of the unit ball of $\R^n$. For every Borel set $E\subset\R^d$ and every $r>0$ the \emph{local Steiner formula} holds:
\begin{equation}\label{eq:steiner}
\Leb\big\{x\in\R^d:\ 0<d_C(x)\le r,\ P_Cx\in E\big\}=\sum_{j=0}^{d-1}r^{d-j}\,\omega_{d-j}\,\Phi_j(C,E).
\end{equation}
The left side is the volume of the set of points outside $C$ which are at distance at most $r$ from $C$ and whose nearest point lies in $E$. Here $\Phi_0(C,\cdot),\dots,\Phi_{d-1}(C,\cdot)$ are finite Borel measures concentrated on the boundary of $C$, called the \emph{curvature measures} of $C$. One puts $\Phi_d(C,E)=\Leb(C\cap E)$. The total masses $V_j(C)=\Phi_j(C,\R^d)$ are the \emph{intrinsic volumes} of $C$. We shall use that $V_0(C)=1$, that $V_j(C)>0$ for $0\le j\le k$, that $V_j(C)=0$ and $\Phi_j(C,\cdot)=0$ for $j>k$ and that
\begin{equation}\label{eq:Phik}
\Phi_k(C,E)=\cH^k(C\cap E),\qquad V_k(C)=\cH^k(C)>0 .
\end{equation}
For a set with nonempty interior, $2V_{d-1}(C)$ is the surface area of $C$. We denote by $\chi_n$ a random variable distributed as the Euclidean norm of a standard Gaussian vector in $\R^n$, with $\chi_0=0$. For $n\ge1$ its density is
\begin{equation}\label{eq:chi}
\frac{2}{2^{n/2}\,\Gamma(n/2)}\,r^{n-1}e^{-r^2/2}\qquad(r>0).
\end{equation}

The statements on $\Phi_j$ for $j\ge k$ follow from \eqref{eq:steiner} by a direct calculation. We include it for the convenience of the reader. If $k=d$, then \eqref{eq:Phik} is the definition of $\Phi_d$. Now let $k<d$. Let $L$ be the affine hull of $C$, and let $L^\perp$ be the space of vectors which are orthogonal to all the differences $c-c'$, $c,c'\in C$. The dimension of $L^\perp$ is $d-k$. Every $x\in\R^d$ can be written in a unique way as $x=y+z$ with $y\in L$ and $z\in L^\perp$. For $c\in C$ the vector $y-c$ is orthogonal to $z$, so $|x-c|^2=|y-c|^2+|z|^2$. Taking the minimum over $c\in C$ we get
\[
P_Cx=P_Cy,\qquad d_C(x)^2=d_C(y)^2+|z|^2 .
\]
Hence the set on the left side of \eqref{eq:steiner} is the set of all $y+z$ with $P_Cy\in E$ and $0<d_C(y)^2+|z|^2\le r^2$. We compute its volume by the Fubini theorem, integrating first in $z$. If $y\in C\cap E$, then $d_C(y)=0$ and $P_Cy=y$. In this case the admissible $z$ form the ball of radius $r$ in $L^\perp$ without its centre. The volume of this set is $\omega_{d-k}r^{d-k}$. If $y\in L\setminus C$ with $P_Cy\in E$ and $d_C(y)\le r$, then the admissible $z$ form the ball of radius $(r^2-d_C(y)^2)^{1/2}$ in $L^\perp$, of volume $\omega_{d-k}(r^2-d_C(y)^2)^{(d-k)/2}$. For all other $y$ there is no admissible $z$. So the left side of \eqref{eq:steiner} is equal to
\begin{equation}\label{eq:xs3c}
\omega_{d-k}\,r^{d-k}\,\cH^k(C\cap E)+\omega_{d-k}\int_{\{y\in L:\ 0<d_C(y)\le r,\ P_Cy\in E\}}\big(r^2-d_C(y)^2\big)^{(d-k)/2}\,\cH^k(dy).
\end{equation}
When $k=0$ the set $L$ is a single point. Then the integral in \eqref{eq:xs3c} is zero and the sums over $j<k$ below are empty. Now $L$ is a Euclidean space of dimension $k$. We apply \eqref{eq:steiner} in $L$ and denote by $\Phi'_0(C,\cdot),\dots,\Phi'_{k-1}(C,\cdot)$ the curvature measures of $C$ as a subset of $L$. For $\varrho>0$ this gives
\[
G(\varrho):=\cH^k\big\{y\in L:\ 0<d_C(y)\le\varrho,\ P_Cy\in E\big\}=\sum_{j=0}^{k-1}\varrho^{\,k-j}\,\omega_{k-j}\,\Phi'_j(C,E).
\]
The function $G$ is the distribution function of the image, under the map $y\mapsto d_C(y)$, of the measure $\cH^k$ restricted to $\{y\in L\setminus C:P_Cy\in E\}$. Its derivative is $G'(\varrho)=\sum_{j<k}(k-j)\,\varrho^{\,k-j-1}\omega_{k-j}\Phi'_j(C,E)$. Since $G(\varrho)=\int_0^\varrho G'(s)\,ds$ for all $\varrho>0$, this image measure has the density $G'$ on $(0,\infty)$. Therefore the integral in \eqref{eq:xs3c} is equal to
\[
\int_0^r\big(r^2-\varrho^2\big)^{(d-k)/2}\,G'(\varrho)\,d\varrho=\sum_{j=0}^{k-1}(k-j)\,\omega_{k-j}\,\Phi'_j(C,E)\int_0^r\big(r^2-\varrho^2\big)^{(d-k)/2}\varrho^{\,k-j-1}\,d\varrho .
\]
We compute the last integral with the change of variable $\varrho=r\sqrt u$, $d\varrho=\frac r2u^{-1/2}du$:
\begin{align*}
\int_0^r\big(r^2-\varrho^2\big)^{(d-k)/2}\varrho^{\,k-j-1}\,d\varrho&=\frac{r^{d-j}}2\int_0^1(1-u)^{(d-k)/2}\,u^{(k-j)/2-1}\,du\\
&=\frac{r^{d-j}}2\cdot\frac{\Gamma\big(\frac{k-j}2\big)\,\Gamma\big(\frac{d-k}2+1\big)}{\Gamma\big(\frac{d-j}2+1\big)} .
\end{align*}
Here the power of $r$ is $(d-k)+(k-j-1)+1=d-j$. In the second line we used the value of the beta integral. Since $\frac{k-j}2\,\Gamma(\frac{k-j}2)=\Gamma(\frac{k-j}2+1)$ and $\omega_n=\pi^{n/2}/\Gamma(\frac n2+1)$, we get
\[
\omega_{d-k}\,(k-j)\,\omega_{k-j}\cdot\frac12\cdot\frac{\Gamma\big(\frac{k-j}2\big)\,\Gamma\big(\frac{d-k}2+1\big)}{\Gamma\big(\frac{d-j}2+1\big)}=\frac{\pi^{(d-k)/2}\,\pi^{(k-j)/2}}{\Gamma\big(\frac{d-j}2+1\big)}=\omega_{d-j} .
\]
Thus \eqref{eq:xs3c} is equal to
\[
\omega_{d-k}\,r^{d-k}\,\cH^k(C\cap E)+\sum_{j=0}^{k-1}r^{d-j}\,\omega_{d-j}\,\Phi'_j(C,E).
\]
This is a polynomial in $r$ which is equal to the right side of \eqref{eq:steiner} for every $r>0$. Two polynomials which agree for all $r>0$ have the same coefficients. Comparing the coefficients of $r^{d-j}$ we find that $\Phi_j(C,E)=0$ for $k<j\le d-1$, that $\Phi_k(C,E)=\cH^k(C\cap E)$ and that $\Phi_j(C,E)=\Phi'_j(C,E)$ for $j<k$. Also $\Phi_d(C,E)=\Leb(C\cap E)=0$, because $C$ lies in $L$. Finally $\cH^k(C)>0$, since $C$ has interior points relative to $L$.

\begin{theorem}\label{thm:gibbs}
Let $C\subset\R^d$ be a nonempty compact convex set of dimension $k$ and let $T>0$. Then
\begin{equation}\label{eq:wills}
\cZ_T(C)=\sum_{j=0}^{k}(2\pi T)^{(d-j)/2}\,V_j(C).
\end{equation}
Moreover, if $x$ is distributed according to $\gamma_T^C$, then the law of the pair $(P_Cx,d_C(x))$ is
\begin{equation}\label{eq:mixture}
\sum_{j=0}^{k}w_j(T)\ \frac{\Phi_j(C,\cdot)}{V_j(C)}\otimes\Law\big(\sqrt T\,\chi_{d-j}\big),\qquad w_j(T)=\frac{(2\pi T)^{(d-j)/2}\,V_j(C)}{\cZ_T(C)}.
\end{equation}
\end{theorem}

\begin{proof}
We denote by $M$ the image of the Lebesgue measure on $\R^d\setminus C$ under the map $x\mapsto(P_Cx,d_C(x))$. It is a measure on $\R^d\times(0,\infty)$, because $d_C(x)>0$ for $x\notin C$. Let $E\subset\R^d$ be a Borel set and let $r>0$. By the definition of $M$ and by \eqref{eq:steiner},
\[
M\big(E\times(0,r]\big)=\Leb\big\{x\in\R^d:\ 0<d_C(x)\le r,\ P_Cx\in E\big\}=\sum_{j=0}^{d-1}r^{d-j}\omega_{d-j}\Phi_j(C,E).
\]
For every integer $n\ge1$ we have $r^n=\int_0^rn\,\varrho^{\,n-1}\,d\varrho$. We use this with $n=d-j$ in each term and get
\[
M\big(E\times(0,r]\big)=\sum_{j=0}^{d-1}r^{d-j}\omega_{d-j}\Phi_j(C,E)=\sum_{j=0}^{d-1}\Phi_j(C,E)\int_0^r(d-j)\,\omega_{d-j}\,\varrho^{\,d-j-1}\,d\varrho .
\]
The right side is the value at $E\times(0,r]$ of the measure which appears in \eqref{eq:M} below. So this measure and $M$ agree on all sets of the form $E\times(0,r]$. These sets are closed under finite intersections, since
\[
\big(E\times(0,r]\big)\cap\big(E'\times(0,r']\big)=(E\cap E')\times\big(0,\min\{r,r'\}\big].
\]
They generate the Borel sets of $\R^d\times(0,\infty)$. Both sides are finite on them. Also $\R^d\times(0,\infty)$ is the union of the increasing sequence of sets $\R^d\times(0,m]$, $m=1,2,\dots$. Two measures which agree on such a family of sets are equal. Therefore
\begin{equation}\label{eq:M}
M(dp,d\varrho)=\sum_{j=0}^{d-1}\Phi_j(C,dp)\,(d-j)\,\omega_{d-j}\,\varrho^{\,d-j-1}\,d\varrho .
\end{equation}

Next we compute the integrals in the variable $\varrho$. Let $n=d-j\ge1$. Since $\Gamma(\frac n2+1)=\frac n2\,\Gamma(\frac n2)$, we have
\[
n\,\omega_n=\frac{n\,\pi^{n/2}}{\Gamma(\frac n2+1)}=\frac{n\,\pi^{n/2}}{\frac n2\,\Gamma(\frac n2)}=\frac{2\pi^{n/2}}{\Gamma(\frac n2)} .
\]
We also write $\pi^{n/2}=(2\pi T)^{n/2}/(2T)^{n/2}$. Then, for $\varrho>0$,
\[
n\,\omega_n\,\varrho^{\,n-1}e^{-\varrho^2/2T}=\frac{2\pi^{n/2}}{\Gamma(n/2)}\,\varrho^{\,n-1}e^{-\varrho^2/2T}=(2\pi T)^{n/2}\cdot\frac{2}{(2T)^{n/2}\Gamma(n/2)}\,\varrho^{\,n-1}e^{-\varrho^2/2T}.
\]
We show that the last factor is the density of $\sqrt T\,\chi_n$. Let $f_n$ be the density \eqref{eq:chi} of $\chi_n$. For $a>0$, the change of variable $\varrho=\sqrt T\,r$ gives
\[
\P\big(\sqrt T\,\chi_n\le a\big)=\P\big(\chi_n\le a/\sqrt T\big)=\int_0^{a/\sqrt T}f_n(r)\,dr=\int_0^{a}f_n\Big(\frac{\varrho}{\sqrt T}\Big)\,\frac{d\varrho}{\sqrt T} .
\]
So the density of $\sqrt T\,\chi_n$ at the point $\varrho>0$ is
\[
\frac1{\sqrt T}\,f_n\Big(\frac{\varrho}{\sqrt T}\Big)=\frac1{\sqrt T}\cdot\frac{2}{2^{n/2}\,\Gamma(n/2)}\cdot\frac{\varrho^{\,n-1}}{T^{(n-1)/2}}\,e^{-\varrho^2/2T}=\frac{2}{(2T)^{n/2}\Gamma(n/2)}\,\varrho^{\,n-1}e^{-\varrho^2/2T},
\]
because $\sqrt T\cdot T^{(n-1)/2}\cdot2^{n/2}=T^{n/2}\,2^{n/2}=(2T)^{n/2}$. This is the last factor above. It has integral one, being the density of a probability measure. Let $\nu_n$ denote the law of $\sqrt T\,\chi_n$, so that $\nu_0$ is the Dirac mass at $0$. We have shown that
\begin{equation}\label{eq:xs3a}
n\,\omega_n\,\varrho^{\,n-1}e^{-\varrho^2/2T}\,d\varrho=(2\pi T)^{n/2}\,\nu_n(d\varrho)\qquad(n\ge1).
\end{equation}

We now put \eqref{eq:M} and \eqref{eq:xs3a} together. Let $F$ be a bounded Borel function on $\R^d\times[0,\infty)$. We split the integral of $F(P_Cx,d_C(x))\,e^{-d_C(x)^2/2T}$ over $\R^d$ into the integral over $C$ and the integral over its complement. If $x\in C$, then $P_Cx=x$ and $d_C(x)=0$, so the integrand is $F(x,0)$. On the complement we use the definition of the image measure, which says that
\[
\int_{\R^d\setminus C}G(P_Cx,d_C(x))\,dx=\int G(p,\varrho)\,M(dp,d\varrho)
\]
for every Borel function $G$ on $\R^d\times(0,\infty)$ such that $x\mapsto G(P_Cx,d_C(x))$ is integrable on $\R^d\setminus C$. The function $G(p,\varrho)=F(p,\varrho)\,e^{-\varrho^2/2T}$ has this property, because $F$ is bounded and $\cZ_T(C)$ is finite. We obtain
\[
\int_{\R^d}F(P_Cx,d_C(x))\,e^{-d_C(x)^2/2T}dx=\int_CF(x,0)\,dx+\int F(p,\varrho)\,e^{-\varrho^2/2T}M(dp,d\varrho).
\]
The measure $\Phi_d(C,\cdot)$ is the Lebesgue measure restricted to $C$, so the first term on the right is $\int F(p,0)\,\Phi_d(C,dp)$. In the second term we insert \eqref{eq:M} and then \eqref{eq:xs3a} with $n=d-j$:
\begin{align*}
\int F(p,\varrho)\,e^{-\varrho^2/2T}M(dp,d\varrho)&=\sum_{j=0}^{d-1}\int\!\!\int F(p,\varrho)\,(d-j)\,\omega_{d-j}\,\varrho^{\,d-j-1}e^{-\varrho^2/2T}\,d\varrho\ \Phi_j(C,dp)\\
&=\sum_{j=0}^{d-1}(2\pi T)^{(d-j)/2}\int\!\!\int F(p,\varrho)\,\Phi_j(C,dp)\,\nu_{d-j}(d\varrho).
\end{align*}
All these integrals are finite, because $F$ is bounded, the measures $\Phi_j(C,\cdot)$ are finite and the $\nu_n$ are probability measures. Putting the two parts together we have
\begin{align*}
\int_{\R^d}F(P_Cx,d_C(x))\,e^{-d_C(x)^2/2T}dx&=\int F(p,0)\,\Phi_d(C,dp)\\
&\quad+\sum_{j=0}^{d-1}(2\pi T)^{(d-j)/2}\int\!\!\int F(p,\varrho)\,\Phi_j(C,dp)\,\nu_{d-j}(d\varrho),
\end{align*}
where $\nu_n$ is the law of $\sqrt T\chi_n$. Since $\nu_0$ is the Dirac mass at $0$ and $(2\pi T)^0=1$, the term with $\Phi_d$ is the term $j=d$ of the same sum, namely
\[
\int F(p,0)\,\Phi_d(C,dp)=(2\pi T)^{0}\int\!\!\int F(p,\varrho)\,\Phi_d(C,dp)\,\nu_0(d\varrho).
\]
Also $\Phi_j(C,\cdot)=0$ for $j>k$, so only the terms with $j\le k$ remain. Thus
\begin{equation}\label{eq:xs3b}
\int_{\R^d}F(P_Cx,d_C(x))\,e^{-d_C(x)^2/2T}dx=\sum_{j=0}^{k}(2\pi T)^{(d-j)/2}\int\!\!\int F(p,\varrho)\,\Phi_j(C,dp)\,\nu_{d-j}(d\varrho).
\end{equation}
Take $F=1$ in \eqref{eq:xs3b}. The left side is $\cZ_T(C)$. On the right side each double integral is equal to $\Phi_j(C,\R^d)\,\nu_{d-j}([0,\infty))=V_j(C)$. This gives \eqref{eq:wills}. Now divide \eqref{eq:xs3b} by $\cZ_T(C)$. For $j\le k$ we have $V_j(C)>0$, so we can write $\Phi_j(C,dp)=V_j(C)\cdot\Phi_j(C,dp)/V_j(C)$. We obtain
\[
\int_{\R^d}F(P_Cx,d_C(x))\,\gamma_T^C(dx)=\sum_{j=0}^{k}w_j(T)\int\!\!\int F(p,\varrho)\,\frac{\Phi_j(C,dp)}{V_j(C)}\,\nu_{d-j}(d\varrho).
\]
Since this holds for every bounded Borel function $F$, it is the statement \eqref{eq:mixture}. This completes the proof.
\end{proof}

\begin{remark}
With the change of variable $x=\sqrt{2\pi T}\,z$ formula \eqref{eq:wills} becomes
\[
\int_{\R^d}e^{-\pi d_K(z)^2}\,dz=\sum_{j=0}^dV_j(K)\qquad\text{for }K=C/\sqrt{2\pi T},
\]
which is the integral representation of the Wills functional \cite{Hadwiger1975}. Indeed, $dx=(2\pi T)^{d/2}dz$ and $d_C(x)=\sqrt{2\pi T}\,d_K(z)$, so that $\cZ_T(C)=(2\pi T)^{d/2}\int_{\R^d}e^{-\pi d_K(z)^2}dz$. On the other hand $V_j$ is homogeneous of degree $j$, so that $V_j(C)=(2\pi T)^{j/2}V_j(K)$ and the right side of \eqref{eq:wills} is $(2\pi T)^{d/2}\sum_jV_j(K)$. For a polytope, $\Phi_j(C,\cdot)$ is the sum over the faces $F$ of dimension $j$ of $\cH^j$ restricted to $F$ multiplied by the external angle of $C$ at $F$ \cite[Section 4.2]{Schneider2014}. So for a polytope \eqref{eq:mixture} has the same form as in Example \ref{ex:square}.
\end{remark}

\subsection{Consequences}
The first corollary gives the mean square distance from the feasible set at equilibrium, which measures how much the constraint is violated at temperature $T$.

\begin{corollary}\label{cor:violation}
In the setting of Theorem \ref{thm:gibbs},
\begin{equation}\label{eq:violation}
\int d_C^2\,d\gamma_T^C=T\sum_{j=0}^{k}(d-j)\,w_j(T),\qquad\text{hence}\qquad(d-k)\,T\le\int d_C^2\,d\gamma_T^C\le d\,T .
\end{equation}
If $C$ has nonempty interior, then $\gamma_T^C(C)=V_d(C)/\cZ_T(C)$ and, as $T\to0$,
\[
\int d_C^2\,d\gamma_T^C=\sqrt{2\pi}\,\frac{V_{d-1}(C)}{V_d(C)}\,T^{3/2}\,\big(1+O(\sqrt T)\big).
\]
\end{corollary}

\begin{proof}
Let $x$ be distributed according to $\gamma_T^C$. By \eqref{eq:mixture} the law of $d_C(x)$ is the mixture of the laws of $\sqrt T\,\chi_{d-j}$, $0\le j\le k$, with the weights $w_j(T)$. Hence
\[
\int d_C^2\,d\gamma_T^C=\sum_{j=0}^kw_j(T)\,\E\big(\sqrt T\,\chi_{d-j}\big)^2=\sum_{j=0}^kw_j(T)\,T\,\E\chi_{d-j}^2 .
\]
We now check that $\E\chi_n^2=n$. For $n=0$ this is clear, since $\chi_0=0$. Let $n\ge1$. By definition $\chi_n^2$ has the law of $g_1^2+\dots+g_n^2$, where $g_1,\dots,g_n$ are independent standard Gaussian variables. Since $\E g_i^2=1$, we get $\E\chi_n^2=n$. We can also find the same value from the density \eqref{eq:chi}. To see this, let $m>0$. The substitution $s=r^2/2$, that is, $r=\sqrt{2s}$ and $dr=ds/\sqrt{2s}$, gives
\begin{align*}
\int_0^\infty r^{m-1}e^{-r^2/2}\,dr&=\int_0^\infty(2s)^{(m-1)/2}e^{-s}\,\frac{ds}{\sqrt{2s}}\\
&=2^{m/2-1}\int_0^\infty s^{m/2-1}e^{-s}\,ds=2^{m/2-1}\,\Gamma\Big(\frac m2\Big).
\end{align*}
For $m=n$ this shows that \eqref{eq:chi} has integral one. For $m=n+2$ it gives
\[
\E\chi_n^2=\frac{2}{2^{n/2}\,\Gamma(\frac n2)}\int_0^\infty r^{n+1}e^{-r^2/2}\,dr=\frac{2\cdot2^{n/2}\,\Gamma(\frac n2+1)}{2^{n/2}\,\Gamma(\frac n2)}=2\cdot\frac n2=n .
\]
Inserting $\E\chi_{d-j}^2=d-j$ in the sum above we obtain the equality in \eqref{eq:violation}.

For the inequalities we note that $\sum_{j=0}^kw_j(T)=1$ by \eqref{eq:wills} and the definition of the weights and that $d-k\le d-j\le d$ for $0\le j\le k$. So
\[
(d-k)\,T=T\sum_{j=0}^k(d-k)\,w_j(T)\le T\sum_{j=0}^k(d-j)\,w_j(T)\le T\sum_{j=0}^kd\,w_j(T)=d\,T .
\]

Suppose now that $C$ has nonempty interior, so that $k=d$. Since $d_C=0$ on $C$ and $V_d(C)=\Leb(C)$,
\[
\gamma_T^C(C)=\frac1{\cZ_T(C)}\int_Ce^{-d_C(x)^2/2T}\,dx=\frac{\Leb(C)}{\cZ_T(C)}=\frac{V_d(C)}{\cZ_T(C)} .
\]
This number is $w_d(T)$, the mass of $\{d_C=0\}=C$. In the sum in \eqref{eq:violation} the term $j=d$ vanishes, because of the factor $d-j$. With the definition of the weights we get
\[
\int d_C^2\,d\gamma_T^C=T\sum_{j=0}^{d-1}(d-j)\,w_j(T)=\frac{T}{\cZ_T(C)}\sum_{j=0}^{d-1}(d-j)\,(2\pi T)^{(d-j)/2}\,V_j(C).
\]
The leading term of the sum is the term $j=d-1$, which is $\sqrt{2\pi T}\,V_{d-1}(C)$. It corresponds to $w_{d-1}(T)=\sqrt{2\pi T}\,V_{d-1}(C)/\cZ_T(C)$. For $j\le d-2$ the exponent $(d-j)/2$ is at least $1$, so for $T\le1$ each of the other terms is at most a constant times $T$. Since $V_{d-1}(C)>0$, this gives
\[
\sum_{j=0}^{d-1}(d-j)\,(2\pi T)^{(d-j)/2}\,V_j(C)=\sqrt{2\pi T}\,V_{d-1}(C)\,\big(1+O(\sqrt T)\big).
\]
In the same way, by \eqref{eq:wills} with $k=d$,
\[
\cZ_T(C)=V_d(C)+\sum_{j=0}^{d-1}(2\pi T)^{(d-j)/2}\,V_j(C)=V_d(C)\,\big(1+O(\sqrt T)\big),
\]
because $V_d(C)>0$ and every term of the last sum is $O(\sqrt T)$. Dividing, we find
\[
\int d_C^2\,d\gamma_T^C=T\cdot\frac{\sqrt{2\pi T}\,V_{d-1}(C)\,\big(1+O(\sqrt T)\big)}{V_d(C)\,\big(1+O(\sqrt T)\big)}=\sqrt{2\pi}\,\frac{V_{d-1}(C)}{V_d(C)}\,T^{3/2}\,\big(1+O(\sqrt T)\big).
\]
This completes the proof.
\end{proof}

Our second corollary gives the zero temperature limit of the Gibbs state, with a rate.

\begin{corollary}\label{cor:zerotemp}
In the setting of Theorem \ref{thm:gibbs}, let $D=\diam C$. Then
\begin{equation}\label{eq:zt1}
\TV\big((P_C)_\#\gamma_T^C,\ u_C\big)\le1-w_k(T)\le\sum_{j=0}^{k-1}(2\pi T)^{(k-j)/2}\,\frac{V_j(C)}{V_k(C)},
\end{equation}
\begin{equation}\label{eq:zt2}
\Wt\big(\gamma_T^C,\ u_C\big)\le\sqrt{dT}+D\,\big(1-w_k(T)\big)^{1/2}.
\end{equation}
In particular $\gamma_T^C\to u_C$ as $T\to0$ and $\Wt(\gamma_T^C,u_C)=O(T^{1/4})$.
\end{corollary}

\begin{proof}
By \eqref{eq:mixture} the first marginal of the law of $(P_Cx,d_C(x))$ is the sum of the measures $w_j(T)\,\Phi_j(C,\cdot)/V_j(C)$, $0\le j\le k$. Here $\Phi_k(C,\cdot)/V_k(C)=\cH^k(C\cap\,\cdot\,)/\cH^k(C)=u_C$ by \eqref{eq:Phik}, so the term $j=k$ is $w_k(T)\,u_C$. Thus
\[
(P_C)_\#\gamma_T^C=w_k(T)\,u_C+\sum_{j<k}w_j(T)\,\frac{\Phi_j(C,\cdot)}{V_j(C)} .
\]
Two probability measures of the form $w\,u+(1-w)\nu$ and $u$ have total variation distance at most $1-w$. To see this here, put $w=w_k(T)$. If $w=1$, then $(P_C)_\#\gamma_T^C=u_C$ and there is nothing to prove. Let $w<1$. Since the weights have sum one, $\sum_{j<k}w_j(T)=1-w$. Hence
\[
\nu=\frac1{1-w}\sum_{j<k}w_j(T)\,\frac{\Phi_j(C,\cdot)}{V_j(C)}
\]
is a probability measure with $(P_C)_\#\gamma_T^C=w\,u_C+(1-w)\,\nu$. For every Borel set $A$,
\begin{align*}
\big|(P_C)_\#\gamma_T^C(A)-u_C(A)\big|&=\big|w\,u_C(A)+(1-w)\,\nu(A)-u_C(A)\big|\\
&=(1-w)\,\big|\nu(A)-u_C(A)\big|\le1-w .
\end{align*}
Taking the supremum over $A$ we get the first inequality in \eqref{eq:zt1}. For the second one, we use \eqref{eq:wills} and the definition of $w_k(T)$, then divide the numerator and the denominator by $(2\pi T)^{(d-k)/2}$:
\begin{align*}
1-w_k(T)&=\frac{\sum_{j<k}(2\pi T)^{(d-j)/2}V_j(C)}{\sum_{j\le k}(2\pi T)^{(d-j)/2}V_j(C)}\\
&=\frac{\sum_{j<k}(2\pi T)^{(k-j)/2}V_j(C)}{\sum_{j\le k}(2\pi T)^{(k-j)/2}V_j(C)}\le\frac{\sum_{j<k}(2\pi T)^{(k-j)/2}V_j(C)}{V_k(C)} .
\end{align*}
Here we used $\frac{d-j}2-\frac{d-k}2=\frac{k-j}2$. The last step holds because all terms of the denominator are nonnegative and the term $j=k$ is $V_k(C)$.

For \eqref{eq:zt2} we use the triangle inequality
\[
\Wt\big(\gamma_T^C,u_C\big)\le\Wt\big(\gamma_T^C,(P_C)_\#\gamma_T^C\big)+\Wt\big((P_C)_\#\gamma_T^C,u_C\big).
\]
We claim that
\[
\Wt\big(\gamma_T^C,(P_C)_\#\gamma_T^C\big)^2\le\int d_C^2\,d\gamma_T^C\le dT,\qquad\Wt\big((P_C)_\#\gamma_T^C,u_C\big)^2\le D^2\big(1-w_k(T)\big).
\]
For the first of these inequalities, let $x$ be distributed according to $\gamma_T^C$. The pair $(x,P_Cx)$ is a coupling of $\gamma_T^C$ and $(P_C)_\#\gamma_T^C$. For this coupling $\E|x-P_Cx|^2=\int d_C^2\,d\gamma_T^C$, which is at most $dT$ by \eqref{eq:violation}. For the second, both measures are carried by $C$, which has diameter $D$. So \eqref{eq:WTV} applies. Together with \eqref{eq:zt1} it gives
\[
\Wt\big((P_C)_\#\gamma_T^C,u_C\big)^2\le D^2\,\TV\big((P_C)_\#\gamma_T^C,u_C\big)\le D^2\big(1-w_k(T)\big).
\]
Taking square roots and adding, we obtain \eqref{eq:zt2}.

It remains to prove the last statement. Let $T\le1/(2\pi)$. For $j<k$ we have $(k-j)/2\ge\frac12$ and $2\pi T\le1$, hence $(2\pi T)^{(k-j)/2}\le(2\pi T)^{1/2}$. By \eqref{eq:zt1},
\[
1-w_k(T)\le\sqrt{2\pi T}\,\sum_{j=0}^{k-1}\frac{V_j(C)}{V_k(C)} .
\]
We insert this in \eqref{eq:zt2} and use $\sqrt T\le T^{1/4}$:
\[
\Wt\big(\gamma_T^C,u_C\big)\le\sqrt d\,T^{1/4}+D\,(2\pi)^{1/4}\Big(\sum_{j=0}^{k-1}\frac{V_j(C)}{V_k(C)}\Big)^{1/2}T^{1/4} .
\]
So $\Wt(\gamma_T^C,u_C)=O(T^{1/4})$. Hence $\Wt(\gamma_T^C,u_C)\to0$ as $T\to0$. Finally, convergence in $\Wt$ implies weak convergence.
\end{proof}

The estimate $\Wt(\gamma_T^C,u_C)=O(T^{1/4})$ of Corollary \ref{cor:zerotemp} is an upper bound only. We do not claim that the exponent $1/4$ is sharp.

\subsection{Dependence on the set}
At a fixed positive temperature the Gibbs state depends continuously on the feasible set. At zero temperature this is false. This gives a second pair of limits which do not commute.

\begin{proposition}\label{prop:context}
\begin{enumerate}[label=\textup{(\alph*)}]
\item Let $C_n$ and $C$ be nonempty compact convex subsets of $\R^d$ such that $\haus(C_n,C)\to0$, and let $T>0$. Then the Gibbs states $\gamma_T^{C_n}$ converge to $\gamma_T^C$ in total variation.
\item Let $C_q\subset\R^2$ be the triangle with vertices $(0,0)$, $(1,0)$ and $(1,q)$ for $q>0$, and let $C_0=[0,1]\times\{0\}$. Then $\haus(C_q,C_0)=q$, and
\[
\lim_{q\to0}\ \lim_{T\to0}\gamma_T^{C_q}=2x_1\one_{[0,1]}(x_1)\,dx_1\otimes\delta_0\ \ne\ u_{C_0}=\lim_{T\to0}\ \lim_{q\to0}\gamma_T^{C_q}.
\]
\end{enumerate}
\end{proposition}

\begin{proof}
(a) Put $h_n=\haus(C_n,C)$. We write $g_n=e^{-d_{C_n}^2/2T}$ and $g=e^{-d_C^2/2T}$. By \eqref{eq:dhaus}, $|d_{C_n}-d_C|\le h_n$, with $h_n\to0$. So $d_{C_n}(x)\to d_C(x)$ for every $x$. Hence $e^{-d_{C_n}^2/2T}\to e^{-d_C^2/2T}$ pointwise. Let $n$ be so large that $h_n\le1$. Then $d_{C_n}\ge d_C-1$ and $d_{C_n}\ge0$. Thus $d_{C_n}\ge(d_C-1)_+$. It follows that
\[
0\le g_n\le G,\qquad\text{where }G=e^{-(d_C-1)_+^2/2T}.
\]
The function $G$ is integrable. Indeed, put $\|C\|=\max_{c\in C}|c|$. Then $d_C(x)\ge|x|-\|C\|$. Hence $G\le1$ everywhere and $G(x)\le e^{-(|x|-\|C\|-1)^2/2T}$ for $|x|\ge\|C\|+1$. We also have $g\le G$, so $|g_n-g|\le2G$. By dominated convergence,
\[
\cZ_T(C_n)=\int_{\R^d}g_n\,dx\longrightarrow\int_{\R^d}g\,dx=\cZ_T(C),\qquad\int_{\R^d}|g_n-g|\,dx\longrightarrow0 .
\]
It remains to see that the densities converge in $L^1$. The densities of $\gamma_T^{C_n}$ and of $\gamma_T^C$ are $g_n/\cZ_T(C_n)$ and $g/\cZ_T(C)$. For every Borel set $A$ the number $|\gamma_T^{C_n}(A)-\gamma_T^C(A)|$ is at most the integral of the absolute value of the difference of the densities. Therefore
\begin{align*}
\TV\big(\gamma_T^{C_n},\gamma_T^C\big)&\le\int_{\R^d}\Big|\frac{g_n}{\cZ_T(C_n)}-\frac{g}{\cZ_T(C)}\Big|\,dx\\
&\le\frac1{\cZ_T(C_n)}\int_{\R^d}|g_n-g|\,dx+\Big|\frac1{\cZ_T(C_n)}-\frac1{\cZ_T(C)}\Big|\int_{\R^d}g\,dx\\
&=\frac1{\cZ_T(C_n)}\int_{\R^d}|g_n-g|\,dx+\frac{|\cZ_T(C)-\cZ_T(C_n)|}{\cZ_T(C_n)} ,
\end{align*}
where we used $\int g\,dx=\cZ_T(C)$. Both terms tend to zero, because $\cZ_T(C_n)\to\cZ_T(C)>0$.

(b) We first compute the Hausdorff distance. The triangle $C_q$ is the set
\[
C_q=\{(x_1,x_2):\ 0\le x_1\le1,\ 0\le x_2\le qx_1\}.
\]
Hence $C_0\subset C_q$, so $d_{C_q}$ vanishes on $C_0$. For a point $c=(x_1,x_2)$ of $C_q$ the nearest point of $C_0$ is $(x_1,0)$, so $d_{C_0}(c)=x_2\le qx_1\le q$, with equality for $c=(1,q)$. Thus the point of $C_q$ farthest from $C_0$ is $(1,q)$, at distance $q$. Hence $\haus(C_q,C_0)=q$.

Next we compute the limit on the left. Let $q>0$. By Corollary \ref{cor:zerotemp}, $\lim_{T\to0}\gamma_T^{C_q}=u_{C_q}$. The area of $C_q$ is $\int_0^1qx_1\,dx_1=q/2$, so $u_{C_q}$ has the density $2/q$ on $C_q$. The first coordinate under $u_{C_q}$ has the density
\[
\int_0^{qx_1}\frac2q\,dx_2=\frac{qx_1}{q/2}=2x_1\qquad(0\le x_1\le1),
\]
and the second coordinate lies in $[0,q]$. Now let $f$ be a bounded continuous function on $\R^2$. With the substitution $x_2=qx_1s$ in the inner integral,
\[
\int f\,du_{C_q}=\int_0^1\int_0^{qx_1}f(x_1,x_2)\,\frac2q\,dx_2\,dx_1=\int_0^1\int_0^1f(x_1,qx_1s)\,2x_1\,ds\,dx_1 .
\]
As $q\to0$ we have $f(x_1,qx_1s)\to f(x_1,0)$. Since the integrand is bounded by $2\sup|f|$, bounded convergence gives
\[
\lim_{q\to0}\int f\,du_{C_q}=\int_0^1f(x_1,0)\,2x_1\,dx_1 .
\]
Hence $u_{C_q}\to2x_1\one_{[0,1]}(x_1)dx_1\otimes\delta_0$ weakly as $q\to0$. This is the first iterated limit.

For the other order, fix $T>0$. Since $\haus(C_q,C_0)=q\to0$, part (a) gives $\lim_{q\to0}\gamma_T^{C_q}=\gamma_T^{C_0}$. By Corollary \ref{cor:zerotemp}, $\lim_{T\to0}\gamma_T^{C_0}=u_{C_0}$. The set $C_0$ is a segment of length one, so $u_{C_0}=\one_{[0,1]}(x_1)dx_1\otimes\delta_0$. Thus the first coordinate under $u_{C_0}$ has the density $1$ on $[0,1]$. The two limits are different. For example, the set $\{x_1\le\frac12\}$ has mass $\int_0^{1/2}2x_1\,dx_1=\frac14$ under the first limit and mass $\frac12$ under the second.
\end{proof}

\begin{remark}
Part (b) shows that the zero temperature limit of the Gibbs state is not continuous in the set $C$ for the Hausdorff distance. The limit can change when the set loses a dimension. At a fixed positive temperature the Gibbs state depends continuously on the set.
\end{remark}

\section{Convex sets: the limit diagram and the mixing time}\label{sec:diagram}

We come to the question \eqref{eq:question} for a fixed compact convex set. Three time scales are present in the state equation \eqref{eq:SDEgen}. In a time of order one the drift brings the state to its projection on $C$. Inside $C$ the drift vanishes and the state moves like a Brownian motion with the small diffusion constant $T$, so it needs a time of order $1/T$ to explore $C$. After a time much longer than $1/T$ it has forgotten where it entered. The theorem below makes this precise and gives the limit in each of the three ranges of times.

\subsection{Notation and statement}
Throughout this section, $C\subset\R^d$ is a nonempty compact convex set, $k=\dim C$, $L=\aff C$ is its affine hull, $L_0=L-L$ is the linear subspace parallel to $L$ and
\[
D=\diam C,\quad r_C=\sup\{r>0:\ \text{some ball of radius $r$ of $L$ is contained in }C\}\ \ (k\ge1).
\]
The initial law $\rho_0$ is a probability measure on $\R^d$ with finite second moment. We put
\[
m_0=\Big(\int d_C^2\,d\rho_0\Big)^{1/2}.
\]
The state process $X=X^T$ solves \eqref{eq:SDEgen} with $\Law(X_0)=\rho_0$. We write $\rho_{T,t}=\Law(X^T_t)$. The noiseless flow is defined by
\begin{equation}\label{eq:flow}
\varphi_t(x)=P_Cx+e^{-t}(x-P_Cx)\qquad(t\ge0,\ x\in\R^d).
\end{equation}
By \ref{P5} with $\lambda=e^{-t}$ we have $P_C\varphi_t(x)=P_Cx$, hence $\varphi_t(x)-P_C\varphi_t(x)=e^{-t}(x-P_Cx)$ and
\begin{equation}\label{eq:flowode}
\frac{d}{dt}\varphi_t(x)=-e^{-t}(x-P_Cx)=-\big(\varphi_t(x)-P_C\varphi_t(x)\big).
\end{equation}
Since $\varphi_0(x)=x$, integration of \eqref{eq:flowode} gives
\[
\varphi_t(x)=x-\int_0^t\big(\varphi_u(x)-P_C\varphi_u(x)\big)\,du ,
\]
so $\varphi_t(x)$ is the solution of \eqref{eq:state} with $w=0$, for the constant set $C$ and the initial point $x$.

\emph{Reflected Brownian motion.} Let $k\ge1$. For $x\in C$ let $N_C(x)=\{n\in L_0:\ \ip{n}{c-x}\le0\text{ for all }c\in C\}$ be the normal cone of $C$ at $x$ relative to $L$. Let $\beta$ be a standard Brownian motion in $L_0$ and let $\xi$ be a random point of $C$ independent of $\beta$. By \cite{Tanaka1979} there is a unique pair $(R,K)$ of continuous adapted processes such that $R_s\in C$ for all $s\ge0$, $K$ has locally bounded variation with total variation process $|K|$ and
\begin{equation}\label{eq:rbm}
\begin{aligned}
R_s&=\xi+\sqrt2\,\beta_s-K_s,\qquad K_s=\int_0^sn_u\,d|K|_u,\\
n_u&\in N_C(R_u),\ |n_u|=1\quad\text{for }d|K|\text{ almost every }u.
\end{aligned}
\end{equation}
Here we identify $L$ with $\R^k$. In $\R^k$ the set $C$ is bounded and contains a ball, which is the situation of \cite{Tanaka1979}. If $x$ is an interior point of $C$ relative to $L$, then $N_C(x)=\{0\}$. So the last condition in \eqref{eq:rbm} says that $K$ acts only at the times when $R$ is on the boundary of $C$ relative to $L$, where it keeps $R$ in $C$. The pair $(R,K)$ is the image of $(\xi,\beta)$ under a measurable map which is defined path by path. We call $R$ the \emph{reflected Brownian motion} in $C$ with generator $\Delta$. For a probability measure $\nu$ on $C$ we denote by $\mu_s[\nu]$ the law of $R_s$ when $\Law(\xi)=\nu$, so that $\mu_s[\nu]=\int\mu_s[\delta_x]\,\nu(dx)$. If $k=0$, so that $C=\{p\}$, we put $\mu_s[\nu]=\delta_p$. Finally we put
\begin{equation}\label{eq:Lambda}
\Lambda_C(S)=\frac{\sqrt3}{2\,r_C}\,\big(D^4+8D^2S+4k^2S^2\big)^{1/2}\quad(k\ge1),\qquad\Lambda_C(S)=0\quad(k=0).
\end{equation}

\begin{theorem}\label{thm:diagram}
Let $C$, $\rho_0$ and $\rho_{T,t}$ be as above and let $\mu_s=\mu_s[(P_C)_\#\rho_0]$.
\begin{enumerate}[label=\textup{(\alph*)}]
\item \textup{(Short times.)} For all $T>0$ and $t\ge0$,
\[
\Wt\big(\rho_{T,t},(\varphi_t)_\#\rho_0\big)\le\sqrt{2dTt},\qquad\Wt\big(\rho_{T,t},(P_C)_\#\rho_0\big)\le\sqrt{2dTt}+e^{-t}m_0 .
\]
\item \textup{(Times of order $1/T$.)} Let $0<T\le1/e$ and $S>0$. For every $s$ with $T\log(1/T)\le s\le S$,
\[
\begin{aligned}
\Wt\big(\rho_{T,s/T},\mu_s\big)^2\le{}&4T^2m_0^2+12\,d\,T\log\frac1T\\
&+2\Big[Tm_0+11\Big(T\Big(k+\log\Big(1+\frac ST\Big)\Big)\Big)^{1/2}\Big]\Lambda_C(S).
\end{aligned}
\]
\item \textup{(Long times at fixed temperature.)} For every $T>0$, $\Wt(\rho_{T,t},\gamma_T^C)\to0$ as $t\to\infty$, and $\Wt(\gamma_T^C,u_C)\to0$ as $T\to0$.
\item \textup{(All joint limits.)} Let $T_n\to0$ and $t_n\to\infty$ with $T_nt_n\to s\in[0,\infty]$. Then
\[
\Wt\big(\rho_{T_n,t_n},\mu_s\big)\to0,\qquad\text{where }\mu_0=(P_C)_\#\rho_0\ \text{ and }\ \mu_\infty=u_C .
\]
\end{enumerate}
\end{theorem}

We prove the theorem in the next three subsections. Part (b) uses the envelope of Section \ref{sec:envelope}.

\subsection{Short times}

\begin{lemma}\label{lem:ballistic}
Let $C\subset\R^d$ be a nonempty closed convex set, let $X$ solve \eqref{eq:SDEgen} and let $\varphi$ be the flow \eqref{eq:flow}. Then for all $t\ge0$
\[
\E|X_t-\varphi_t(X_0)|^2\le2dTt,\qquad\E\sup_{0\le u\le t}|X_u-\varphi_u(X_0)|^2\le8dTt .
\]
\end{lemma}

\begin{proof}
Put $\Xi_u=X_u-\varphi_u(X_0)$. By \eqref{eq:SDEgen} and \eqref{eq:flowode},
\[
d\Xi_u=-\big[(X_u-P_CX_u)-(\varphi_u(X_0)-P_C\varphi_u(X_0))\big]\,du+\sqrt{2T}\,dW_u,\qquad\Xi_0=0 .
\]
We write $b_u$ for the coefficient of $du$, so that $d\Xi_u=b_u\,du+\sqrt{2T}\,dW_u$. Then we apply the It\^o formula to the function $x\mapsto|x|^2$. Let $\Xi^i$, $b^i$ and $W^i$, $1\le i\le d$, be the coordinates. The quadratic variation of $\sqrt{2T}\,W^i$ on $[0,t]$ is $2Tt$. So the It\^o formula for the square of a real semimartingale gives
\[
(\Xi^i_t)^2=2\int_0^t\Xi^i_u\,b^i_u\,du+2\sqrt{2T}\int_0^t\Xi^i_u\,dW^i_u+2Tt\qquad(1\le i\le d).
\]
We add these $d$ equations. With $M_t=\int_0^t\ip{\Xi_u}{dW_u}$ we get
\[
|\Xi_t|^2=-2\int_0^t\ip{\Xi_u}{(X_u-P_CX_u)-(\varphi_u(X_0)-P_C\varphi_u(X_0))}\,du+2\sqrt{2T}\,M_t+2dTt .
\]
The first integrand is nonnegative by \ref{P2}. Indeed, \ref{P2} with $x=X_u$ and $y=\varphi_u(X_0)$, so that $x-y=\Xi_u$, reads
\[
-\ip{\Xi_u}{b_u}=\ip{\Xi_u}{(X_u-P_CX_u)-(\varphi_u(X_0)-P_C\varphi_u(X_0))}\ge|b_u|^2\ge0 .
\]
Hence, for all $t\ge0$,
\begin{equation}\label{eq:delta}
|\Xi_t|^2\le2\sqrt{2T}\,M_t+2dTt .
\end{equation}
The process $M$ is a local martingale. At this point we do not know that its expectation vanishes, so we stop it. Let $\tau_n=\inf\{u:|\Xi_u|\ge n\}$. For $u\le\tau_n$ we have $|\Xi_u|\le n$, so the stopped process $M_{t\wedge\tau_n}$ is a martingale and $\E M_{t\wedge\tau_n}=0$. Taking expectations in \eqref{eq:delta} at the time $t\wedge\tau_n$ gives
\[
\E|\Xi_{t\wedge\tau_n}|^2\le2\sqrt{2T}\,\E M_{t\wedge\tau_n}+2dT\,\E(t\wedge\tau_n)\le2dTt .
\]
The paths of $\Xi$ are continuous, so $\tau_n\to\infty$ and $\Xi_{t\wedge\tau_n}\to\Xi_t$ almost surely as $n\to\infty$. By the Fatou lemma,
\[
\E|\Xi_t|^2\le\liminf_{n\to\infty}\E|\Xi_{t\wedge\tau_n}|^2\le2dTt .
\]
This is the first inequality. It shows that $\E\int_0^t|\Xi_u|^2\,du$ is finite, so $M$ is a square integrable martingale. By the It\^o isometry
\[
\E M_t^2=\int_0^t\E|\Xi_u|^2\,du\le\int_0^t2dTu\,du=dTt^2 .
\]
The Doob inequality for $M$ is $\E\sup_{u\le t}M_u^2\le4\,\E M_t^2$. Together with the Schwarz inequality it gives
\[
\E\sup_{u\le t}M_u\le\Big(\E\sup_{u\le t}M_u^2\Big)^{1/2}\le2\big(\E M_t^2\big)^{1/2}\le2\sqrt{dT}\,t .
\]
Now let $u\le t$. By \eqref{eq:delta} at the time $u$,
\[
|\Xi_u|^2\le2\sqrt{2T}\,M_u+2dTu\le2\sqrt{2T}\,\sup_{v\le t}M_v+2dTt .
\]
The right side does not depend on $u$. We take the supremum over $u\le t$ on the left and then expectations. Using $\sqrt{2T}\sqrt{dT}=\sqrt{2d}\,T$ we get
\[
\E\sup_{u\le t}|\Xi_u|^2\le2\sqrt{2T}\cdot2\sqrt{dT}\,t+2dTt=\big(4\sqrt{2d}+2d\big)Tt\le8dTt,
\]
since $4\sqrt{2d}\le4\sqrt2\,d<6d$. This completes the proof.
\end{proof}

\begin{proof}[Proof of Theorem \ref{thm:diagram}(a)]
The pair $(X_t,\varphi_t(X_0))$ is a coupling of $\rho_{T,t}$ and $(\varphi_t)_\#\rho_0$, that is, a pair of random variables with these two laws. By the definition of $\Wt$ and by Lemma \ref{lem:ballistic},
\[
\Wt\big(\rho_{T,t},(\varphi_t)_\#\rho_0\big)^2\le\E|X_t-\varphi_t(X_0)|^2\le2dTt .
\]
So the first inequality holds. In the same way the pair $(\varphi_t(X_0),P_CX_0)$ is a coupling of $(\varphi_t)_\#\rho_0$ and $(P_C)_\#\rho_0$. By \eqref{eq:flow} we have $\varphi_t(x)-P_Cx=e^{-t}(x-P_Cx)$. Since $|x-P_Cx|=d_C(x)$, it follows that
\[
\E|\varphi_t(X_0)-P_CX_0|^2=e^{-2t}\,\E\,d_C(X_0)^2=e^{-2t}m_0^2 .
\]
This gives $\Wt\big((\varphi_t)_\#\rho_0,(P_C)_\#\rho_0\big)\le e^{-t}m_0$. The second inequality follows from the triangle inequality for $\Wt$:
\[
\Wt\big(\rho_{T,t},(P_C)_\#\rho_0\big)\le\Wt\big(\rho_{T,t},(\varphi_t)_\#\rho_0\big)+\Wt\big((\varphi_t)_\#\rho_0,(P_C)_\#\rho_0\big)\le\sqrt{2dTt}+e^{-t}m_0 .
\]
\end{proof}

\subsection{Times of order $1/T$}
We first separate the directions normal to $L$ from the directions in $L$. Let $\pi_L$ be the orthogonal projection of $\R^d$ onto the affine space $L$. We write $\pi_0$ for the orthogonal projection of $\R^d$ onto the linear subspace $L_0$, so that $\pi_Lx-\pi_Ly=\pi_0(x-y)$ for all $x,y\in\R^d$. Let $x\in\R^d$ and $c\in C$. Since $C\subset L$, the vector $\pi_Lx-c$ lies in $L_0$. The vector $x-\pi_Lx$ is orthogonal to $L_0$. So the Pythagoras theorem gives $|x-c|^2=|x-\pi_Lx|^2+|\pi_Lx-c|^2$. The first term on the right does not depend on $c$. Thus the point of $C$ which is nearest to $x$ is also the point of $C$ which is nearest to $\pi_Lx$. Taking the minimum over $c\in C$ we get
\begin{equation}\label{eq:xs4apyth}
d_C(x)^2=|x-\pi_Lx|^2+d_C(\pi_Lx)^2 .
\end{equation}
Hence
\begin{equation}\label{eq:split}
P_Cx=P_C(\pi_Lx),\qquad x-P_Cx=(x-\pi_Lx)+\big(\pi_Lx-P_C(\pi_Lx)\big),\qquad d_C(\pi_Lx)\le d_C(x).
\end{equation}
In the second relation the vector $x-\pi_Lx$ belongs to $L_0^\perp$ and the vector $\pi_Lx-P_C(\pi_Lx)$ belongs to $L_0$.

Let $W=W^\parallel+W^\perp$, with $W^\parallel=\pi_0W$, be the decomposition of the Brownian motion into its components in $L_0$ and in $L_0^\perp$. These are independent standard Brownian motions in the two subspaces. Put
\[
X^\parallel_t=\pi_LX_t\in L,\qquad X^\perp_t=X_t-\pi_LX_t\in L_0^\perp .
\]
We project \eqref{eq:SDEgen}, which in integral form reads
\[
X_t=X_0-\int_0^t(X_u-P_CX_u)\,du+\sqrt{2T}\,W_t .
\]
By \eqref{eq:split} with $x=X_u$,
\[
X_u-P_CX_u=X^\perp_u+\big(X^\parallel_u-P_CX^\parallel_u\big),\qquad X^\perp_u\in L_0^\perp,\quad X^\parallel_u-P_CX^\parallel_u\in L_0 .
\]
So $\pi_0$ maps $X_u-P_CX_u$ to $X^\parallel_u-P_CX^\parallel_u$, while the projection onto $L_0^\perp$ maps it to $X^\perp_u$. In the same way $X^\parallel_t-X^\parallel_0=\pi_0(X_t-X_0)$ and $X^\perp_t-X^\perp_0$ is the projection of $X_t-X_0$ onto $L_0^\perp$. Applying the two projections to the integral equation we obtain two equations which are not coupled:
\begin{equation}\label{eq:AN}
dX^\parallel_t=-(X^\parallel_t-P_CX^\parallel_t)\,dt+\sqrt{2T}\,dW^\parallel_t,\qquad dX^\perp_t=-X^\perp_t\,dt+\sqrt{2T}\,dW^\perp_t .
\end{equation}
The process $X^\perp$ is an OU process in the space $L_0^\perp$ of dimension $d-k$. We compute its second moment. By the second equation in \eqref{eq:AN},
\[
d\big(e^tX^\perp_t\big)=e^tX^\perp_t\,dt+e^t\,dX^\perp_t=\sqrt{2T}\,e^t\,dW^\perp_t ,
\]
and so
\begin{equation}\label{eq:xs4aou}
X^\perp_t=e^{-t}X^\perp_0+\sqrt{2T}\int_0^te^{-(t-u)}\,dW^\perp_u .
\end{equation}
The stochastic integral has mean zero and is independent of $X_0$. By the It\^o isometry in the space $L_0^\perp$ of dimension $d-k$,
\[
\E\Big|\sqrt{2T}\int_0^te^{-(t-u)}\,dW^\perp_u\Big|^2=2T\,(d-k)\int_0^te^{-2(t-u)}\,du=(d-k)\,T\,(1-e^{-2t}).
\]
Also $|X^\perp_0|^2\le d_C(X_0)^2$ by \eqref{eq:xs4apyth}, so that $\E|X^\perp_0|^2\le m_0^2$. We take the expectation of the square of the norm in \eqref{eq:xs4aou}. The expectation of the inner product of the two terms on the right vanishes, which gives
\begin{equation}\label{eq:Nmoment}
\E|X^\perp_t|^2=e^{-2t}\,\E|X^\perp_0|^2+(d-k)\,T\,(1-e^{-2t})\le e^{-2t}m_0^2+(d-k)\,T .
\end{equation}
The process $X^\parallel$ is the state process in the space $L$ of dimension $k$, in which $C$ has nonempty interior. So the main work is in dimension $k$ with a set having interior points.

The next lemma contains what we need about reflected Brownian motion and its comparison with the state process in the slow time $s=Tt$.

\begin{lemma}\label{lem:penal}
Let $k\ge1$, let $\beta$ be a standard Brownian motion in $L_0$, let $(R,K)$ be as in \eqref{eq:rbm}, let $\eps>0$ and let $Y$ be a continuous process in $L$ with
\begin{equation}\label{eq:Y}
Y_s=Y_0-\frac1\eps\int_0^s(Y_u-P_CY_u)\,du+\sqrt2\,\beta_s .
\end{equation}
\begin{enumerate}[label=\textup{(\roman*)}]
\item For all $0\le s_0\le s$,
\begin{equation}\label{eq:penal}
|Y_s-R_s|^2\le|Y_{s_0}-R_{s_0}|^2+2\sup_{s_0\le u\le s}d_C(Y_u)\,\big(|K|_s-|K|_{s_0}\big).
\end{equation}
\item $\E|R_s-R_0|^2\le2ks$ for all $s\ge0$.
\item $\E|K|_S^2\le\Lambda_C(S)^2$ for all $S\ge0$, with $\Lambda_C$ as in \eqref{eq:Lambda}.
\item If $R'$ is the solution of \eqref{eq:rbm} with the same $\beta$ and another initial point $\xi'\in C$, then $|R_s-R'_s|\le|\xi-\xi'|$ for all $s\ge0$.
\end{enumerate}
\end{lemma}

\begin{proof}
(i) In the difference of \eqref{eq:Y} and \eqref{eq:rbm} the Brownian terms cancel. This gives
\[
Y_s-R_s=(Y_0-\xi)-\frac1\eps\int_0^s(Y_u-P_CY_u)\,du+\int_0^sn_u\,d|K|_u .
\]
We subtract from this the same relation at the time $s_0$ and get
\[
Y_s-R_s=(Y_{s_0}-R_{s_0})-\frac1\eps\int_{s_0}^s(Y_u-P_CY_u)\,du+\int_{s_0}^sn_u\,d|K|_u .
\]
So $Y-R$ is continuous and of bounded variation. The ordinary chain rule for the function $x\mapsto|x|^2$ therefore gives
\begin{equation}\label{eq:chain}
|Y_s-R_s|^2=|Y_{s_0}-R_{s_0}|^2-\frac2\eps\int_{s_0}^s\ip{Y_u-R_u}{Y_u-P_CY_u}\,du+2\int_{s_0}^s\ip{Y_u-R_u}{n_u}\,d|K|_u .
\end{equation}
In the first integral we write $Y_u-R_u=(Y_u-P_CY_u)+(P_CY_u-R_u)$. Since $R_u\in C$, property \ref{P1} gives $\ip{Y_u-P_CY_u}{R_u-P_CY_u}\le0$. Therefore
\[
\ip{Y_u-R_u}{Y_u-P_CY_u}=|Y_u-P_CY_u|^2+\ip{P_CY_u-R_u}{Y_u-P_CY_u}\ge0 .
\]
Thus the first integral in \eqref{eq:chain}, which enters with the factor $-2/\eps$, is nonnegative. In the second integral we use the same decomposition. Since $P_CY_u\in C$ and $n_u\in N_C(R_u)$, we have $\ip{P_CY_u-R_u}{n_u}\le0$. As $|n_u|=1$, this gives
\[
\ip{Y_u-R_u}{n_u}=\ip{Y_u-P_CY_u}{n_u}+\ip{P_CY_u-R_u}{n_u}\le|Y_u-P_CY_u|=d_C(Y_u).
\]
The measure $d|K|$ is nonnegative, so
\[
2\int_{s_0}^s\ip{Y_u-R_u}{n_u}\,d|K|_u\le2\int_{s_0}^sd_C(Y_u)\,d|K|_u\le2\sup_{s_0\le u\le s}d_C(Y_u)\,\big(|K|_s-|K|_{s_0}\big).
\]
Inserting the two estimates in \eqref{eq:chain} gives \eqref{eq:penal}.

(ii) We first write the It\^o formula that we use here and in (iii). Let $Z$ be a continuous process in $L_0$ of the form
\[
Z_s=Z_0+\sqrt2\,\beta_s-K_s ,
\]
where $Z_0$ is a function of $\xi$. We claim that
\begin{equation}\label{eq:xs4aito}
|Z_s|^2=|Z_0|^2+2\sqrt2\int_0^s\ip{Z_u}{d\beta_u}+2ks-2\int_0^s\ip{Z_u}{n_u}\,d|K|_u .
\end{equation}
Indeed, choose an orthonormal basis of $L_0$ and let $Z^i$, $\beta^i$ and $K^i$, $1\le i\le k$, be the coordinates. Then $\beta^1,\dots,\beta^k$ are independent real Brownian motions and $Z^i=Z^i_0+\sqrt2\,\beta^i-K^i$. The process $K^i$ is continuous and of bounded variation, so it does not contribute to the quadratic variation. Hence the quadratic variation of $Z^i$ on $[0,s]$ is that of $\sqrt2\,\beta^i$, which is $2s$. The It\^o formula for the square of a real semimartingale gives
\[
(Z^i_s)^2=(Z^i_0)^2+2\sqrt2\int_0^sZ^i_u\,d\beta^i_u-2\int_0^sZ^i_u\,dK^i_u+2s\qquad(1\le i\le k).
\]
We add these $k$ equations and use $dK_u=n_u\,d|K|_u$ from \eqref{eq:rbm}. This gives \eqref{eq:xs4aito}.

Now take $Z_s=R_s-R_0$. By \eqref{eq:rbm} we have $R_0=\xi$ and $Z_s=\sqrt2\,\beta_s-K_s$, so $Z_0=0$ and \eqref{eq:xs4aito} reads
\[
|R_s-R_0|^2=2\sqrt2\int_0^s\ip{R_u-R_0}{d\beta_u}+2ks-2\int_0^s\ip{R_u-R_0}{n_u}\,d|K|_u .
\]
Since $R_0\in C$ and $n_u\in N_C(R_u)$ we have $\ip{n_u}{R_0-R_u}\le0$, that is, $\ip{R_u-R_0}{n_u}\ge0$. So the last term is nonpositive. The stochastic integral is a martingale with mean zero, because $|R_u-R_0|\le D$. Taking expectations gives $\E|R_s-R_0|^2\le2ks$, which is (ii).

(iii) Since $k=\dim C\ge1$, the convex set $C$ has interior points relative to $L$, so $r_C>0$. We choose a point $c_0\in C$ such that the ball of $L$ with centre $c_0$ and radius $r_C$ is contained in $C$. Such a point exists because $C$ is compact. Indeed, take radii $r_j<r_C$ increasing to $r_C$ and points $c_j\in C$ such that the ball of $L$ with centre $c_j$ and radius $r_j$ is contained in $C$. A subsequence of $(c_j)$ converges to a point $c_0\in C$. If $y\in L$ and $|y-c_0|<r_C$, then $|y-c_j|\le|y-c_0|+|c_0-c_j|<r_j$ for some large $j$ of the subsequence, so $y\in C$. Since $C$ is closed, the closed ball of $L$ with centre $c_0$ and radius $r_C$ is contained in $C$.

Let $x\in C$ and $n\in N_C(x)$ with $|n|=1$. Then $c_0+r_Cn\in C$, hence $\ip{n}{c_0+r_Cn-x}\le0$ by the definition of the normal cone. Since $\ip{n}{r_Cn}=r_C$, this means $\ip{n}{c_0-x}+r_C\le0$, that is,
\begin{equation}\label{eq:inward}
\ip{x-c_0}{n}\ge r_C .
\end{equation}
We apply \eqref{eq:xs4aito} to $Z_s=R_s-c_0=(\xi-c_0)+\sqrt2\,\beta_s-K_s$. With $M_S=\int_0^S\ip{R_u-c_0}{d\beta_u}$ we get
\[
|R_S-c_0|^2=|R_0-c_0|^2+2\sqrt2\,M_S+2kS-2\int_0^S\ip{R_u-c_0}{n_u}\,d|K|_u .
\]
By \eqref{eq:inward} with $x=R_u$ and $n=n_u$,
\[
\int_0^S\ip{R_u-c_0}{n_u}\,d|K|_u\ge r_C\int_0^Sd|K|_u=r_C\,|K|_S .
\]
Hence
\[
2r_C\,|K|_S\le|R_0-c_0|^2-|R_S-c_0|^2+2kS+2\sqrt2\,M_S .
\]
Using $|R_0-c_0|\le D$ and $|R_S-c_0|^2\ge0$,
\[
2r_C\,|K|_S\le D^2+2kS+2\sqrt2\,M_S .
\]
The left side is nonnegative, so we may square both sides. By the inequality $(\alpha_1+\alpha_2+\alpha_3)^2\le3(\alpha_1^2+\alpha_2^2+\alpha_3^2)$,
\[
4r_C^2\,|K|_S^2\le\big(D^2+2kS+2\sqrt2\,M_S\big)^2\le3\big(D^4+4k^2S^2+8M_S^2\big).
\]
Since $|R_u-c_0|\le D$, the process $M$ is a square integrable martingale. The It\^o isometry gives $\E M_S^2=\E\int_0^S|R_u-c_0|^2du\le D^2S$. Taking expectations, we get
\[
4r_C^2\,\E|K|_S^2\le3\big(D^4+4k^2S^2+8D^2S\big).
\]
We divide by $4r_C^2$ and compare with \eqref{eq:Lambda}:
\[
\E|K|_S^2\le\frac3{4r_C^2}\big(D^4+8D^2S+4k^2S^2\big)=\Lambda_C(S)^2 ,
\]
which is (iii).

(iv) We write \eqref{eq:rbm} for $R$ and for $R'$, with the same $\beta$, and take the difference. The Brownian terms cancel:
\[
R_s-R'_s=(\xi-\xi')-\int_0^sn_u\,d|K|_u+\int_0^sn'_u\,d|K'|_u .
\]
So the difference $R-R'$ is of bounded variation. The chain rule for $x\mapsto|x|^2$ now gives
\[
|R_s-R'_s|^2=|\xi-\xi'|^2-2\int_0^s\ip{R_u-R'_u}{n_u}\,d|K|_u-2\int_0^s\ip{R'_u-R_u}{n'_u}\,d|K'|_u .
\]
Both integrands are nonnegative. Indeed, $R'_u\in C$ and $n_u\in N_C(R_u)$ give $\ip{n_u}{R'_u-R_u}\le0$. Similarly, $R_u\in C$ and $n'_u\in N_C(R'_u)$ give $\ip{n'_u}{R_u-R'_u}\le0$. Hence $|R_s-R'_s|^2\le|\xi-\xi'|^2$. This completes the proof.
\end{proof}

\begin{proof}[Proof of Theorem \ref{thm:diagram}(b)]
Assume first $k\ge1$. Fix $T\in(0,1/e]$ and $S>0$. Let $s$ be such that $T\log(1/T)\le s\le S$.

\emph{Step 1: the slow time.} Define
\[
Y_u=X^\parallel_{u/T},\qquad\beta_u=\sqrt T\,W^\parallel_{u/T},\qquad u\ge0 .
\]
The process $\beta$ is continuous and has independent centred Gaussian increments. For $s'\le u$ the covariance of $\beta_u-\beta_{s'}$ is $T\,(u-s')/T=u-s'$ times the identity of $L_0$. So $\beta$ is a standard Brownian motion in $L_0$. We verify that $Y$ solves \eqref{eq:Y} with $\eps=T$. In integral form the first equation in \eqref{eq:AN} is
\[
X^\parallel_t=X^\parallel_0-\int_0^t\big(X^\parallel_r-P_CX^\parallel_r\big)\,dr+\sqrt{2T}\,W^\parallel_t\qquad(t\ge0).
\]
Let $s'\ge0$. We put $t=s'/T$ and substitute $r=u/T$, $dr=du/T$, in the integral:
\[
\int_0^{s'/T}\big(X^\parallel_r-P_CX^\parallel_r\big)\,dr=\frac1T\int_0^{s'}\big(X^\parallel_{u/T}-P_CX^\parallel_{u/T}\big)\,du=\frac1T\int_0^{s'}(Y_u-P_CY_u)\,du .
\]
Also $\sqrt{2T}\,W^\parallel_{s'/T}=\sqrt2\,\sqrt T\,W^\parallel_{s'/T}=\sqrt2\,\beta_{s'}$. Hence
\[
Y_{s'}=Y_0-\frac1T\int_0^{s'}(Y_u-P_CY_u)\,du+\sqrt2\,\beta_{s'}\qquad(s'\ge0),
\]
which is \eqref{eq:Y} with $\eps=T$ and $Y_0=\pi_LX_0$.

Let $(R,K)$ be the solution of \eqref{eq:rbm} driven by this $\beta$ with $\xi=P_CX_0$. The point $\xi$ is a function of $X_0$, so it is independent of $\beta$. By \eqref{eq:split}, $\xi=P_CY_0$. The law of $\xi$ is $(P_C)_\#\rho_0$. Then $\Law(R_s)=\mu_s$ and $\Law(X_{s/T})=\rho_{T,s/T}$, so $(X_{s/T},R_s)$ is a coupling of the two laws. Now $X_{s/T}-R_s=X^\perp_{s/T}+(Y_s-R_s)$ with $X^\perp_{s/T}\in L_0^\perp$ and $Y_s-R_s\in L_0$. By the Pythagoras theorem, $|X_{s/T}-R_s|^2=|X^\perp_{s/T}|^2+|Y_s-R_s|^2$. Therefore
\begin{equation}\label{eq:pyth}
\Wt\big(\rho_{T,s/T},\mu_s\big)^2\le\E|X_{s/T}-R_s|^2=\E|X^\perp_{s/T}|^2+\E|Y_s-R_s|^2 .
\end{equation}
We estimate the two terms on the right in turn.

\emph{Step 2: the normal part.} Since $s/T\ge\log(1/T)$ we have $e^{-2s/T}\le e^{-2\log(1/T)}=T^2$. So \eqref{eq:Nmoment} at the time $t=s/T$ gives
\begin{equation}\label{eq:normalpart}
\E|X^\perp_{s/T}|^2\le T^2m_0^2+(d-k)\,T .
\end{equation}

\emph{Step 3: the initial layer.} Put $\ell=\log(1/T)$ and $s_0=T\ell$. Then $\ell\ge1$ because $T\le1/e$. Also $s_0\le s$. In the time scale of the state equation, $s_0$ corresponds to the time $\ell$, that is, $Y_{s_0}=X^\parallel_\ell$. We also have $Y_0=X^\parallel_0$ and $R_0=\xi=P_CX^\parallel_0$. We compare $Y_{s_0}$ and $R_{s_0}$ through the points $\varphi_\ell(Y_0)$ and $P_CY_0=R_0$:
\[
|Y_{s_0}-R_{s_0}|\le|X^\parallel_\ell-\varphi_\ell(X^\parallel_0)|+|\varphi_\ell(X^\parallel_0)-P_CX^\parallel_0|+|R_0-R_{s_0}| .
\]
We bound the second moment of each of the three terms. For the first term, recall that by \eqref{eq:AN} the process $X^\parallel$ solves the state equation in the space $L$ of dimension $k$ with the Brownian motion $W^\parallel$, and that the noiseless flow of this equation is again \eqref{eq:flow}. So Lemma \ref{lem:ballistic} in $L$, with $k$ in place of $d$ and with $t=\ell$, gives
\[
\E|X^\parallel_\ell-\varphi_\ell(X^\parallel_0)|^2\le2kT\ell .
\]
For the second term we use \eqref{eq:flow}, the equality $e^{-\ell}=T$ and the last relation in \eqref{eq:split}:
\[
|\varphi_\ell(X^\parallel_0)-P_CX^\parallel_0|=e^{-\ell}\,d_C(X^\parallel_0)=T\,d_C(\pi_LX_0)\le T\,d_C(X_0).
\]
Squaring and taking expectations,
\[
\E|\varphi_\ell(X^\parallel_0)-P_CX^\parallel_0|^2\le T^2\,\E\,d_C(X_0)^2=T^2m_0^2 .
\]
For the third term, Lemma \ref{lem:penal}(ii) at the time $s_0=T\ell$ gives
\[
\E|R_0-R_{s_0}|^2\le2ks_0=2kT\ell .
\]
By the inequality $(\alpha_1+\alpha_2+\alpha_3)^2\le3(\alpha_1^2+\alpha_2^2+\alpha_3^2)$ we conclude that
\begin{equation}\label{eq:layer}
\E|Y_{s_0}-R_{s_0}|^2\le3\big(2kT\ell+T^2m_0^2+2kT\ell\big)=12\,k\,T\log\frac1T+3T^2m_0^2 .
\end{equation}

\emph{Step 4: the distance of $Y$ from $C$ after the initial layer.} Let $\zeta^\parallel$ be the envelope process driven by $W^\parallel$, that is, $d\zeta^\parallel_t=-\zeta^\parallel_t\,dt+\sqrt{2T}\,dW^\parallel_t$ and $\zeta^\parallel_0=0$. Theorem \ref{thm:envelope} applied to $X^\parallel$ (with $v=0$) and \eqref{eq:split} give $d_C(X^\parallel_t)\le e^{-t}d_C(X_0)+|\zeta^\parallel_t|$. We take $t=u/T$ with $s_0\le u\le S$. Then $u/T\ge\ell$, so $e^{-u/T}\le e^{-\ell}=T$. Thus
\[
d_C(Y_u)=d_C(X^\parallel_{u/T})\le T\,d_C(X_0)+|\zeta^\parallel_{u/T}| .
\]
Since $0\le u/T\le S/T$, it follows that
\[
\sup_{s_0\le u\le S}d_C(Y_u)\le T\,d_C(X_0)+\sup_{0\le t\le S/T}|\zeta^\parallel_t| .
\]
By Corollary \ref{cor:oumoment} in dimension $k$ with $N=S/T$,
\[
\E\sup_{0\le t\le S/T}|\zeta^\parallel_t|^2\le120\,T\Big(k+\log\Big(1+\frac ST\Big)\Big).
\]
We also have $\E\,d_C(X_0)^2=m_0^2$. So the triangle inequality for the norm of $L^2$ of the probability space gives
\begin{equation}\label{eq:supY}
\Big(\E\sup_{s_0\le u\le S}d_C(Y_u)^2\Big)^{1/2}\le Tm_0+\Big(120\,T\Big(k+\log\Big(1+\frac ST\Big)\Big)\Big)^{1/2} .
\end{equation}

It remains to combine the estimates of Steps 3 and 4. We use Lemma \ref{lem:penal}(i) with $\eps=T$ between $s_0$ and $s$. The process $|K|$ is nonnegative and nondecreasing. Since $s\le S$, we get $|K|_s-|K|_{s_0}\le|K|_S$. Also the supremum over $[s_0,s]$ is at most the supremum over $[s_0,S]$. Thus
\[
|Y_s-R_s|^2\le|Y_{s_0}-R_{s_0}|^2+2\sup_{s_0\le u\le S}d_C(Y_u)\,|K|_S .
\]
We take expectations. By the Schwarz inequality, by \eqref{eq:supY} and by Lemma \ref{lem:penal}(iii),
\[
\begin{aligned}
\E\Big[\sup_{s_0\le u\le S}d_C(Y_u)\,|K|_S\Big]&\le\Big(\E\sup_{s_0\le u\le S}d_C(Y_u)^2\Big)^{1/2}\big(\E|K|_S^2\big)^{1/2}\\
&\le\Big[Tm_0+\Big(120\,T\Big(k+\log\Big(1+\frac ST\Big)\Big)\Big)^{1/2}\Big]\Lambda_C(S).
\end{aligned}
\]
Here $\sqrt{120}<11$, because $120<121=11^2$. Together with \eqref{eq:layer} we obtain
\[
\E|Y_s-R_s|^2\le12\,k\,T\log\frac1T+3T^2m_0^2+2\Big[Tm_0+11\Big(T\Big(k+\log\Big(1+\frac ST\Big)\Big)\Big)^{1/2}\Big]\Lambda_C(S).
\]
We add the bounds. By \eqref{eq:pyth}, \eqref{eq:normalpart} and the last display,
\[
\begin{aligned}
\Wt\big(\rho_{T,s/T},\mu_s\big)^2\le{}&T^2m_0^2+(d-k)\,T+12\,k\,T\log\frac1T+3T^2m_0^2\\
&+2\Big[Tm_0+11\Big(T\Big(k+\log\Big(1+\frac ST\Big)\Big)\Big)^{1/2}\Big]\Lambda_C(S).
\end{aligned}
\]
The two terms with $m_0$ add up to $4T^2m_0^2$. Since $\log(1/T)\ge1$ we have $(d-k)T\le12(d-k)T\log(1/T)$. Therefore
\[
(d-k)\,T+12\,k\,T\log\frac1T\le12\,(d-k)\,T\log\frac1T+12\,k\,T\log\frac1T=12\,d\,T\log\frac1T .
\]
This is the bound in (b).

Now let $k=0$. Then $C=\{p\}$ and $L=\{p\}$, so that $\pi_Lx=p$, $X^\perp_t=X_t-p$ and $d_C(x)=|x-p|$. The process $X_t-p$ is an OU process and $\mu_s=\delta_p$. In every coupling of $\rho_{T,s/T}$ and $\delta_p$ the second variable is equal to $p$ almost surely. By \eqref{eq:Nmoment} with $k=0$ and by $e^{-2s/T}\le T^2$,
\[
\Wt(\rho_{T,s/T},\delta_p)^2=\E|X_{s/T}-p|^2\le T^2m_0^2+dT .
\]
Since $\log(1/T)\ge1$ and $\Lambda_C(S)=0$, the right side is at most the bound in (b). This completes the proof.
\end{proof}

\subsection{Long times}
For the limit $s=\infty$ we need the following fact. Reflected Brownian motion in $C$ converges to the uniform law at a rate which does not depend on the starting point. When $C$ has nonempty interior, this follows from the heat kernel bounds of \cite{BassHsu1991}. We give a proof which uses only the tools of this paper. These are the comparison of Lemma \ref{lem:penal} and the fact that the Gibbs states are logarithmically concave. The same proof prepares the estimate in total variation of Subsection \ref{subsec:tvmix}. We recall that there is a universal constant $c_*$ such that every logarithmically concave probability measure $\gamma$ on a Euclidean space satisfies the Poincar\'e inequality
\begin{equation}\label{eq:KLS}
\int\Big(f-\int f\,d\gamma\Big)^2d\gamma\le c_P(\gamma)\int|\nabla f|^2\,d\gamma,\qquad c_P(\gamma)\le c_*\int\Big|y-\int z\,\gamma(dz)\Big|^2\gamma(dy),
\end{equation}
for all smooth bounded $f$ with bounded gradient \cite[Theorem 1.2 and (1.8)]{Bobkov1999}. For the uniform law on a convex body the corresponding isoperimetric inequality is \cite[Theorem 5.1]{KLS1995}. Here $c_P(\gamma)$ is the smallest constant for which the first inequality holds.

The first lemma turns this inequality into an estimate for the state process. We use it again in Subsection \ref{subsec:tvmix} and in \cite[Section~3]{DarapaneniMeta}.

\begin{lemma}\label{lem:L2}
Let $C\subset\R^d$ be a nonempty compact convex set with diameter $D$, let $T>0$ and let $\nu$ be a probability measure on $\R^d$ which has a bounded density with compact support. Let $h_0$ be the density of $\nu$ with respect to $\gamma_T^C$, and let $\nu P_t$ be the law at time $t$ of the solution of \eqref{eq:SDEgen} started from $\nu$. Then for all $t\ge0$
\begin{equation}\label{eq:L2}
\TV\big(\nu P_t,\gamma_T^C\big)\le\frac12\,\exp\Big(-\frac{Tt}{c_*M_T}\Big)\,\|h_0-1\|_{L^2(\gamma_T^C)},
\end{equation}
where, for any fixed point $c$ of $C$,
\[
M_T:=\int|y-c|^2\,\gamma_T^C(dy)\le2D^2+2dT .
\]
\end{lemma}

\begin{proof}
We write $V=V_C$ and $\gamma=\gamma_T^C$. Since $C$ is convex, $\nabla V(x)=x-P_Cx$ by \ref{P4}. This map is nonexpansive by \ref{P2}.

First we smooth the potential. Let $\theta_n$ be the densities of Step 1 of the proof of Lemma \ref{lem:invariant} and let $V_n=V*\theta_n$. The function $V_n$ is smooth. It is convex because it is an average of translates of the convex function $V$. We have $\nabla V_n=(\nabla V)*\theta_n$, where $\theta_n$ is a probability density supported in the ball of radius $1/n$. Hence, for all $x,x'\in\R^d$,
\begin{align*}
|\nabla V_n(x)-\nabla V_n(x')|&\le\int|\nabla V(x-z)-\nabla V(x'-z)|\,\theta_n(z)\,dz\le|x-x'|,\\
|\nabla V_n(x)-\nabla V(x)|&\le\int|\nabla V(x-z)-\nabla V(x)|\,\theta_n(z)\,dz\le\int|z|\,\theta_n(z)\,dz\le\frac1n .
\end{align*}
So the map $\nabla V_n$ is nonexpansive, the second derivatives of $V_n$ are bounded and
\begin{equation}\label{eq:Vnconv}
|\nabla V_n-\nabla V|\le\frac1n ,\qquad|V_n-V|\le\frac{d_C+1}{n},
\end{equation}
the second inequality being \eqref{eq:Vn}.

Let $X^n$ be the solution of $dX^n=-\nabla V_n(X^n)\,dt+\sqrt{2T}\,dW$ with $X^n_0=X_0$ and $\Law(X_0)=\nu$. Let $\gamma_n=\cZ_n^{-1}e^{-V_n/T}dy$, where $\cZ_n=\int e^{-V_n/T}dy$. This process is a diffusion with a smooth potential whose second derivatives are bounded. Its generator is $\cL f=T\Delta f-\ip{\nabla V_n}{\nabla f}$. For a smooth function $g$ we compute
\[
\nabla\cdot\big(e^{-V_n/T}\nabla g\big)=e^{-V_n/T}\Big(\Delta g-\frac1T\ip{\nabla V_n}{\nabla g}\Big)=\frac1T\,e^{-V_n/T}\,\cL g .
\]
Let $f$ and $g$ be smooth with compact support. An integration by parts gives
\begin{equation}\label{eq:xs4bibp}
\begin{aligned}
\int f\,\cL g\,d\gamma_n&=\frac{T}{\cZ_n}\int f\,\nabla\cdot\big(e^{-V_n/T}\nabla g\big)\,dy=-\frac{T}{\cZ_n}\int\ip{\nabla f}{\nabla g}\,e^{-V_n/T}\,dy\\
&=-T\int\ip{\nabla f}{\nabla g}\,d\gamma_n .
\end{aligned}
\end{equation}
The right side does not change when $f$ and $g$ are exchanged. Hence the transition semigroup $(P^n_t)$ is symmetric in $L^2(\gamma_n)$. We recall from Step 2 of the proof of Lemma \ref{lem:invariant} that $\gamma_n$ is invariant. Now let $c_P$ be a constant for which $\gamma_n$ satisfies the Poincar\'e inequality. We claim that, for every $f\in L^2(\gamma_n)$ and $t\ge0$,
\[
\Big\|P^n_tf-\int f\,d\gamma_n\Big\|_{L^2(\gamma_n)}\le e^{-Tt/c_P}\Big\|f-\int f\,d\gamma_n\Big\|_{L^2(\gamma_n)} .
\]
This is the argument of \cite[Section 4.2]{BGL2014}. We write out the computation. Let $f$ be in the domain of the generator in $L^2(\gamma_n)$. Put $f_t=P^n_tf$ and $\bar f=\int f\,d\gamma_n$. Since $\gamma_n$ is invariant, $\int f_t\,d\gamma_n=\bar f$ for all $t$. Let
\[
v(t)=\int\big(f_t-\bar f\big)^2\,d\gamma_n
\]
be the variance of $f_t$. We have $\partial_tf_t=\cL f_t$ and $\int\cL f_t\,d\gamma_n=0$. We differentiate $v$ and use \eqref{eq:xs4bibp} and then the Poincar\'e inequality for $f_t$:
\[
v'(t)=2\int\big(f_t-\bar f\big)\,\cL f_t\,d\gamma_n=2\int f_t\,\cL f_t\,d\gamma_n=-2T\int|\nabla f_t|^2\,d\gamma_n\le-\frac{2T}{c_P}\,v(t).
\]
Here we use that the smooth functions with compact support form a core of the Dirichlet form, because the second derivatives of $V_n$ are bounded. So \eqref{eq:xs4bibp} and the Poincar\'e inequality extend from the functions in \eqref{eq:KLS} to the domain of the form. From the differential inequality we get
\[
\frac{d}{dt}\Big(e^{2Tt/c_P}\,v(t)\Big)=e^{2Tt/c_P}\Big(v'(t)+\frac{2T}{c_P}\,v(t)\Big)\le0,\qquad\text{hence}\qquad v(t)\le e^{-2Tt/c_P}\,v(0).
\]
Taking square roots proves the claim for $f$ in the domain of the generator. Since this domain is dense in $L^2(\gamma_n)$ and both sides of the claim are continuous in $f$ for the norm of $L^2(\gamma_n)$, the claim holds for all $f\in L^2(\gamma_n)$.

Since $V_n$ is convex, $\gamma_n$ is logarithmically concave. Let $\bar y_n=\int z\,\gamma_n(dz)$ be its barycentre. For the fixed point $c$ of $C$,
\[
\int|y-c|^2\,\gamma_n(dy)=\int|y-\bar y_n|^2\,\gamma_n(dy)+|\bar y_n-c|^2 ,
\]
because the cross term $2\int\ip{y-\bar y_n}{\bar y_n-c}\,\gamma_n(dy)$ is zero. So the barycentre minimises the second moment. Hence \eqref{eq:KLS} gives
\[
c_P(\gamma_n)\le c_*M_n,\qquad M_n:=\int|y-c|^2\,\gamma_n(dy).
\]
The law $\nu$ has the density $h_0$ with respect to $\gamma$. Hence it has the density $h_n=h_0\,(d\gamma/d\gamma_n)$ with respect to $\gamma_n$. The function $h_n$ is bounded with compact support and satisfies $\int h_n\,d\gamma_n=1$. Let $g$ be a bounded Borel function. Since $X^n_0$ has the law $h_n\,d\gamma_n$ and $P^n_t$ is symmetric,
\[
\E\,g(X^n_t)=\int\big(P^n_tg\big)\,h_n\,d\gamma_n=\int g\,\big(P^n_th_n\big)\,d\gamma_n .
\]
Thus $\Law(X^n_t)$ has the density $P^n_th_n$ with respect to $\gamma_n$. The total variation distance of two probability measures is half of the $L^1$ distance of their densities. By the Schwarz inequality,
\[
\frac12\int\big|P^n_th_n-1\big|\,d\gamma_n\le\frac12\Big(\int\big|P^n_th_n-1\big|^2\,d\gamma_n\Big)^{1/2}.
\]
We apply the claim to $f=h_n$, for which $\int f\,d\gamma_n=1$, with $c_P=c_P(\gamma_n)\le c_*M_n$. This gives
\begin{equation}\label{eq:tvpenn}
\TV\big(\Law(X^n_t),\gamma_n\big)\le\frac12\,\|P^n_th_n-1\|_{L^2(\gamma_n)}\le\frac12\,\exp\Big(-\frac{Tt}{c_*M_n}\Big)\,\|h_n-1\|_{L^2(\gamma_n)} .
\end{equation}

We now remove the smoothing. The processes $X^n$ and $X$ are driven by the same Brownian motion and start at the same point. Subtracting the two equations, we get
\[
X^n_t-X_t=-\int_0^t\big(\nabla V_n(X^n_u)-\nabla V(X_u)\big)\,du .
\]
By \eqref{eq:Vnconv} and since $\nabla V$ is nonexpansive,
\begin{align*}
|\nabla V_n(X^n_u)-\nabla V(X_u)|&\le|\nabla V_n(X^n_u)-\nabla V(X^n_u)|+|\nabla V(X^n_u)-\nabla V(X_u)|\\
&\le\frac1n+|X^n_u-X_u| .
\end{align*}
Therefore
\[
|X^n_t-X_t|\le\int_0^t\Big(\frac1n+|X^n_u-X_u|\Big)\,du,\qquad\text{hence}\qquad|X^n_t-X_t|\le\frac tn\,e^{t}
\]
by the Gronwall lemma. Indeed, put $g(t)=|X^n_t-X_t|$ and $G(t)=\int_0^tg(u)\,du$. The integral inequality says that $G'(t)=g(t)\le t/n+G(t)$. Hence
\[
\frac{d}{dt}\big(e^{-t}G(t)\big)=e^{-t}\big(G'(t)-G(t)\big)\le\frac tn\,e^{-t},\qquad G(t)\le\frac{e^{t}}n\int_0^tu\,e^{-u}\,du=\frac{e^t-1-t}{n},
\]
and so
\[
g(t)\le\frac tn+G(t)\le\frac{e^t-1}{n}\le\frac tn\,e^t .
\]
In the last step we used $e^t-1=\int_0^te^u\,du\le t\,e^t$. So $X^n_t\to X_t$ on every path. Hence $\Law(X^n_t)\to\nu P_t$ weakly as $n\to\infty$.

By Step 1 of the proof of Lemma \ref{lem:invariant} we have $e^{-V_n/T}\le e^{2/T}e^{-d_C^2/4T}$ and $\cZ_n\to\cZ_T(C)$. So the densities of $\gamma_n$ are bounded by a constant multiple of $e^{-d_C^2/4T}$. By \eqref{eq:Vnconv} they converge pointwise to the density of $\gamma$. The function $|y-c|^2e^{-d_C(y)^2/4T}$ is integrable. By dominated convergence, $\gamma_n\to\gamma$ in total variation and $M_n\to M_T$. The ratio
\[
\frac{d\gamma}{d\gamma_n}=\frac{\cZ_n}{\cZ_T(C)}\,e^{(V_n-V)/T}
\]
tends to one uniformly on compact sets by \eqref{eq:Vnconv}, because $d_C$ is bounded on compact sets. Since $\int h_n\,d\gamma_n=1$,
\[
\|h_n-1\|_{L^2(\gamma_n)}^2=\int h_n^2\,d\gamma_n-2\int h_n\,d\gamma_n+1=\int h_n^2\,d\gamma_n-1 ,
\]
and $h_n^2\,d\gamma_n=h_0^2\,(d\gamma/d\gamma_n)^2\,d\gamma_n=h_0^2\,(d\gamma/d\gamma_n)\,d\gamma$. The function $h_0$ is bounded with compact support, so
\[
\|h_n-1\|_{L^2(\gamma_n)}^2=\int h_0^2\,\frac{d\gamma}{d\gamma_n}\,d\gamma-1\longrightarrow\int h_0^2\,d\gamma-1=\|h_0-1\|_{L^2(\gamma)}^2 .
\]
The last equality holds for the same reason as the first, since $\int h_0\,d\gamma=1$. We now let $n\to\infty$ in \eqref{eq:tvpenn}. On the left we use \eqref{eq:TVlsc} for the weakly convergent sequences $\Law(X^n_t)\to\nu P_t$ and $\gamma_n\to\gamma$. On the right $M_n\to M_T$ and the norms converge. We get \eqref{eq:L2}.

It remains to bound $M_T$. For every $y$ we have $|y-c|\le|y-P_Cy|+|P_Cy-c|\le d_C(y)+D$. Hence $|y-c|^2\le2D^2+2d_C(y)^2$. Using Corollary \ref{cor:violation},
\[
M_T\le2D^2+2\int d_C^2\,d\gamma\le2D^2+2dT .
\]
This completes the proof.
\end{proof}

\begin{lemma}\label{lem:ergodic}
Let $k\ge1$. For every probability measure $\nu$ on $C$ and every $s\ge0$,
\[
\Wt\big(\mu_s[\nu],u_C\big)\le2D\exp\Big(-\frac{s}{(k+4)\,c_*D^2}\Big).
\]
\end{lemma}

\begin{proof}
We identify $L$ with $\R^k$, so that $C$ is a compact convex subset of $\R^k$ with nonempty interior. In this proof $\Leb$ is the Lebesgue measure of $\R^k$. Since $k\ge1$, we have $D>0$ and $\Leb(C)>0$. For $s=0$ the right side is $2D$ and the two measures are carried by $C$, so there is nothing to prove. Let $s>0$.

\emph{Step 1: a regular initial law.} Fix $x\in C$ and $r\in(0,D]$, and let $A=\{y\in C:|y-x|\le r\}$. Consider the set
\[
x+\frac rD\,(C-x)=\Big\{\Big(1-\frac rD\Big)x+\frac rD\,c\ :\ c\in C\Big\}.
\]
Its points are convex combinations of two points of $C$, because $0<r/D\le1$. So it is contained in $C$. It is also contained in the ball of radius $r$ around $x$, because
\[
\Big|\Big(1-\frac rD\Big)x+\frac rD\,c-x\Big|=\frac rD\,|c-x|\le r\qquad(c\in C).
\]
Thus the set is contained in $A$. It is the image of $C$ under a homothety of ratio $r/D$, so its Lebesgue measure is $(r/D)^k\,\Leb(C)$. Hence
\begin{equation}\label{eq:volA}
\Leb(A)\ge\Big(\frac rD\Big)^k\Leb(C) .
\end{equation}
Let $\nu_r$ be the uniform probability measure on $A$.

\emph{Step 2: the penalised process.} For $\eps>0$ let $Y^\eps$ be the solution of \eqref{eq:Y} with $\Law(Y^\eps_0)=\nu_r$. We change the time by $t=s/\eps$. Put $\widetilde X_t=Y^\eps_{\eps t}$ and $\widetilde W_t=\eps^{-1/2}\beta_{\eps t}$, which is a standard Brownian motion. With the substitution $u=\eps v$ in \eqref{eq:Y},
\[
\widetilde X_t=Y^\eps_0-\frac1\eps\int_0^{\eps t}\big(Y^\eps_u-P_CY^\eps_u\big)\,du+\sqrt2\,\beta_{\eps t}=\widetilde X_0-\int_0^t\big(\widetilde X_v-P_C\widetilde X_v\big)\,dv+\sqrt{2\eps}\,\widetilde W_t .
\]
So $Y^\eps_s=\widetilde X_{s/\eps}$ is the state process of \eqref{eq:SDEgen} in $\R^k$ at temperature $\eps$ and time $s/\eps$. Let $\gamma_\eps=\gamma_\eps^C$ be the Gibbs state in $\R^k$. Its density equals $1/\cZ_\eps(C)$ on $C$. The density of $\nu_r$ is $\one_A/\Leb(A)$ with $A\subset C$. Thus $\nu_r$ has the density $h_0=\cZ_\eps(C)\one_A/\Leb(A)$ with respect to $\gamma_\eps$. We compute its norm. First
\[
\int h_0^2\,d\gamma_\eps=\int_A\frac{\cZ_\eps(C)^2}{\Leb(A)^2}\cdot\frac{1}{\cZ_\eps(C)}\,dy=\frac{\cZ_\eps(C)}{\Leb(A)} ,
\]
and then, since $\int h_0\,d\gamma_\eps=1$,
\[
\|h_0-1\|_{L^2(\gamma_\eps)}^2=\int h_0^2\,d\gamma_\eps-1=\frac{\cZ_\eps(C)}{\Leb(A)}-1\le\frac{\cZ_\eps(C)}{\Leb(A)} .
\]
We apply Lemma \ref{lem:L2} in dimension $k$, at temperature $\eps$ and time $s/\eps$. The exponent in \eqref{eq:L2} is then $\eps\cdot(s/\eps)/(c_*M_\eps)=s/(c_*M_\eps)$. Hence
\begin{equation}\label{eq:tvpen}
\TV\big(\Law(Y^\eps_s),\gamma_\eps\big)\le\frac12\,\exp\Big(-\frac{s}{c_*M_\eps}\Big)\Big(\frac{\cZ_\eps(C)}{\Leb(A)}\Big)^{1/2},\qquad M_\eps\le2\big(D^2+k\eps\big).
\end{equation}

\emph{Step 3: the limit $\eps\to0$.} Let $(R,K)$ solve \eqref{eq:rbm} with the same $\beta$ and with $\xi=Y^\eps_0\in C$. Lemma \ref{lem:penal}(i) with $s_0=0$ and $Y_0^\eps=R_0$ gives
\[
|Y^\eps_s-R_s|^2\le2\sup_{0\le u\le s}d_C(Y^\eps_u)\,|K|_s .
\]
The process $Y^\eps_s$ is the state process at temperature $\eps$ and time $s/\eps$, and $d_C(Y^\eps_0)=0$. Let $\zeta$ be the envelope process at temperature $\eps$. Theorem \ref{thm:envelope} gives $d_C(Y^\eps_u)\le|\zeta_{u/\eps}|$. So $\sup_{u\le s}d_C(Y^\eps_u)\le\sup_{t\le s/\eps}|\zeta_t|$. We take expectations and use the Schwarz inequality:
\[
\E|Y^\eps_s-R_s|^2\le2\Big(\E\sup_{0\le t\le s/\eps}|\zeta_t|^2\Big)^{1/2}\big(\E|K|_s^2\big)^{1/2} .
\]
We bound the first factor by Corollary \ref{cor:oumoment} in dimension $k$ with $T=\eps$ and $N=s/\eps$, and the second by Lemma \ref{lem:penal}(iii). Since $\eps\log(1+s/\eps)\to0$ as $\eps\to0$, we obtain
\[
\E|Y^\eps_s-R_s|^2\le2\Big(120\,\eps\Big(k+\log\Big(1+\frac s\eps\Big)\Big)\Big)^{1/2}\Lambda_C(s)\longrightarrow0\qquad(\eps\to0).
\]
The law of $R_s$ is $\mu_s[\nu_r]$ for every $\eps$. Hence $\Law(Y^\eps_s)\to\mu_s[\nu_r]$ weakly. Also $\gamma_\eps\to u_C$ weakly and $\cZ_\eps(C)\to\Leb(C)$ by Theorem \ref{thm:gibbs} and Corollary \ref{cor:zerotemp}. Finally $M_\eps\le2(D^2+k\eps)$, so that
\[
\exp\Big(-\frac{s}{c_*M_\eps}\Big)\le\exp\Big(-\frac{s}{2c_*(D^2+k\eps)}\Big)\longrightarrow e^{-s/(2c_*D^2)} .
\]
Letting $\eps\to0$ in \eqref{eq:tvpen} and using \eqref{eq:TVlsc} and \eqref{eq:volA},
\begin{equation}\label{eq:tvrbm}
\TV\big(\mu_s[\nu_r],u_C\big)\le\frac12\,e^{-s/(2c_*D^2)}\Big(\frac{\Leb(C)}{\Leb(A)}\Big)^{1/2}\le\frac12\,e^{-s/(2c_*D^2)}\Big(\frac Dr\Big)^{k/2}.
\end{equation}

\emph{Step 4: from $\nu_r$ to a point.} Let $R^x$ and $R'$ be the solutions of \eqref{eq:rbm} with the same $\beta$, started from $x$ and from a random point with law $\nu_r$. This random point lies in $A$, so its distance from $x$ is at most $r$. By Lemma \ref{lem:penal}(iv), $|R^x_s-R'_s|\le r$. Hence $\Wt(\mu_s[\delta_x],\mu_s[\nu_r])\le r$. The measures $\mu_s[\nu_r]$ and $u_C$ are carried by $C$, which has diameter $D$. By the triangle inequality, by \eqref{eq:WTV} and by \eqref{eq:tvrbm},
\[
\Wt\big(\mu_s[\delta_x],u_C\big)\le r+D\Big(\frac12\,e^{-s/(2c_*D^2)}\Big(\frac Dr\Big)^{k/2}\Big)^{1/2}\le r+D\,e^{-s/(4c_*D^2)}\Big(\frac Dr\Big)^{k/4}.
\]
We choose $r$ so that the two terms on the right are equal. The equation
\[
r=D\,e^{-s/(4c_*D^2)}\Big(\frac Dr\Big)^{k/4}
\]
is the same as
\[
\Big(\frac rD\Big)^{(k+4)/4}=e^{-s/(4c_*D^2)},\qquad\text{that is,}\qquad\frac rD=\exp\Big(-\frac{s}{4c_*D^2}\cdot\frac{4}{k+4}\Big).
\]
So we choose $r=D\exp(-s/((k+4)c_*D^2))$, which belongs to $(0,D]$. Then the second term is
\begin{align*}
D\exp\Big(-\frac{s}{4c_*D^2}+\frac k4\cdot\frac{s}{(k+4)c_*D^2}\Big)&=D\exp\Big(-\frac{s}{c_*D^2}\cdot\frac{(k+4)-k}{4(k+4)}\Big)\\
&=D\exp\Big(-\frac{s}{(k+4)c_*D^2}\Big),
\end{align*}
which is equal to $r$. Thus the sum of the two terms is $2D\exp(-s/((k+4)c_*D^2))$. This proves the lemma for $\nu=\delta_x$. For a general $\nu$ we have $\mu_s[\nu]=\int\mu_s[\delta_x]\,\nu(dx)$. The convexity of $\Wt^2$ then gives
\[
\Wt\big(\mu_s[\nu],u_C\big)^2\le\int\Wt\big(\mu_s[\delta_x],u_C\big)^2\,\nu(dx)\le\Big(2D\exp\Big(-\frac{s}{(k+4)c_*D^2}\Big)\Big)^2 .
\]
The claim follows.
\end{proof}

\begin{proof}[Proof of Theorem \ref{thm:diagram}(c)]
By Lemma \ref{lem:invariant}, $\rho_{T,t}(A)\to\gamma_T^C(A)$ for every Borel set $A$ as $t\to\infty$. We call this setwise convergence. It implies weak convergence. Convergence in $\Wt$ is the same as weak convergence together with convergence of the second moments. So it is enough to show that the variables $|X_t|^2$, $t\ge0$, are uniformly integrable. Fix $c\in C$. By \eqref{eq:env3},
\[
|X_t-c|\le d_C(X_t)+D\le e^{-t}d_C(X_0)+|\zeta_t|+D .
\]
Since $|X_t|\le|c|+|X_t-c|$ and $(z_1+z_2+z_3)^2\le3z_1^2+3z_2^2+3z_3^2$, we get
\[
|X_t|^2\le3\big(|c|+D\big)^2+3\,d_C(X_0)^2+3\,|\zeta_t|^2 .
\]
The first term is a constant. The second term is an integrable variable which does not depend on $t$, because $\E\,d_C(X_0)^2=m_0^2<\infty$. In the third term, $|\zeta_t|$ is the norm of a Gaussian vector whose covariance is $T(1-e^{-2t})$ times the identity by \eqref{eq:oucov}. This covariance is bounded by $T$ times the identity. Hence
\[
\E|\zeta_t|^4=d(d+2)\,T^2\big(1-e^{-2t}\big)^2\le d(d+2)\,T^2 ,
\]
so the variables $|\zeta_t|^2$ are bounded in $L^2$ and therefore uniformly integrable. It follows that $\{|X_t|^2\}_{t\ge0}$ is uniformly integrable. Hence the second moments of $\rho_{T,t}$ converge to the second moment of $\gamma_T^C$. The statement about $\gamma_T^C$ is Corollary \ref{cor:zerotemp}.
\end{proof}

\begin{proof}[Proof of Theorem \ref{thm:diagram}(d)]
We consider the three cases.

\emph{Case $s=0$.} Here $T_nt_n\to0$ and $t_n\to\infty$. By (a),
\[
\Wt\big(\rho_{T_n,t_n},(P_C)_\#\rho_0\big)\le\sqrt{2dT_nt_n}+e^{-t_n}m_0\longrightarrow0 .
\]

\emph{Case $0<s<\infty$.} Put $s_n=T_nt_n$ and $S=s+1$. Then $s_n\to s>0$. Also $T_n\log(1/T_n)\to0$, because $T_n\to0$. So for large $n$ we have $T_n\le1/e$ and
\[
T_n\log(1/T_n)\le s_n\le s+1 .
\]
For such $n$ we may apply (b) with $T=T_n$ and with $s_n$ in place of $s$. Since $s_n/T_n=t_n$, it gives
\[
\begin{aligned}
\Wt\big(\rho_{T_n,t_n},\mu_{s_n}\big)^2\le{}&4T_n^2m_0^2+12\,d\,T_n\log\frac1{T_n}\\
&+2\Big[T_nm_0+11\Big(T_n\Big(k+\log\Big(1+\frac S{T_n}\Big)\Big)\Big)^{1/2}\Big]\Lambda_C(S).
\end{aligned}
\]
The numbers $S$, $m_0$ and $\Lambda_C(S)$ are fixed, while $T_n\to0$, $T_n\log(1/T_n)\to0$, $T_n\log(1+S/T_n)\to0$. Thus the right side of (b) tends to zero as $T\to0$ for fixed $S$ and $m_0$, so that $\Wt(\rho_{T_n,t_n},\mu_{s_n})\to0$. It remains to compare $\mu_{s_n}$ and $\mu_s$. If $k=0$ they are equal. Let $k\ge1$ and let $R$ be the reflected Brownian motion with $\Law(R_0)=(P_C)_\#\rho_0$. The pair $(R_{s_n},R_s)$ is a coupling of $\mu_{s_n}$ and $\mu_s$, so
\[
\Wt(\mu_{s_n},\mu_s)^2\le\E|R_{s_n}-R_s|^2\longrightarrow0 .
\]
The limit holds by dominated convergence, because the paths of $R$ are continuous and $|R_{s_n}-R_s|\le D$. By the triangle inequality for $\Wt$ we get $\Wt(\rho_{T_n,t_n},\mu_s)\to0$.

\emph{Case $s=\infty$.} First let $k=0$, so that $C=\{p\}$, $u_C=\delta_p$ and $X^\perp_t=X_t-p$. Then the claim follows from \eqref{eq:Nmoment}:
\[
\Wt\big(\rho_{T_n,t_n},u_C\big)^2=\E|X_{t_n}-p|^2\le e^{-2t_n}m_0^2+d\,T_n\longrightarrow0 .
\]
Now let $k\ge1$. Fix $S\ge1$. For large $n$ we have $T_n\le1/e$, $T_n\log(1/T_n)\le S$ and $t_n\ge S/T_n$, because $T_n\to0$ and $T_nt_n\to\infty$. Put $t_n'=t_n-S/T_n\ge0$. By the Markov property, $\rho_{T_n,t_n}$ is the law at time $S/T_n$ of the state process started from the law $\rho_n'=\rho_{T_n,t_n'}$. We apply (b) to this process with $s=S$. Its constant $m_0$ is replaced by
\[
m_n'=\Big(\int d_C^2\,d\rho_n'\Big)^{1/2}=\big(\E\,d_C(X_{t_n'})^2\big)^{1/2}\le\big(2m_0^2+2dT_n\big)^{1/2},
\]
where we used \eqref{eq:env3} and $\E|\zeta_t|^2\le dT$. In detail, for every $t\ge0$,
\[
\E\,d_C(X_t)^2\le\E\big(e^{-t}d_C(X_0)+|\zeta_t|\big)^2\le2e^{-2t}\,\E\,d_C(X_0)^2+2\,\E|\zeta_t|^2\le2m_0^2+2dT .
\]
This also shows that $\rho_n'$ has a finite second moment. The image of $\rho_n'$ under $P_C$ is the initial law of the reflected Brownian motion in (b). So (b) reads
\[
\begin{aligned}
\Wt\big(\rho_{T_n,t_n},\ \mu_S[(P_C)_\#\rho_n']\big)^2\le{}&4T_n^2(m_n')^2+12\,d\,T_n\log\frac1{T_n}\\
&+2\Big[T_nm_n'+11\Big(T_n\Big(k+\log\Big(1+\frac S{T_n}\Big)\Big)\Big)^{1/2}\Big]\Lambda_C(S).
\end{aligned}
\]
Since $m_n'$ is bounded, the right side tends to zero as in the second case. Hence
\[
\Wt\big(\rho_{T_n,t_n},\ \mu_S[(P_C)_\#\rho_n']\big)\longrightarrow0\qquad(n\to\infty).
\]
By Lemma \ref{lem:ergodic}, $\Wt(\mu_S[(P_C)_\#\rho_n'],u_C)\le2D\exp(-S/((k+4)c_*D^2))$. Hence, by the triangle inequality,
\[
\limsup_{n\to\infty}\Wt\big(\rho_{T_n,t_n},u_C\big)\le2D\exp\Big(-\frac{S}{(k+4)c_*D^2}\Big).
\]
The left side does not depend on $S$. Letting $S\to\infty$ finishes the proof.
\end{proof}

\subsection{The commutation criterion}

\begin{corollary}\label{cor:commute}
Let $C$ and $\rho_0$ be as in Theorem \ref{thm:diagram}.
\begin{enumerate}[label=\textup{(\roman*)}]
\item The iterated limits exist in $\Wt$ and
\[
\lim_{T\to0}\lim_{t\to\infty}\rho_{T,t}=u_C,\qquad\lim_{t\to\infty}\lim_{T\to0}\rho_{T,t}=(P_C)_\#\rho_0 .
\]
\item The set of all limit points of $\rho_{T,t}$ as $T\to0$ and $t\to\infty$ is the curve $\{\mu_s:0\le s\le\infty\}$.
\item The following three statements are equivalent. First, the two iterated limits in \textup{(i)} coincide. Second, the joint limit of $\rho_{T,t}$ as $T\to0$, $t\to\infty$ exists. Third, $(P_C)_\#\rho_0=u_C$.
\item If $C$ has nonempty interior, then $(P_C)_\#\rho_0=u_C$ holds only for $\rho_0=u_C$.
\end{enumerate}
\end{corollary}

\begin{proof}
(i) The first limit is given by Theorem \ref{thm:diagram}(c). The inner limit is $\gamma_T^C$. As $T\to0$ we have $\gamma_T^C\to u_C$. For the second, Theorem \ref{thm:diagram}(a) gives $\Wt(\rho_{T,t},(\varphi_t)_\#\rho_0)\le\sqrt{2dTt}$. Its proof gives $\Wt((\varphi_t)_\#\rho_0,(P_C)_\#\rho_0)^2\le e^{-2t}m_0^2$. So, for each $t$,
\[
\lim_{T\to0}\rho_{T,t}=(\varphi_t)_\#\rho_0\qquad\text{and}\qquad\Wt\big((\varphi_t)_\#\rho_0,(P_C)_\#\rho_0\big)\le e^{-t}m_0 .
\]
The last bound tends to zero as $t\to\infty$.

(ii) Let $(T_n,t_n)$ be a sequence with $T_n\to0$, $t_n\to\infty$. Since $[0,\infty]$ is compact, there is a subsequence along which $T_nt_n$ converges to some $s\in[0,\infty]$. By Theorem \ref{thm:diagram}(d) the limit of $\rho_{T_n,t_n}$ along this subsequence is $\mu_s$. So every limit point is on the curve. Conversely, every $\mu_s$ is obtained. Take $t_n=n$ and $T_n=s/n$ for $0<s<\infty$, $T_n=n^{-2}$ for $s=0$ and $T_n=n^{-1/2}$ for $s=\infty$. Then $T_nt_n$ is equal to $s$, to $1/n$ and to $\sqrt n$ respectively, so Theorem \ref{thm:diagram}(d) applies.

(iii) By (i), the two iterated limits coincide exactly when $(P_C)_\#\rho_0=u_C$. Assume that the joint limit exists. Then by (ii) all $\mu_s$ are equal. So $\mu_0=\mu_\infty$, that is, $(P_C)_\#\rho_0=u_C$. Conversely, assume that $(P_C)_\#\rho_0=u_C$. We show below that then $\mu_s=u_C$ for all $s$. This implies that the joint limit exists and is equal to $u_C$. Indeed, if $\rho_{T,t}$ did not converge to $u_C$, there would be a number $\eta>0$ and a sequence with $T_n\to0$, $t_n\to\infty$ and $\Wt(\rho_{T_n,t_n},u_C)\ge\eta$ for all $n$. By the proof of (ii) a subsequence converges to some $\mu_s=u_C$, which is a contradiction.

We now prove that $\mu_s=u_C$ for all $s$. The assumption $(P_C)_\#\rho_0=u_C$ means that $\mu_s=\mu_s[u_C]$. For $k=0$ the claim is clear. Let $k\ge1$. We need two properties of the family $\mu_s[\nu]$. The first is the semigroup property
\begin{equation}\label{eq:xs4bsemi}
\mu_{s+s'}[\nu]=\mu_{s'}\big[\mu_s[\nu]\big]\qquad(s,s'\ge0).
\end{equation}
To see this, let $(R,K)$ solve \eqref{eq:rbm} with $\Law(\xi)=\nu$. Then the pair $(R_{s+u},K_{s+u}-K_s)_{u\ge0}$ solves \eqref{eq:rbm} with the initial point $R_s$ and the Brownian motion $(\beta_{s+u}-\beta_s)_{u\ge0}$, which is independent of $R_s$. By the uniqueness in \eqref{eq:rbm}, the law of $R_{s+s'}$ is $\mu_{s'}[\Law(R_s)]$. This is \eqref{eq:xs4bsemi}. The second property is
\begin{equation}\label{eq:xs4bcontr}
\Wt\big(\mu_{s'}[\nu],\mu_{s'}[\nu']\big)\le\Wt(\nu,\nu') .
\end{equation}
For the proof, let $(\xi,\xi')$ be any coupling of $\nu$ and $\nu'$ which is independent of $\beta$, and let $R$ and $R'$ be the solutions of \eqref{eq:rbm} started from $\xi$ and $\xi'$. Lemma \ref{lem:penal}(iv) gives $|R_{s'}-R'_{s'}|\le|\xi-\xi'|$. Hence $\Wt(\mu_{s'}[\nu],\mu_{s'}[\nu'])^2\le\E|\xi-\xi'|^2$. Taking the infimum over the couplings gives \eqref{eq:xs4bcontr}.

Now fix $s'\ge0$ and let $s\to\infty$ in $\mu_{s+s'}[u_C]=\mu_{s'}[\mu_s[u_C]]$. By Lemma \ref{lem:ergodic} the left side tends to $u_C$. The right side tends to $\mu_{s'}[u_C]$ by \eqref{eq:xs4bcontr} and Lemma \ref{lem:ergodic}. In formulas,
\begin{align*}
\Wt\big(\mu_{s'}[u_C],u_C\big)&\le\Wt\big(\mu_{s'}[u_C],\mu_{s'}[\mu_s[u_C]]\big)+\Wt\big(\mu_{s+s'}[u_C],u_C\big)\\
&\le\Wt\big(u_C,\mu_s[u_C]\big)+\Wt\big(\mu_{s+s'}[u_C],u_C\big)\le4D\exp\Big(-\frac{s}{(k+4)c_*D^2}\Big),
\end{align*}
and the right side tends to zero as $s\to\infty$. Hence $\mu_{s'}[u_C]=u_C$ for every $s'$.

(iv) Let $k=d$ and $(P_C)_\#\rho_0=u_C$. The boundary $\partial C$ of the convex set $C$ has Lebesgue measure zero, so $u_C(\partial C)=0$. A point $x$ with $P_Cx$ in the interior of $C$ belongs to the interior of $C$, since the projection maps a point outside $C$ to the boundary and a point of $C$ to itself. So the set of all $x$ which are not in the interior of $C$ is $P_C^{-1}(\partial C)$. Hence
\[
\rho_0\big(\R^d\setminus\mathrm{int}\,C\big)=\rho_0\big(P_C^{-1}(\partial C)\big)=u_C(\partial C)=0 .
\]
Thus $\rho_0$ is carried by the interior of $C$, where $P_C$ is the identity, and $\rho_0=(P_C)_\#\rho_0=u_C$.
\end{proof}

\begin{example}\label{ex:interval}
Let $d=1$, $C=[0,1]$ and $\rho_0=\delta_{x_0}$ with $x_0>1$. Then $(P_C)_\#\rho_0=\delta_1$. The measure $\mu_s$ is the law at time $s$ of the reflected Brownian motion on $[0,1]$ with generator $d^2/dy^2$ started at $1$. For $s>0$ it has the density
\[
1+2\sum_{n=1}^\infty(-1)^n\cos(n\pi y)\,e^{-n^2\pi^2s}\qquad(0\le y\le1).
\]
This formula comes from the eigenfunctions of the Neumann problem for $d^2/dy^2$ on $[0,1]$. The functions $e_0=1$ and $e_n(y)=\sqrt2\cos(n\pi y)$, $n\ge1$, satisfy
\[
e_n''=-n^2\pi^2\,e_n,\qquad e_n'(0)=e_n'(1)=0 ,
\]
and they form an orthonormal basis of $L^2[0,1]$. The transition density of the reflected Brownian motion is the kernel of the semigroup generated by $d^2/dy^2$ with the Neumann boundary condition, that is,
\[
p_s(x,y)=\sum_{n=0}^\infty e^{-n^2\pi^2s}\,e_n(x)\,e_n(y)=1+2\sum_{n=1}^\infty e^{-n^2\pi^2s}\cos(n\pi x)\cos(n\pi y).
\]
For $x=1$ we have $\cos(n\pi)=(-1)^n$, which gives the density above. As $s\to\infty$ the sum over $n\ge1$ tends to zero uniformly in $y$, so the density tends to $1$.

The curve $\mu_s$ joins $\delta_1$ at $s=0$ to the uniform law at $s=\infty$. So the two iterated limits are different. By Corollary \ref{cor:commute}(iv) the joint limit exists only when $\rho_0$ is the uniform law on $[0,1]$. When $k<d$ there are more initial laws for which the limits commute. These are all laws on $\R^d$ whose image under $P_C$ is uniform on $C$, for example the product of $u_C$ with any law on $L_0^\perp$ having a second moment.
\end{example}

\begin{remark}\label{rem:rate}
Few initial laws satisfy the condition $(P_C)_\#\rho_0=u_C$. So the answer to the question \eqref{eq:question} is negative in most cases, even when the set is convex. The limit depends on the way in which $(T,t)$ tends to $(0,\infty)$ only through the number $s=\lim tT$. Figure \ref{fig:diagram} summarises the ranges of times, together with those of the companion paper \cite{DarapaneniMeta}.

The bound in part (b) is meant for $T\to0$. We did not try to optimise its constants. It is of order $(T\log(1/T))^{1/2}$ for the square of the distance, that is, the rate is $T^{1/4}$ up to a logarithm. When $k<d$ the rate cannot be better than $T^{1/2}$. The reason is that $\mu_s$ is carried by $L$. So in every coupling the distance of the two variables is at least the distance of the first one from $L$. Therefore, by \eqref{eq:Nmoment},
\[
\Wt\big(\rho_{T,s/T},\mu_s\big)^2\ge\E|X^\perp_{s/T}|^2\ge(d-k)\,T\,\big(1-e^{-2s/T}\big).
\]
For polyhedra the rate $T^{1/2}$, up to a logarithm, is known for the penalised equation \cite{Slominski2013}. We do not know the optimal rate for general convex sets.
\end{remark}

\begin{figure}[ht]
\centering
\begin{tikzpicture}[scale=0.95]
\draw[->] (0,0) -- (9.3,0) node[below] {$\log t$};
\draw[->] (0,0) -- (0,5.6) node[above] {$\log(1/T)$};
\draw[thick] (0,0) -- (5.3,5.3) node[above right] {$tT=s$};
\draw[thick,dashed,domain=1.0:9,samples=60] plot (\x,{1.75*ln(\x)});
\node[right] at (7.7,4.75) {$t=e^{c/T}$};
\draw[->,thin] (8.2,4.5) -- (8.0,3.75);
\node at (1.5,4.2) {$(P_C)_\#\rho_0$};
\node at (1.5,3.6) {\footnotesize short times};
\node at (6.0,4.35) {$\sum_ia_iu_{C_i}$};
\node at (6.0,3.8) {\footnotesize metastable};
\node at (6.9,1.4) {$\sum_iw_iu_{C_i}$};
\node at (6.9,0.8) {\footnotesize equilibrium};
\end{tikzpicture}
\caption{The limits of $\rho_{T,t}$ in the plane of $\log t$ and $\log(1/T)$. Above the line $tT=s$ the limit is the projection of the initial law. On the line it is the law $\mu_s$ of reflected Brownian motion. Below the line it is uniform on $C$ when $C$ is convex (Theorem \ref{thm:diagram}). The dashed curve concerns a disjoint union of convex sets $C_i$, which is treated in the companion paper \cite{DarapaneniMeta}. Here we use the notation of \cite[Section~3]{DarapaneniMeta}. For such a set the region below the line is divided. The basin weights $a_i=\rho_0(B_i)$ persist for $t\le e^{c/T}$, $c<\min_i\delta_i^2/8$ \cite[Theorem~3.6]{DarapaneniMeta}. The equilibrium weights $w_i$ are reached for $t\ge e^{c/T}$, $c>\lfloor m/2\rfloor\Delta_*^2/8$ \cite[Theorem~3.14]{DarapaneniMeta}. For two components the two thresholds coincide.}
\label{fig:diagram}
\end{figure}
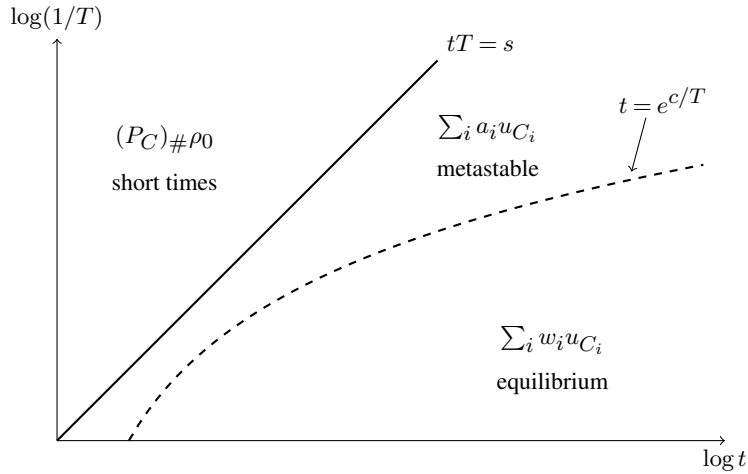

\subsection{Mixing in total variation}\label{subsec:tvmix}
Theorem \ref{thm:diagram}(d) gives convergence to equilibrium in the Wasserstein distance when $tT\to\infty$. We now prove convergence in total variation, with an explicit bound, when $t$ is a large multiple of $\log(1/T)/T$. The proof needs a lower bound for the transition density at the natural scale $\sqrt T$. We state it in a form which also covers the disconnected sets of \cite[Section~3]{DarapaneniMeta}. Recall that $\Pi_C(x)$ is the set of nearest points of $x$ in $C$, and that $\omega_d$ is the volume of the unit ball. We write $\bar B(p,r)$ for the closed ball with centre $p$ and radius $r$.

\begin{lemma}\label{lem:density}
Let $C\subset\R^d$ be nonempty and closed, let $C'\subset C$ be nonempty, closed and convex, and let $\vartheta\in(0,\infty]$ be such that $\Pi_C(x)=\{P_{C'}x\}$ for all $x$ with $d_{C'}(x)<\vartheta$. Let $\lambda\ge0$ and let $T>0$ satisfy $(\lambda+1+8\sqrt d)\sqrt T<\vartheta$. Let $y\in\R^d$ with $d_{C'}(y)\le\lambda\sqrt T$, let $p=P_{C'}y$, and let $X^y$ be the solution of \eqref{eq:SDEgen} with $X^y_0=y$. Then for every Borel set $A$ contained in the ball $\bar B(p,\sqrt T)$,
\[
\P\big(X^y_1\in A\big)\ge c_{d,\lambda}\,T^{-d/2}\,\Leb(A),\qquad c_{d,\lambda}:=\frac12\,(4\pi)^{-d/2}\exp\Big(-\frac14-\frac34(\lambda+1)^2-\frac d3\Big).
\]
\end{lemma}

For a closed convex set $C$ we take $C'=C$ and $\vartheta=\infty$. Then there is no condition on $T$.

\begin{proof}
\emph{Step 1: the Girsanov formula.} Let $Z_u=y+\sqrt{2T}\,W_u$, $0\le u\le1$, and let $b(x)=-(x-P_Cx)$. By \ref{S2}, for every bounded Borel function $F$ on $C([0,1],\R^d)$,
\begin{equation}\label{eq:girs}
\E\,F(X^y)=\E\big[F(Z)\,\cE\big],\qquad\cE=\exp\Big(\frac1{2T}\int_0^1\ip{b(Z_u)}{dZ_u}-\frac1{4T}\int_0^1|b(Z_u)|^2\,du\Big).
\end{equation}
We recall where the density $\cE$ comes from. Put $\eta_u=b(Z_u)/\sqrt{2T}$. Since $dZ_u=\sqrt{2T}\,dW_u$,
\[
\frac1{2T}\int_0^1\ip{b(Z_u)}{dZ_u}=\int_0^1\ip{\eta_u}{dW_u},\qquad\frac1{4T}\int_0^1|b(Z_u)|^2\,du=\frac12\int_0^1|\eta_u|^2\,du .
\]
So $\cE=\exp\big(\int_0^1\ip{\eta_u}{dW_u}-\frac12\int_0^1|\eta_u|^2\,du\big)$ is the exponential density of the Girsanov theorem. We have $\E\,\cE=1$ because $b$ has linear growth, as recalled in \ref{S2}. Under the probability measure $\cE\,d\P$ the process $\widetilde W_u=W_u-\int_0^u\eta_v\,dv$, $0\le u\le1$, is a Brownian motion. Hence
\[
Z_u=y+\sqrt{2T}\,W_u=y+\int_0^ub(Z_v)\,dv+\sqrt{2T}\,\widetilde W_u .
\]
Thus, under this measure, $Z$ is a solution of \eqref{eq:SDEgen} started at $y$. The solution is unique in law by \ref{S1}. Hence the law of $Z$ under $\cE\,d\P$ is the law of $X^y$. This is \eqref{eq:girs}.

\emph{Step 2: a lower bound for $\cE$.} Let $G_0$ be the event that $d_{C'}(Z_u)<\vartheta$ for all $u\le1$. Let $\tau_0=\inf\{u:\ d_{C'}(Z_u)\ge\vartheta\}$, so that $G_0=\{\tau_0>1\}$. For $u<\tau_0$ the point $Z_u$ has the unique nearest point $P_{C'}Z_u$ in $C$, so we have $b(Z_u)=-(Z_u-P_{C'}Z_u)=-\nabla V_{C'}(Z_u)$ and $|b(Z_u)|=d_{C'}(Z_u)$. Hence the stochastic integrals of $b(Z)$ and of $-\nabla V_{C'}(Z)$ with respect to $Z$ coincide up to the time $\tau_0\wedge1$, so they coincide on $G_0$. The function $V_{C'}$ is convex with a Lipschitz gradient. For such a function $f$ and the Brownian path $Z$,
\begin{equation}\label{eq:convexito}
f(Z_1)-f(Z_0)\ge\int_0^1\ip{\nabla f(Z_u)}{dZ_u}\qquad\text{almost surely}.
\end{equation}
We prove \eqref{eq:convexito}. First let $f$ be smooth and convex. The quadratic variation of each coordinate of $Z$ is $2Tu$, so the It\^o formula reads
\[
f(Z_1)-f(Z_0)=\int_0^1\ip{\nabla f(Z_u)}{dZ_u}+T\int_0^1\Delta f(Z_u)\,du .
\]
The term $T\int_0^1\Delta f(Z_u)\,du$ is nonnegative because $\Delta f\ge0$ for a convex function. This is \eqref{eq:convexito} for smooth convex $f$. In the general case let $\ell$ be the Lipschitz constant of $\nabla f$ and apply the smooth case to $f_n=f*\theta_n$, which is smooth and convex. As $n\to\infty$ we have $f_n\to f$ pointwise. Also $|\nabla f_n-\nabla f|\le\ell/n$ as in the proof of Lemma \ref{lem:L2}. By the It\^o isometry,
\[
\E\Big|\int_0^1\ip{\nabla f_n(Z_u)-\nabla f(Z_u)}{dZ_u}\Big|^2=2T\,\E\int_0^1|\nabla f_n(Z_u)-\nabla f(Z_u)|^2\,du\le\frac{2T\ell^2}{n^2}\longrightarrow0 .
\]
So, along a subsequence, the stochastic integrals converge almost surely. Letting $n\to\infty$ gives \eqref{eq:convexito}.

By \eqref{eq:convexito} for $f=V_{C'}$ and since $V_{C'}(Z_0)\ge0$, on $G_0$
\[
\frac1{2T}\int_0^1\ip{b(Z_u)}{dZ_u}=-\frac1{2T}\int_0^1\ip{\nabla V_{C'}(Z_u)}{dZ_u}\ge-\frac{V_{C'}(Z_1)}{2T}=-\frac{d_{C'}(Z_1)^2}{4T} .
\]
Also $|b(Z_u)|^2=d_{C'}(Z_u)^2$ for all $u\le1$ on $G_0$. Therefore
\begin{equation}\label{eq:Elower}
\cE\ge\exp\Big(-\frac{d_{C'}(Z_1)^2}{4T}-\frac1{4T}\int_0^1d_{C'}(Z_u)^2\,du\Big)\qquad\text{on }G_0 .
\end{equation}

\emph{Step 3: the bridge.} Write $W_u=uW_1+\beta^\circ_u$, where $\beta^\circ_u=W_u-uW_1$ is a Brownian bridge which is independent of $W_1$. The independence holds because the two are jointly Gaussian and the covariance of $\beta^\circ_u$ and $W_1$ is $(u-u)$ times the identity, which is zero. Since $Z_1-y=\sqrt{2T}\,W_1$,
\[
Z_u=y+\sqrt{2T}\,\big(uW_1+\beta^\circ_u\big)=y+u\,(Z_1-y)+\sqrt{2T}\,\beta^\circ_u ,
\]
and $\beta^\circ$ is independent of $Z_1$. Let $G_1$ be the event that $\sup_{u\le1}|\beta^\circ_u|\le4\sqrt{2d}$. We have $|\beta^\circ_u|\le|W_u|+u|W_1|\le2\sup_{u\le1}|W_u|$. Also $\E\sup_{u\le1}|W_u|^2\le4\E|W_1|^2=4d$ by the Doob inequality for the submartingale $|W_u|$. By the Chebyshev inequality,
\[
\P(G_1^c)\le\P\Big(\sup_{u\le1}|W_u|>2\sqrt{2d}\Big)\le\frac{\E\sup_{u\le1}|W_u|^2}{(2\sqrt{2d})^2}\le\frac{4d}{8d}=\frac12 .
\]
Assume now that $Z_1\in A$ and that $G_1$ holds. We write $Z_u=(1-u)\,y+u\,Z_1+\sqrt{2T}\,\beta^\circ_u$. The function $d_{C'}$ is nonexpansive, which gives the term $\sqrt{2T}\,|\beta^\circ_u|$ below. It is also convex, which gives the two other terms. We also use $d_{C'}(y)\le\lambda\sqrt T$ and $d_{C'}(Z_1)\le|Z_1-p|\le\sqrt T$, which holds because $p\in C'$ and $A\subset\bar B(p,\sqrt T)$. On $G_1$ we have $\sqrt{2T}\,|\beta^\circ_u|\le\sqrt{2T}\cdot4\sqrt{2d}=8\sqrt d\,\sqrt T$. Thus
\begin{equation}\label{eq:dKbridge}
\begin{aligned}
d_{C'}(Z_u)&\le(1-u)\,d_{C'}(y)+u\,d_{C'}(Z_1)+\sqrt{2T}\,|\beta^\circ_u|\le(\lambda+1)\sqrt T+\sqrt{2T}\,|\beta^\circ_u|\\
&\le(\lambda+1+8\sqrt d)\sqrt T<\vartheta .
\end{aligned}
\end{equation}
So $G_0$ holds. By \eqref{eq:dKbridge} and $(z_1+z_2)^2\le2z_1^2+2z_2^2$ we have $d_{C'}(Z_u)^2\le2(\lambda+1)^2T+4T|\beta^\circ_u|^2$. We divide by $4T$, integrate and use $d_{C'}(Z_1)\le\sqrt T$ once more:
\[
\frac{d_{C'}(Z_1)^2}{4T}\le\frac14,\qquad\frac1{4T}\int_0^1d_{C'}(Z_u)^2\,du\le\frac{(\lambda+1)^2}2+\int_0^1|\beta^\circ_u|^2\,du .
\]

Now we put the estimates together. We use \eqref{eq:girs} with $F=\one_{\{Z_1\in A\}}$. On the event $\{Z_1\in A\}\cap G_1$ the event $G_0$ holds. So \eqref{eq:Elower} and Step 3 give
\[
\cE\ge\exp\Big(-\frac14-\frac{(\lambda+1)^2}2-\int_0^1|\beta^\circ_u|^2\,du\Big).
\]
Hence
\begin{align*}
\P\big(X^y_1\in A\big)&=\E\big[\one_{\{Z_1\in A\}}\,\cE\big]\ge\E\big[\one_{\{Z_1\in A\}}\one_{G_1}\,\cE\big]\\
&\ge\exp\Big(-\frac14-\frac{(\lambda+1)^2}2\Big)\,\E\Big[\one_{\{Z_1\in A\}}\one_{G_1}\exp\Big(-\int_0^1|\beta^\circ_u|^2\,du\Big)\Big].
\end{align*}
The event $G_1$ and the integral depend only on $\beta^\circ$. By the independence of $\beta^\circ$ and $Z_1$,
\[
\P\big(X^y_1\in A\big)\ge\P(Z_1\in A)\;\exp\Big(-\frac14-\frac{(\lambda+1)^2}2\Big)\;\E\Big[\one_{G_1}\exp\Big(-\int_0^1|\beta^\circ_u|^2\,du\Big)\Big].
\]
We bound the last factor. Put $I=\int_0^1|\beta^\circ_u|^2\,du$. Each coordinate of $\beta^\circ_u$ has the variance $\E(W^1_u-uW^1_1)^2=u-2u^2+u^2=u(1-u)$, where $W^1$ is the first coordinate of $W$. So
\[
\E\,I=\E\int_0^1|\beta^\circ_u|^2\,du=d\int_0^1u(1-u)\,du=d\,\Big(\frac12-\frac13\Big)=\frac d6 .
\]
The Jensen inequality for the convex function $z\mapsto e^{-z}$ and the conditional law given $G_1$ gives
\[
\E\big[\one_{G_1}e^{-I}\big]=\P(G_1)\,\E\big[e^{-I}\,\big|\,G_1\big]\ge\P(G_1)\exp\big(-\E[I\,|\,G_1]\big),
\]
and since $I\ge0$,
\[
\E[I\,|\,G_1]=\frac{\E[I\one_{G_1}]}{\P(G_1)}\le\frac{\E\,I}{\P(G_1)} .
\]
The function $z\mapsto z\,e^{-a/z}$ is increasing on $(0,1]$ for $a\ge0$, because its derivative is $(1+a/z)\,e^{-a/z}>0$. Since $\P(G_1)\ge\frac12$, we get with $a=d/6$
\[
\E\Big[\one_{G_1}\exp\Big(-\int_0^1|\beta^\circ_u|^2\,du\Big)\Big]\ge\P(G_1)\exp\Big(-\frac{1}{\P(G_1)}\,\E\int_0^1|\beta^\circ_u|^2\,du\Big)\ge\frac12\,e^{-d/3} .
\]
Finally $Z_1$ is Gaussian with mean $y$ and covariance $2T$ times the identity. For $x\in A$ we have $|x-y|\le|x-p|+|p-y|\le(1+\lambda)\sqrt T$. Hence
\[
\P(Z_1\in A)=\int_A(4\pi T)^{-d/2}e^{-|x-y|^2/4T}\,dx\ge(4\pi T)^{-d/2}e^{-(\lambda+1)^2/4}\,\Leb(A).
\]
We multiply the three bounds:
\[
\P\big(X^y_1\in A\big)\ge(4\pi T)^{-d/2}e^{-(\lambda+1)^2/4}\,\Leb(A)\cdot\exp\Big(-\frac14-\frac{(\lambda+1)^2}2\Big)\cdot\frac12\,e^{-d/3} .
\]
Since $\frac14(\lambda+1)^2+\frac12(\lambda+1)^2=\frac34(\lambda+1)^2$, the right side is $c_{d,\lambda}\,T^{-d/2}\,\Leb(A)$. This gives the claim.
\end{proof}

We can now state the result. Let
\[
\lambda_d=\tfrac12+\sqrt{2d},\qquad\alpha_d=\tfrac14\,c_{d,\lambda_d}\,\omega_d\in(0,1),
\]
which depend only on the dimension. Lemma \ref{lem:density}, applied to a closed convex set with $A=\bar B(p,\sqrt T)$, gives $1\ge c_{d,\lambda}\,T^{-d/2}\,\omega_dT^{d/2}=c_{d,\lambda}\,\omega_d$ for every $\lambda\ge0$. So $\alpha_d\le\frac14$.

\begin{theorem}\label{thm:tvmix}
Let $C\subset\R^d$ be a nonempty compact convex set with diameter $D$, let $\ell_0>0$ and $0<T\le1$, and put
\[
r_T=\log_+\frac{4\ell_0}{\sqrt T}+1+\frac{c_*(2D^2+2d)}{T}\,\log\Big(2\Big(\frac{\sqrt e\,\cZ_1(C)}{\omega_d}\Big)^{1/2}T^{-d/4}\Big).
\]
Then for every $x$ with $d_C(x)\le\ell_0$, every integer $n\ge1$ and every $t\ge n\,r_T$,
\[
\TV\big(P_t(x,\cdot),\gamma_T^C\big)\le(1-\alpha_d)^n+2\exp\Big(-\frac{(\ell_0-\sqrt{dT})_+^2}{2T}\Big).
\]
\end{theorem}

\begin{proof}
Let $M=\{y:d_C(y)\le2\ell_0\}$. We write $r_T=t_a+1+t_c$, where
\[
t_a=\log_+\frac{4\ell_0}{\sqrt T}
\]
and $t_c$ is the last term in the definition of $r_T$. The number $t_c$ is positive, because $\cZ_1(C)\ge(2\pi)^{d/2}V_0(C)=(2\pi)^{d/2}$ by \eqref{eq:wills}. Indeed, $(2\pi)^{d/2}=\int e^{-|x|^2/2}dx\ge e^{-1/2}\omega_d$, where we integrated only over the unit ball. So $\sqrt e\,\cZ_1(C)/\omega_d\ge1$. Since $T^{-d/4}\ge1$, the logarithm in $t_c$ is at least $\log2$.

The process spends the time $t_a$ to come close to $C$, the time $1$ to spread at the scale $\sqrt T$ and the time $t_c$ to mix inside $C$.

\emph{Step 1: approach to $C$.} Let $y\in M$. By \eqref{eq:env3},
\[
d_C(X^y_{t_a})\le e^{-t_a}\,2\ell_0+|\zeta_{t_a}|\le\frac{\sqrt T}2+|\zeta_{t_a}| .
\]
For the second inequality there are two cases. If $4\ell_0\ge\sqrt T$, then $e^{-t_a}=\sqrt T/(4\ell_0)$ and $e^{-t_a}\,2\ell_0=\sqrt T/2$. If $4\ell_0<\sqrt T$, then $t_a=0$ and $2\ell_0<\sqrt T/2$. Let $G_a=\{|\zeta_{t_a}|\le\sqrt{2dT}\}$. Since $\E|\zeta_{t_a}|^2\le dT$ by \eqref{eq:oucov}, the Chebyshev inequality gives
\[
\P(G_a^c)=\P\big(|\zeta_{t_a}|>\sqrt{2dT}\big)\le\frac{dT}{2dT}=\frac12 .
\]
So the event $G_a$ has probability at least $\frac12$. On $G_a$ we have
\[
d_C(X^y_{t_a})\le\Big(\frac12+\sqrt{2d}\Big)\sqrt T=\lambda_d\sqrt T .
\]

\emph{Step 2: a lower bound for the transition probability.} For $y'$ with $d_C(y')\le\lambda_d\sqrt T$ let $\nu_{y'}$ be the uniform probability measure on the ball $\bar B(P_Cy',\sqrt T)$, whose volume is $\omega_dT^{d/2}$. We apply Lemma \ref{lem:density} with $C'=C$, $\vartheta=\infty$ and $\lambda=\lambda_d$. For every Borel set $A$,
\begin{align*}
P_1(y',A)\ge P_1\big(y',A\cap\bar B(P_Cy',\sqrt T)\big)&\ge c_{d,\lambda_d}\,T^{-d/2}\,\Leb\big(A\cap\bar B(P_Cy',\sqrt T)\big)\\
&=c_{d,\lambda_d}\,\omega_d\,\nu_{y'}(A).
\end{align*}
So the lemma says that
\[
P_1(y',\cdot)\ge c_{d,\lambda_d}\,\omega_d\;\nu_{y'} .
\]

\emph{Step 3: mixing from a ball.} On $\bar B(P_Cy',\sqrt T)$ we have $d_C\le\sqrt T$. Hence $V_C\le T/2$. Also $\cZ_T(C)\le\cZ_1(C)$, because $T\le1$ and so $e^{-d_C^2/2T}\le e^{-d_C^2/2}$. The measure $\nu_{y'}$ has the Lebesgue density $\one_{\bar B(P_Cy',\sqrt T)}/(\omega_dT^{d/2})$ and $\gamma_T^C$ has the Lebesgue density $e^{-V_C/T}/\cZ_T(C)$. So the density $h_0=\cZ_T(C)\,e^{V_C/T}\one_{\bar B(P_Cy',\sqrt T)}/(\omega_dT^{d/2})$ of $\nu_{y'}$ with respect to $\gamma_T^C$ satisfies
\[
\|h_0-1\|_{L^2(\gamma_T^C)}^2\le\int h_0^2\,d\gamma_T^C=\frac{\cZ_T(C)}{(\omega_dT^{d/2})^2}\int_{\bar B(P_Cy',\sqrt T)}e^{V_C/T}\,dx\le\frac{\sqrt e\,\cZ_1(C)}{\omega_d\,T^{d/2}} .
\]
In the last step we used $e^{V_C/T}\le\sqrt e$ on the ball, the volume $\omega_dT^{d/2}$ of the ball and $\cZ_T(C)\le\cZ_1(C)$. By Lemma \ref{lem:L2} and the definition of $t_c$, using $M_T\le2D^2+2dT\le2D^2+2d$,
\[
\TV\big(\nu_{y'}P_{t_c},\gamma_T^C\big)\le\frac12\,\exp\Big(-\frac{T\,t_c}{c_*(2D^2+2d)}\Big)\Big(\frac{\sqrt e\,\cZ_1(C)}{\omega_d\,T^{d/2}}\Big)^{1/2}\le\frac14 .
\]
The value $\frac14$ comes out exactly. By the definition of $t_c$,
\[
\frac{T\,t_c}{c_*(2D^2+2d)}=\log\Big(2\Big(\frac{\sqrt e\,\cZ_1(C)}{\omega_d}\Big)^{1/2}T^{-d/4}\Big)=\log\Big(2\Big(\frac{\sqrt e\,\cZ_1(C)}{\omega_d\,T^{d/2}}\Big)^{1/2}\Big),
\]
so that
\[
\exp\Big(-\frac{T\,t_c}{c_*(2D^2+2d)}\Big)=\frac12\Big(\frac{\sqrt e\,\cZ_1(C)}{\omega_d\,T^{d/2}}\Big)^{-1/2},
\]
and the product of the three factors is $\frac12\cdot\frac12=\frac14$.

\emph{Step 4: one round.} Let $y\in M$ and let $A$ be a Borel set. The event $G_a$ is determined by the Brownian motion up to the time $t_a$. By the Markov property at the time $t_a$,
\[
P_{r_T}(y,A)=\E\big[P_{1+t_c}(X^y_{t_a},A)\big]\ge\E\big[\one_{G_a}\,P_{1+t_c}(X^y_{t_a},A)\big].
\]
On $G_a$ the point $y'=X^y_{t_a}$ satisfies $d_C(y')\le\lambda_d\sqrt T$ by Step 1. For such a point the Markov property at the time $1$ and Step 2 give
\begin{align*}
P_{1+t_c}(y',A)=\int P_1(y',dz)\,P_{t_c}(z,A)&\ge c_{d,\lambda_d}\,\omega_d\int\nu_{y'}(dz)\,P_{t_c}(z,A)\\
&=c_{d,\lambda_d}\,\omega_d\,\big(\nu_{y'}P_{t_c}\big)(A).
\end{align*}
Thus, by the Markov property at the times $t_a$ and $t_a+1$ and by Steps 1 and 2,
\[
P_{r_T}(y,\cdot)\ge\E\big[\one_{G_a}\,c_{d,\lambda_d}\omega_d\;\nu_{X^y_{t_a}}P_{t_c}\big]=c_{d,\lambda_d}\omega_d\,\P(G_a)\;m_y,\qquad m_y:=\E\big[\nu_{X^y_{t_a}}P_{t_c}\,\big|\,G_a\big].
\]
The probability measure $m_y$ is a mixture of measures which are at distance at most $\frac14$ from $\gamma_T^C$ by Step 3. The total variation distance is convex, so $\TV(m_y,\gamma_T^C)\le\frac14$. By the triangle inequality, $\TV(m_y,m_{\tilde y})\le\frac12$ for $y,\tilde y\in M$. Since $\P(G_a)\ge\frac12$, we have $P_{r_T}(y,\cdot)\ge\frac12c_{d,\lambda_d}\omega_d\,m_y$ for every $y\in M$. By Lemma \ref{lem:doeblin}(i) with $a=\frac12c_{d,\lambda_d}\omega_d$,
\[
\TV\big(P_{r_T}(y,\cdot),P_{r_T}(\tilde y,\cdot)\big)\le1-\tfrac12c_{d,\lambda_d}\omega_d\cdot\tfrac12=1-\alpha_d\qquad(y,\tilde y\in M).
\]

\emph{Step 5: iteration.} Let $d_C(x)\le\ell_0$ and put $e_n=\TV(P_{nr_T}(x,\cdot),\gamma_T^C)$. We apply Lemma \ref{lem:doeblin}(ii) with $\Pi=P_{r_T}$, $\alpha=\alpha_d$, $\mu=P_{nr_T}(x,\cdot)$ and $\mu'=\gamma_T^C$. Here $\mu\Pi=P_{(n+1)r_T}(x,\cdot)$ by the Markov property and $\mu'\Pi=\gamma_T^C$ since $\gamma_T^C$ is invariant. This gives
\[
e_{n+1}\le(1-\alpha_d)\,e_n+\alpha_d\,\big(P_{nr_T}(x,M^c)+\gamma_T^C(M^c)\big).
\]
By \eqref{eq:env3} we have $d_C(X^x_t)\le e^{-t}d_C(x)+|\zeta_t|\le\ell_0+|\zeta_t|$. So $d_C(X^x_t)>2\ell_0$ implies $|\zeta_t|>\ell_0$. By \eqref{eq:gaussconc} with $\tr Q=d$, $\|Q\|=1$ and $\ell=(\ell_0-\sqrt{dT})_+$, for every $t$,
\[
P_t(x,M^c)=\P\big(d_C(X^x_t)>2\ell_0\big)\le\P\big(|\zeta_t|>\ell_0\big)\le q:=\exp\Big(-\frac{(\ell_0-\sqrt{dT})_+^2}{2T}\Big).
\]
When $\ell_0\le\sqrt{dT}$ this bound is trivial, since then $q=1$. By Lemma \ref{lem:invariant}, $\gamma_T^C(M^c)=\lim_{t\to\infty}P_t(x,M^c)\le q$. Hence $e_{n+1}\le(1-\alpha_d)e_n+2\alpha_dq$. We show by induction that $e_n\le(1-\alpha_d)^n+2q$. For $n=0$ this holds because $e_0\le1$. If it holds for $n$, then
\begin{align*}
e_{n+1}&\le(1-\alpha_d)\big((1-\alpha_d)^n+2q\big)+2\alpha_dq\\
&=(1-\alpha_d)^{n+1}+2q(1-\alpha_d)+2\alpha_dq=(1-\alpha_d)^{n+1}+2q .
\end{align*}
Finally, $t\mapsto\TV(P_t(x,\cdot),\gamma_T^C)$ is nonincreasing because $\gamma_T^C$ is invariant. Indeed, for $t\ge t'$,
\[
\TV\big(P_t(x,\cdot),\gamma_T^C\big)=\TV\big(P_{t'}(x,\cdot)P_{t-t'},\ \gamma_T^CP_{t-t'}\big)\le\TV\big(P_{t'}(x,\cdot),\gamma_T^C\big),
\]
since a Markov kernel does not increase the total variation distance. So the bound holds for all $t\ge nr_T$.
\end{proof}

\begin{remark}\label{rem:tvmix}
The last term in Theorem \ref{thm:tvmix} does not depend on $n$. So for a fixed temperature the theorem is not a statement of convergence as $t\to\infty$ (this convergence is Lemma \ref{lem:invariant}). The term is the probability that the state is found far from $C$. It is exponentially small in $1/T$. The constant $\alpha_d$ is explicit but very small. It decreases fast with the dimension (for example $\alpha_1\approx7\cdot10^{-5}$ and $\alpha_3\approx3\cdot10^{-8}$). We did not try to optimise it.
\end{remark}

\begin{corollary}\label{cor:mixtime}
Let $C\subset\R^d$ be a nonempty compact convex set and $\ell_0>0$.
\begin{enumerate}[label=\textup{(\roman*)}]
\item If $T\to0$ and $t=t(T)$ is such that $tT/\log(1/T)\to\infty$, then
\[
\sup_{x:\,d_C(x)\le\ell_0}\TV\big(P_t(x,\cdot),\gamma_T^C\big)\longrightarrow0 .
\]
\item If $\dim C\ge1$, there is $s_0>0$ such that
\[
\liminf_{T\to0}\ \inf_{x\in C}\TV\big(P_{s_0/T}(x,\cdot),\gamma_T^C\big)\ge\frac12 .
\]
\item Let $t_{\mathrm{mix}}(T)=\inf\{t\ge0:\ \sup_{x\in C}\TV(P_t(x,\cdot),\gamma_T^C)\le\frac14\}$. If $\dim C\ge1$, there is a constant $A$ such that, for all small $T$,
\[
\frac{s_0}{T}\le t_{\mathrm{mix}}(T)\le\frac{A}{T}\,\log\frac1T .
\]
\end{enumerate}
\end{corollary}

\begin{proof}
(i) We have $r_T\le A_1\log(1/T)/T$ for small $T$, with a constant $A_1$ depending on $C$, $d$ and $\ell_0$. To see this, let $T\le1/e$, so that $\log(1/T)\ge1$ and $1/T\ge1$. The first two terms of $r_T$ satisfy
\[
\log_+\frac{4\ell_0}{\sqrt T}\le\log_+(4\ell_0)+\frac12\log\frac1T,\qquad1\le\log\frac1T .
\]
Put $a_0=\log\big(2(\sqrt e\,\cZ_1(C)/\omega_d)^{1/2}\big)$, which is positive by the beginning of the proof of Theorem \ref{thm:tvmix}. The third term of $r_T$ is
\[
\frac{c_*(2D^2+2d)}{T}\Big[a_0+\frac d4\log\frac1T\Big]\le\frac{c_*(2D^2+2d)}{T}\Big(a_0+\frac d4\Big)\log\frac1T .
\]
So we may take
\[
A_1=\log_+(4\ell_0)+\frac32+c_*(2D^2+2d)\Big(a_0+\frac d4\Big).
\]
Let $n=\lfloor t/r_T\rfloor$. Then
\[
n\ge\frac t{r_T}-1\ge\frac{tT}{A_1\log(1/T)}-1\longrightarrow\infty ,
\]
and $t\ge n\,r_T$. Both terms of the bound of Theorem \ref{thm:tvmix} tend to zero, the first because $n\to\infty$ and the second because $T\to0$ and $\ell_0>0$. The bound does not depend on $x$.

(ii) Let $k=\dim C\ge1$, $x\in C$, $\eta>0$ and $s>0$, and let $\bar B_x$ be the closed ball of radius $\eta$ around $x$. We use the coupling of the proof of Theorem \ref{thm:diagram}(b) with $\rho_0=\delta_x$. In this coupling $R_0=x$ and $\E|X_{s/T}-R_s|^2\to0$ as $T\to0$, uniformly in $x\in C$ because $m_0=0$. Indeed, that proof bounds $\E|X_{s/T}-R_s|^2$ by the right side of (b). With $m_0=0$ and $S=s$ this reads, for $T\le1/e$ and $T\log(1/T)\le s$,
\[
\E|X_{s/T}-R_s|^2\le12\,d\,T\log\frac1T+22\Big(T\Big(k+\log\Big(1+\frac sT\Big)\Big)\Big)^{1/2}\Lambda_C(s),
\]
and the right side does not depend on $x$. By Lemma \ref{lem:penal}(ii) and the Chebyshev inequality,
\[
\P\big(|R_s-x|\ge\eta/2\big)\le\frac{\E|R_s-x|^2}{(\eta/2)^2}\le\frac{2ks}{\eta^2/4}=\frac{8ks}{\eta^2} .
\]
If $|R_s-x|<\eta/2$ and $|X_{s/T}-R_s|\le\eta/2$, then $X_{s/T}\in\bar B_x$. So, using the Chebyshev inequality again,
\begin{align*}
P_{s/T}(x,\bar B_x)&\ge1-\P\big(|R_s-x|\ge\eta/2\big)-\P\big(|X_{s/T}-R_s|>\eta/2\big)\\
&\ge1-\frac{8ks}{\eta^2}-\frac4{\eta^2}\,\E|X_{s/T}-R_s|^2 .
\end{align*}
Hence
\[
\liminf_{T\to0}\ \inf_{x\in C}P_{s/T}(x,\bar B_x)\ge1-\frac{8ks}{\eta^2} .
\]
On the other hand, if $|y-x|\le\eta$ then $|P_Cy-x|=|P_Cy-P_Cx|\le\eta$, because $P_C$ is nonexpansive. Thus $\bar B_x\subset P_C^{-1}(\bar B_x)$. The set $C\cap\bar B_x$ lies in a ball of radius $\eta$ of the affine space $L$ of dimension $k$, so $u_C(\bar B_x)\le\omega_k\eta^k/\cH^k(C)$. Hence, by \eqref{eq:zt1},
\[
\gamma_T^C(\bar B_x)\le(P_C)_\#\gamma_T^C(\bar B_x)\le u_C(\bar B_x)+\big(1-w_k(T)\big)\le\frac{\omega_k\eta^k}{\cH^k(C)}+\big(1-w_k(T)\big),
\]
and $1-w_k(T)\to0$ as $T\to0$, uniformly in $x$.
Choose $\eta$ with $\omega_k\eta^k/\cH^k(C)\le\frac14$ and then $s_0=\eta^2/(32k)$, so that $8ks_0/\eta^2=\frac14$. Then
\[
\liminf_{T\to0}\ \inf_{x\in C}\TV\big(P_{s_0/T}(x,\cdot),\gamma_T^C\big)\ge\liminf_{T\to0}\ \inf_{x\in C}\big(P_{s_0/T}(x,\bar B_x)-\gamma_T^C(\bar B_x)\big)\ge1-\frac14-\frac14=\frac12 .
\]

(iii) The lower bound follows from (ii), since the total variation distance to equilibrium is nonincreasing in time. Indeed, by (ii), for all small $T$ and every $x\in C$ we have $\TV(P_{s_0/T}(x,\cdot),\gamma_T^C)>\frac14$. By the monotonicity in time shown at the end of the proof of Theorem \ref{thm:tvmix}, the same holds for all $t\le s_0/T$. Hence $t_{\mathrm{mix}}(T)\ge s_0/T$. For the upper bound apply Theorem \ref{thm:tvmix} with $\ell_0=1$ and with $n$ the smallest integer such that $(1-\alpha_d)^n\le\frac18$. This $n$ depends only on $d$. Every $x\in C$ satisfies $d_C(x)=0\le\ell_0$. The last term in the theorem is $2\exp(-(1-\sqrt{dT})_+^2/2T)$, which is smaller than $\frac18$ for small $T$. So for $t=n\,r_T$ the distance is at most $\frac18+\frac18=\frac14$ for all $x\in C$. By the bound for $r_T$ in the proof of (i) with $\ell_0=1$, we get
\[
t_{\mathrm{mix}}(T)\le n\,r_T\le\frac{nA_1}{T}\,\log\frac1T .
\]
Thus $n\,r_T\le A\log(1/T)/T$ with $A=nA_1$.
\end{proof}

\section{The associated control problem}\label{sec:hjb}

\subsection{The control problem and the logarithmic transformation}
Let $C\subset\R^d$ be compact, let $T>0$ and let $g:\R^d\to\R$ be bounded and Lipschitz. We consider an agent who can add a drift $a_u$ to the state equation. The agent pays $\frac12|a_u|^2$ per unit time for this drift and pays $g$ at the final time $t$. The value function of this problem is
\begin{equation}\label{eq:value}
v_T(t,x)=\inf_a\ \E\Big[\frac12\int_0^t|a_u|^2\,du+g(X^a_t)\Big],
\end{equation}
where $X^a$ is the controlled state,
\[
dX^a_u=\big(-(X^a_u-P_CX^a_u)+a_u\big)\,du+\sqrt{2T}\,dW_u,\qquad X^a_0=x,
\]
The infimum is taken over the progressively measurable processes $a$ with $\E\int_0^t|a_u|^2du<\infty$. Here $t$ is the time which remains, so that $v_T$ solves an initial value problem.

We write the dynamic programming equation of this problem. For all $a,\xi\in\R^d$ we have
\[
\ip{a}{\xi}+\tfrac12|a|^2=\tfrac12|a+\xi|^2-\tfrac12|\xi|^2 ,
\]
and the right side is smallest for $a=-\xi$. Since $\inf_a\{\ip{a}{\xi}+\frac12|a|^2\}=-\frac12|\xi|^2$, the dynamic programming equation is
\begin{equation}\label{eq:HJB2}
\begin{aligned}
\partial_tv&=\inf_{a}\Big\{T\Delta v+\ip{-\nabla V_C+a}{\nabla v}+\tfrac12|a|^2\Big\}\\
&=T\Delta v-\ip{\nabla V_C}{\nabla v}-\tfrac12|\nabla v|^2,\qquad v(0,\cdot)=g .
\end{aligned}
\end{equation}
In the second line we used the infimum above with $\xi=\nabla v$.

Put $v=-2T\log\psi$ with $\psi>0$. This substitution is called the logarithmic transformation. It transforms \eqref{eq:HJB2} into the linear equation $\partial_t\psi=T\Delta\psi-\ip{\nabla V_C}{\nabla\psi}$, $\psi(0,\cdot)=e^{-g/2T}$. This is the backward equation of the state process without control. It is the adjoint of the forward Kolmogorov equation \eqref{eq:FP}. We give the computation. Differentiating $v=-2T\log\psi$ in $t$ and in $x$ we get
\[
\partial_tv=-\frac{2T}{\psi}\,\partial_t\psi,\qquad\nabla v=-\frac{2T}{\psi}\,\nabla\psi .
\]
Taking the divergence of the second formula we get
\[
\Delta v=-\frac{2T}{\psi}\,\Delta\psi+\frac{2T}{\psi^2}\,|\nabla\psi|^2 .
\]
Now $|\nabla v|^2=4T^2|\nabla\psi|^2/\psi^2$, so the last term is equal to $|\nabla v|^2/2T$. Thus we have $\nabla v=-2T\nabla\psi/\psi$ and $\Delta v=-2T\Delta\psi/\psi+|\nabla v|^2/2T$. It follows that
\[
T\Delta v-\tfrac12|\nabla v|^2=-\frac{2T}{\psi}\,T\Delta\psi,\qquad-\ip{\nabla V_C}{\nabla v}=\frac{2T}{\psi}\,\ip{\nabla V_C}{\nabla\psi} .
\]
We add these two identities and use that $\psi$ solves the linear equation. This gives
\[
T\Delta v-\ip{\nabla V_C}{\nabla v}-\tfrac12|\nabla v|^2=-\frac{2T}{\psi}\big(T\Delta\psi-\ip{\nabla V_C}{\nabla\psi}\big)=-\frac{2T}{\psi}\,\partial_t\psi=\partial_tv .
\]
This is \eqref{eq:HJB2}. The initial condition $v(0,\cdot)=g$ is the same as $\psi(0,\cdot)=e^{-g/2T}$. The solution of the backward equation with this initial value is $\psi(t,x)=\E\exp(-g(X^x_t)/2T)$. Hence
\begin{equation}\label{eq:log}
v_T(t,x)=-2T\log\E\exp\Big(-\frac{g(X^x_t)}{2T}\Big),
\end{equation}
where $X^x$ is the solution of \eqref{eq:SDEgen} started at $x$. When the coefficients are bounded and smooth, the identity between \eqref{eq:value} and \eqref{eq:log} is proved in \cite[Chapter VI]{FlemingSoner2006}. The same chapter shows that \eqref{eq:log} is then a classical solution of \eqref{eq:HJB2}. The method goes back to \cite{Fleming1978}. For general functionals of the Brownian path the identity between \eqref{eq:value} and \eqref{eq:log} is the variational representation of \cite{BoueDupuis1998}, applied to the functional $W\mapsto g(X^x_t)/2T$. Our drift is not bounded. It is also discontinuous when $C$ is not convex. We do not enter into the question in which sense \eqref{eq:log} solves \eqref{eq:HJB2} in this generality. In the sequel we take \eqref{eq:log} as the \emph{definition} of $v_T$. All the results below are statements about the right side of \eqref{eq:log}.

\subsection{The cost of reaching a point}
We first study the noiseless control problem for a convex set $C$. For an absolutely continuous curve $\phi:[0,t]\to\R^d$ with square integrable derivative we call
\[
a(u)=\phi'(u)+\phi(u)-P_C\phi(u),\qquad J_t(\phi)=J^C_t(\phi)=\frac12\int_0^t|a(u)|^2\,du
\]
the control and the energy of $\phi$. Thus $\phi'=-(\phi-P_C\phi)+a$. The control $a$ is square integrable on $[0,t]$, because $\phi'$ is square integrable and $\phi-P_C\phi$ is continuous.

\begin{proposition}\label{prop:reach}
Let $C\subset\R^d$ be a nonempty closed convex set.
\begin{enumerate}[label=\textup{(\roman*)}]
\item For every curve $\phi$ as above,
\begin{equation}\label{eq:energybound}
d_C(\phi(t))\le e^{-t}d_C(\phi(0))+\big((1-e^{-2t})\,J_t(\phi)\big)^{1/2}.
\end{equation}
\item For every $y\in\R^d$,
\[
\inf\big\{J_t(\phi):\ t>0,\ \phi(0)\in C,\ \phi(t)=y\big\}=d_C(y)^2 .
\]
\end{enumerate}
\end{proposition}

\begin{proof}
(i) Integrating the identity $\phi'=-(\phi-P_C\phi)+a$ from $0$ to $t$ we get
\[
\phi(t)=\phi(0)-\int_0^t\big(\phi(u)-P_C\phi(u)\big)\,du+\int_0^ta(u)\,du .
\]
So the curve $\phi$ solves \eqref{eq:state} with the constant family $C(t)=C$ and $w(t)=\int_0^ta(u)du$. This $w$ is continuous and $w(0)=0$. For this $w$ the solution of \eqref{eq:ou} is $\zeta(t)=\int_0^te^{-(t-u)}a(u)du$. We check this by differentiation. Write $\zeta(t)=e^{-t}\int_0^te^{u}a(u)du$. Then $\zeta(0)=0$, the function $\zeta$ is absolutely continuous and for almost every $t$
\[
\zeta'(t)=-e^{-t}\int_0^te^{u}a(u)\,du+e^{-t}e^{t}a(t)=-\zeta(t)+a(t) .
\]
Integrating from $0$ to $t$ we get
\[
\zeta(t)=-\int_0^t\zeta(u)\,du+\int_0^ta(u)\,du=-\int_0^t\zeta(u)\,du+w(t) ,
\]
which is \eqref{eq:ou}. By Lemma \ref{lem:wellposed} this equation has only one continuous solution. Now we apply \eqref{eq:env2} with the speed $v=0$ and $x_0=\phi(0)$. It gives $d_C(\phi(t))\le e^{-t}d_C(\phi(0))+|\zeta(t)|$. We have $|\zeta(t)|\le\int_0^te^{-(t-u)}|a(u)|du$. We estimate this integral by the Schwarz inequality, applied to the functions $u\mapsto e^{-(t-u)}$ and $u\mapsto|a(u)|$. Thus
\begin{align*}
d_C(\phi(t))&\le e^{-t}d_C(\phi(0))+\int_0^te^{-(t-u)}|a(u)|\,du\\
&\le e^{-t}d_C(\phi(0))+\Big(\int_0^te^{-2(t-u)}du\Big)^{1/2}\Big(\int_0^t|a(u)|^2du\Big)^{1/2},
\end{align*}
and $\int_0^te^{-2(t-u)}du=\frac12(1-e^{-2t})$. Since $\int_0^t|a(u)|^2du=2J_t(\phi)$, the product of the two factors is
\[
\Big(\frac12\,(1-e^{-2t})\Big)^{1/2}\big(2J_t(\phi)\big)^{1/2}=\big((1-e^{-2t})\,J_t(\phi)\big)^{1/2}.
\]
This proves \eqref{eq:energybound}.

(ii) If $\phi(0)\in C$ and $\phi(t)=y$, then $d_C(\phi(0))=0$. Squaring \eqref{eq:energybound} we get $d_C(y)^2\le(1-e^{-2t})J_t(\phi)\le J_t(\phi)$. Hence the infimum is at least $d_C(y)^2$.

We now prove the opposite inequality. If $y\in C$, the constant curve $\phi(u)=y$ has $\phi'=0$ and $\phi-P_C\phi=0$, so its energy is $0=d_C(y)^2$. So we may assume $y\notin C$. Put $p=P_Cy$ and $r=d_C(y)>0$. For $t>1$ we define
\[
\phi(u)=\begin{cases}p+u\,e^{-(t-1)}(y-p),&0\le u\le1,\\[2pt] p+e^{-(t-u)}(y-p),&1\le u\le t.\end{cases}
\]
The two formulas agree at $u=1$, so $\phi$ is continuous. We have $\phi(0)=p\in C$ and $\phi(t)=y$. Every point $\phi(u)$ is of the form $p+\lambda(y-p)$ with $0\le\lambda\le1$. By \ref{P5}, $P_C\phi(u)=p$ for all $u$. Hence $a(u)=\phi'(u)+\phi(u)-p$. We now compute the energy of the two pieces of the path.

On $[0,1]$ we have $\phi'(u)=e^{-(t-1)}(y-p)$ and $\phi(u)-p=u\,e^{-(t-1)}(y-p)$. Since $|y-p|=r$ and $u\le1$, we have $|\phi'(u)|=re^{-(t-1)}$ and $|\phi(u)-p|\le re^{-(t-1)}$, so $|a(u)|\le2re^{-(t-1)}$. Hence the energy on $[0,1]$ satisfies
\[
\frac12\int_0^1|a(u)|^2\,du\le\frac12\,\big(2re^{-(t-1)}\big)^2=2r^2e^{-2(t-1)} .
\]
On $[1,t]$ we have $\phi'(u)=e^{-(t-u)}(y-p)$, which is the same as $\phi'(u)=\phi(u)-p$, so $a(u)=2(\phi(u)-p)$. Thus $|a(u)|^2=4r^2e^{-2(t-u)}$ and
\[
\frac12\int_1^t|a(u)|^2\,du=2r^2\int_1^te^{-2(t-u)}\,du=r^2\big(1-e^{-2(t-1)}\big).
\]
Here we used $\int_1^te^{-2(t-u)}du=\frac12(1-e^{-2(t-1)})$. Adding the energies of the two pieces we get
\[
J_t(\phi)\le2r^2e^{-2(t-1)}+r^2\big(1-e^{-2(t-1)}\big)=r^2+r^2e^{-2(t-1)} .
\]
The right side tends to $r^2$ as $t\to\infty$. So the infimum is at most $r^2=d_C(y)^2$. This completes the proof.
\end{proof}

\begin{remark}
The function $\Theta=d_C^2=2V_C$ of part (ii) is the quasipotential of the set $C$ in the sense of \cite{FreidlinWentzell2012}, that is, the least energy with which a point can be reached from $C$. It solves the stationary equation $\frac12|\nabla\Theta|^2-\ip{\nabla V_C}{\nabla\Theta}=0$, since $\nabla\Theta=2\nabla V_C$. The Gibbs state is then $\gamma_T^C=\cZ_T(C)^{-1}e^{-\Theta/2T}dx$. Part (i) is the deterministic form of the envelope. It says that a control of energy $J$ cannot move the state farther than $\sqrt J$ from a convex feasible set.
\end{remark}

\subsection{The two limits for the value function}
For a bounded function $g$ we write $\osc g=\sup g-\inf g$. For a compact set $C$ we put
\[
g_C=\min_{y\in\R^d}\big\{g(y)+d_C(y)^2\big\}.
\]
The minimum exists because $g$ is bounded and continuous and $d_C(y)\to\infty$ as $|y|\to\infty$. Indeed, fix $c\in C$ and consider the set of all $y$ with $g(y)+d_C(y)^2\le g(c)$. This set contains $c$. It is closed because $g+d_C^2$ is continuous. It is bounded because $d_C(y)^2\le g(c)-g(y)\le g(c)-\inf g$ on it and $C$ is bounded. So this set is compact and the continuous function $g+d_C^2$ attains its minimum on it. This minimum is the minimum over $\R^d$.

\begin{theorem}\label{thm:hjb}
Let $g:\R^d\to\R$ be bounded and Lipschitz and let $v_T$ be defined by \eqref{eq:log}.
\begin{enumerate}[label=\textup{(\alph*)}]
\item Let $C\subset\R^d$ be compact. Then for every $x\in\R^d$
\[
\lim_{T\to0}\ \lim_{t\to\infty}v_T(t,x)=g_C .
\]
\item Let $C\subset\R^d$ be compact and convex. Then for all $t>0$ and $x\in\R^d$ the limit
\begin{equation}\label{eq:v0}
v_0(t,x):=\lim_{T\to0}v_T(t,x)=\inf\big\{g(\phi(t))+J_t(\phi):\ \phi(0)=x\big\}
\end{equation}
exists, and
\begin{equation}\label{eq:v0bounds}
\begin{aligned}
v_0(t,x)&\ge g_C-2e^{-t}d_C(x)\big(e^{-t}d_C(x)+\sqrt{\osc g}\big)&&(t>0),\\
v_0(t,x)&\le g_C+\frac{3\,(D+e^{-t/3}\varrho)^2}{t}+\frac t3\,e^{-2t/3}\varrho^2&&(t\ge3),
\end{aligned}
\end{equation}
where $D=\diam C$ and $\varrho=d_C(x)+\sqrt{\osc g}$. In particular
\[
\lim_{t\to\infty}\ \lim_{T\to0}v_T(t,x)=g_C=\lim_{T\to0}\ \lim_{t\to\infty}v_T(t,x).
\]
\end{enumerate}
\end{theorem}

\begin{proof}
Adding a constant to $g$ adds the same constant to $v_T$ and to $g_C$ without changing $\osc g$. So we may assume $\inf g=0$. Then $0\le g\le\osc g$. We see the statement about $v_T$ from \eqref{eq:log}, because for a constant $\kappa$
\[
-2T\log\E\exp\Big(-\frac{g(X^x_t)+\kappa}{2T}\Big)=\kappa-2T\log\E\exp\Big(-\frac{g(X^x_t)}{2T}\Big).
\]

(a) By Lemma \ref{lem:invariant}, $P_t(x,A)\to\gamma_T^C(A)$ for every Borel set $A$ as $t\to\infty$. A bounded Borel function is a uniform limit of finite linear combinations of indicator functions of Borel sets. Therefore the convergence for every Borel set implies the convergence of the integrals of every bounded Borel function. We apply this to the function $e^{-g/2T}$. So $\E\exp(-g(X^x_t)/2T)\to\int e^{-g/2T}d\gamma_T^C$ as $t\to\infty$. By the definition of the Gibbs state,
\[
\int e^{-g/2T}\,d\gamma_T^C=\cZ_T(C)^{-1}\int_{\R^d}e^{-(g+d_C^2)/2T}\,dx .
\]
This number is positive. Taking $-2T\log$ we get
\[
\lim_{t\to\infty}v_T(t,x)=-2T\log\int_{\R^d}e^{-(g+d_C^2)/2T}\,dx+2T\log\cZ_T(C).
\]
We show in three steps that the right side tends to $g_C$ as $T\to0$.

First, $\cZ_T(C)\le\cZ_1(C)$ for $T\le1$, because $e^{-d_C^2/2T}\le e^{-d_C^2/2}$ for such $T$. If $c\in C$ then $d_C(x)\le|x-c|$ gives $\cZ_T(C)\ge\int e^{-|x-c|^2/2T}dx=(2\pi T)^{d/2}$. So for $0<T\le1$
\[
d\,T\log(2\pi T)\le2T\log\cZ_T(C)\le2T\log\cZ_1(C) .
\]
Both sides tend to zero as $T\to0$. Hence $2T\log\cZ_T(C)\to0$.

Next let $\lambda\in(0,1)$. Since $g+d_C^2\ge g_C\ge0$ and $g+d_C^2\ge d_C^2$, we have
\[
g+d_C^2=(1-\lambda)\big(g+d_C^2\big)+\lambda\big(g+d_C^2\big)\ge(1-\lambda)\,g_C+\lambda\,d_C^2 .
\]
Therefore
\[
\int e^{-(g+d_C^2)/2T}\,dx\le e^{-(1-\lambda)g_C/2T}\int e^{-\lambda d_C^2/2T}\,dx=e^{-(1-\lambda)g_C/2T}\,\cZ_{T/\lambda}(C).
\]
Taking $-2T\log$ of both sides we get
\[
-2T\log\int e^{-(g+d_C^2)/2T}\,dx\ge(1-\lambda)\,g_C-\lambda\cdot\frac{2T}{\lambda}\log\cZ_{T/\lambda}(C) .
\]
By the first step, used with $T/\lambda$ in place of $T$, the last term tends to zero as $T\to0$. Therefore $\liminf_{T\to0}(-2T\log\int e^{-(g+d_C^2)/2T}dx)\ge(1-\lambda)g_C$.

Finally let $y_*$ be a point where the minimum $g_C$ is attained and let $\eta>0$. The function $g+d_C^2$ is continuous and it is equal to $g_C$ at $y_*$. By continuity there is a ball $B$ around $y_*$ on which $g+d_C^2\le g_C+\eta$. Integrating only over this ball we get
\[
\int e^{-(g+d_C^2)/2T}\,dx\ge\Leb(B)\,e^{-(g_C+\eta)/2T},
\]
where $\Leb(B)$ is the volume of $B$. Taking $-2T\log$ we get
\[
-2T\log\int e^{-(g+d_C^2)/2T}\,dx\le g_C+\eta-2T\log\Leb(B) ,
\]
so that $\limsup_{T\to0}(-2T\log\int e^{-(g+d_C^2)/2T}dx)\le g_C+\eta$.
Letting $\lambda\to0$ and $\eta\to0$ proves (a).

(b) \emph{Existence of the limit.} The drift $b(x)=-(x-P_Cx)$ is Lipschitz by \ref{P2}. Hence the large deviation principle for small additive noise applies \cite[Section 5.6]{DemboZeitouni1998}. We use it with the noise parameter $2T$, after a translation of the starting point. As $T\to0$, the laws of $(X^x_u)_{u\le t}$ on the space of continuous paths from $[0,t]$ to $\R^d$ satisfy the large deviation principle with speed $1/2T$ and with a good rate function. The rate function is $\frac12\int_0^t|\phi'(u)-b(\phi(u))|^2du$ when $\phi$ is absolutely continuous with $\phi(0)=x$ and square integrable derivative. It is $+\infty$ for all other paths. Since $\phi'-b(\phi)=\phi'+\phi-P_C\phi$, the rate function of such a path is $J_t(\phi)$. The functional $\phi\mapsto g(\phi(t))$ is bounded and continuous. So the Laplace principle for large deviations \cite[Section 4.3]{DemboZeitouni1998} gives \eqref{eq:v0}. In detail, this principle gives
\[
\lim_{T\to0}2T\log\E\exp\Big(-\frac{g(X^x_t)}{2T}\Big)=\sup\big\{-g(\phi(t))-J_t(\phi):\ \phi(0)=x\big\},
\]
and changing the sign we get \eqref{eq:v0}.

\emph{Lower bound.} Take a curve $\phi$ with $\phi(0)=x$ and write $y=\phi(t)$, $r=d_C(y)$ and $d_t=e^{-t}d_C(x)$. By \eqref{eq:energybound}, $J_t(\phi)\ge(r-d_t)_+^2$. Indeed, \eqref{eq:energybound} gives
\[
r\le d_t+\big((1-e^{-2t})\,J_t(\phi)\big)^{1/2}\le d_t+J_t(\phi)^{1/2},
\]
so $J_t(\phi)^{1/2}\ge r-d_t$. Of course also $J_t(\phi)^{1/2}\ge0$. It is enough to consider curves with $g(y)+J_t(\phi)\le\osc g$, because the curve $u\mapsto\varphi_u(x)$ has zero energy by \eqref{eq:flowode} and gives the value $g(\varphi_t(x))\le\osc g$. For such curves $J_t(\phi)\le\osc g$, since $g\ge0$. So $(r-d_t)_+^2\le\osc g$, which gives $r\le d_t+\sqrt{\osc g}$. Now
\[
r^2-(r-d_t)_+^2\le2r d_t\le2d_t\big(d_t+\sqrt{\osc g}\big).
\]
For the first inequality there are two cases. If $r\ge d_t$ the left side is
\[
r^2-(r-d_t)^2=2r d_t-d_t^2\le2r d_t ,
\]
and if $r<d_t$ then $(r-d_t)_+=0$, so that the left side is $r^2\le r d_t\le2r d_t$. The second inequality holds because $r\le d_t+\sqrt{\osc g}$. Therefore
\[
g(y)+J_t(\phi)\ge g(y)+r^2-2d_t\big(d_t+\sqrt{\osc g}\big)\ge g_C-2d_t\big(d_t+\sqrt{\osc g}\big).
\]
The first inequality uses $J_t(\phi)\ge(r-d_t)_+^2$. The second uses $g(y)+r^2=g(y)+d_C(y)^2\ge g_C$. Taking the infimum over the curves $\phi$ and recalling $d_t=e^{-t}d_C(x)$, we get from \eqref{eq:v0} the first inequality in \eqref{eq:v0bounds}.

\emph{Upper bound.} Let $y_*$ be a minimiser in the definition of $g_C$. Put $p=P_Cy_*$, $r=d_C(y_*)$. Since $d_C(p)=0$ we have $g(y_*)+r^2=g_C\le g(p)\le\osc g$. As $g(y_*)\ge0$, it follows that $r\le\sqrt{\osc g}$. Let $t\ge3$ and $t'=t/3$. We define a curve $\phi$ from $x$ to $y_*$ in three pieces.

On $[0,t']$ the first piece is $\phi(u)=\varphi_u(x)$. By \eqref{eq:flowode} the control $a$ vanishes on this interval. This piece has zero energy and ends at $x_1=\varphi_{t'}(x)$, with $d_C(x_1)=e^{-t'}d_C(x)$. The last equality holds because $x_1=P_Cx+e^{-t'}(x-P_Cx)$ by \eqref{eq:flow} and $P_Cx_1=P_Cx$ by \ref{P5}.

The last piece, on $[2t',3t']$, is $\phi(u)=p+e^{-(t-u)}(y_*-p)$. As in the proof of Proposition \ref{prop:reach}(ii), $a(u)=2(\phi(u)-p)$ and the energy of this piece is $r^2(1-e^{-2t'})\le r^2$. Indeed, $|a(u)|^2=4r^2e^{-2(t-u)}$ and $t-2t'=t'$, so
\[
\frac12\int_{2t'}^{t}|a(u)|^2\,du=2r^2\int_{2t'}^{t}e^{-2(t-u)}\,du=r^2\big(1-e^{-2t'}\big) .
\]
It starts at $y_1=p+e^{-t'}(y_*-p)$ and ends at $y_*$. Here $d_C(y_1)=e^{-t'}r$.

It remains to define the middle piece, which joins $x_1$ to $y_1$. On $[t',2t']$ let $\phi$ be the segment from $x_1$ to $y_1$ run at constant speed, that is,
\[
\phi(u)=x_1+\frac{u-t'}{t'}\,(y_1-x_1)\qquad(t'\le u\le2t').
\]
Then $|\phi'|=|y_1-x_1|/t'$. Since $d_C$ is convex, we also have
\[
|\phi(u)-P_C\phi(u)|=d_C(\phi(u))\le\max\{d_C(x_1),d_C(y_1)\}\le e^{-t'}\varrho .
\]
In the last step we used $d_C(x_1)=e^{-t'}d_C(x)$, $d_C(y_1)=e^{-t'}r$ and $\max\{d_C(x),r\}\le d_C(x)+\sqrt{\osc g}=\varrho$. Also
\[
|y_1-x_1|\le d_C(x_1)+|P_Cx_1-P_Cy_1|+d_C(y_1)\le D+e^{-t'}(d_C(x)+r)\le D+e^{-t'}\varrho ,
\]
because the two projections are points of $C$ and $D$ is the diameter of $C$. So on this interval
\[
|a(u)|\le|\phi'(u)|+|\phi(u)-P_C\phi(u)|\le\frac{|y_1-x_1|}{t'}+e^{-t'}\varrho .
\]
The interval has length $t'$. By the inequality $(\alpha+\beta)^2\le2\alpha^2+2\beta^2$, the energy of this piece is at most
\[
\frac{t'}2\Big(\frac{|y_1-x_1|}{t'}+e^{-t'}\varrho\Big)^2\le\frac{|y_1-x_1|^2}{t'}+t'\,e^{-2t'}\varrho^2\le\frac{(D+e^{-t'}\varrho)^2}{t'}+t'\,e^{-2t'}\varrho^2 .
\]
The three pieces fit together at the points $x_1$ and $y_1$, so $\phi$ is an absolutely continuous curve with square integrable derivative, $\phi(0)=x$ and $\phi(t)=y_*$. Its energy is the sum of the energies of the three pieces. The first piece has energy zero and the last piece has energy at most $r^2$. We estimated the energy of the middle piece in the last display. We also have $v_0(t,x)\le g(y_*)+J_t(\phi)$ by \eqref{eq:v0}. So
\[
v_0(t,x)\le g(y_*)+r^2+\frac{(D+e^{-t'}\varrho)^2}{t'}+t'\,e^{-2t'}\varrho^2 .
\]
Since $g(y_*)+r^2=g_C$, $1/t'=3/t$, $e^{-t'}=e^{-t/3}$ and $t' e^{-2t'}=\frac t3e^{-2t/3}$, this is the second inequality in \eqref{eq:v0bounds}.

Both bounds in \eqref{eq:v0bounds} tend to $g_C$ as $t\to\infty$, since $e^{-t}\to0$, $1/t\to0$ and $te^{-2t/3}\to0$. Hence $v_0(t,x)\to g_C$ as $t\to\infty$. Together with (a) this gives the last statement of (b).
\end{proof}

\begin{remark}
We compare Corollary \ref{cor:commute} and Theorem \ref{thm:hjb}(b). For the law of the state the two limits do not commute even for a convex set, unless the initial law is special. For the value function they commute for every compact convex set and every bounded Lipschitz terminal cost. The proof of Theorem \ref{thm:hjb}(b) shows the reason. Inside a convex set the drift vanishes. Over a long horizon a control can cross the set with an energy which tends to zero, so in the limit the value does not depend on the initial state. Without control the state needs a time of order $1/T$ to cross the set. Now let $C$ be a finite union of pairwise disjoint compact convex sets $C_i$ and let $\delta_i$ be the distance from $C_i$ to the nearest other component. Then, for a suitable terminal cost and a starting point in $C_i$, the two limits of the value function differ by $\delta_i^2/4$ \cite[Theorem~4.1]{DarapaneniMeta}.
\end{remark}

\section{Open questions}\label{sec:open}

We end with some open questions.

\begin{enumerate}[label=\arabic*.,leftmargin=2em]
\item \emph{The optimal rate and the logarithm.} Is the rate $T^{1/4}$ in Theorem \ref{thm:diagram}(b) optimal for general convex sets (Remark \ref{rem:rate})? Is the factor $\log(1/T)$ in Corollary \ref{cor:mixtime} necessary? Is there a cutoff, that is, does the total variation distance to equilibrium fall from nearly one to nearly zero in a window of times which is short compared with the mixing time?
\item \emph{Slowly moving sets.} Theorem \ref{thm:envelope} treats a moving set, but Sections \ref{sec:diagram} and \ref{sec:hjb} treat only a fixed set. Suppose that the set moves with a speed $v$. If $v$ is small compared with $T$, one expects the law of the state to follow the uniform law on $C(t)$. If $v$ is large compared with $T$, the expected limit is the equation without noise, which is a regularisation of the sweeping process \cite{Moreau1977}. When $v$ is of the order of $T$, the limit should be described by reflected Brownian motion in a moving domain.
\item \emph{Infinite dimension.} Theorem \ref{thm:envelope} and Corollary \ref{cor:viability} hold in Hilbert space. The Gibbs state does not make sense there without a reference Gaussian measure. The Wills functional has a counterpart for Gaussian processes, that is, for subsets of Hilbert space \cite{Vitale1996}. Is there a form of Theorem \ref{thm:gibbs} in this setting?
\item \emph{Other norms and other sets.} Suppose that the distance to $C$ is measured in a norm which is not Euclidean. Then a point may have more than one nearest point in a convex set. The potential is then no longer of class $C^{1,1}$. We do not know which geometric property replaces convexity in Theorem \ref{thm:converse}. In the Euclidean case the local Steiner formula holds for small radii for sets of positive reach \cite{Federer1959}. This suggests a version of Theorem \ref{thm:gibbs} for such sets and a version of \cite[Theorem~3.6]{DarapaneniMeta} for unions of them.
\end{enumerate}

\end{document}